\documentclass[11pt,english,12p]{amsart}

\usepackage{graphicx}%
\usepackage{multirow}%
\usepackage{amsmath,amssymb,amsfonts}%
\usepackage{amsthm}%
\usepackage{mathrsfs}%
\usepackage[title]{appendix}%
\usepackage{xcolor}%
\usepackage{textcomp}%
\usepackage{manyfoot}%
\usepackage{booktabs}%
\usepackage{algorithm}%
\usepackage{algorithmicx}%
\usepackage{algpseudocode}%
\usepackage{listings}%

\usepackage[autostyle]{csquotes}
\usepackage{float}
\usepackage{tikz-cd} 
\usetikzlibrary{patterns}
\usetikzlibrary{matrix}
\usepackage{tikz}
\usepackage{amssymb,latexsym}
\usepackage{verbatim}
\usepackage[autostyle]{csquotes}
\usepackage{float}
\usepackage{pst-node}
\usepackage{mathtools}
\usepackage{faktor}
\usepackage{hyperref}
\usetikzlibrary{arrows.meta}
\usetikzlibrary{arrows,positioning,shapes,fit,calc}
\usetikzlibrary{decorations.pathreplacing}
\usepackage[all,cmtip]{xy}

\tikzstyle{littledot}=[circle, fill, inner sep=.4pt,minimum size=.4pt]

\newtheorem{thm}{Theorem}[section]
\newtheorem{lem}[thm]{Lemma}
\newtheorem{prop}[thm]{Proposition}
\newtheorem{cor}[thm]{Corollary}

\newtheorem{defn}[thm]{Definition}
\newtheorem{exmp}[thm]{Example}

\newtheorem{remark}[thm]{Remark}

\newcommand{\R}{{\mathbb R}}

\newcommand{\C}{{\mathbb C}}

\newcommand{\ek}{\stackrel{k}{e}}
\newcommand{\ekp}{\overset{k+1}{e}}
\newcommand{\gh}{\hat{\Gamma}}

\newcommand*{\upSmallFrown}{\mathbin{\raisebox{0.9ex}{$\smallfrown$}}}

\newcommand{\FP}{\textrm{FP}}
\newcommand{\Tr}{\textrm{TC}}
\newcommand{\resp}{\textrm{r}}

\selectfont

\title[Group theoretic properties of Clifford multiplication]{Group theoretic properties of Clifford multiplication on 2-torsion points on the Dirac Spinor Abelian Variety}

\author[J. Brown, I. Grzegorczyk, and R. Su\'arez]{Jennifer Brown, Ivona Grzegorczyk, and Ricardo Su\'arez}
\address{Department of Mathematics, California State University Channel Islands, Camarillo, CA, USA}
\email{jennifer.brown@csuci.edu, ivona.grzegorczyk@csuci.edu, ricardo.suarez532@csuci.edu}



\begin{document}

\normalsize

\begin{abstract}

In this manuscript we consider a special complex torus, denoted $S_{\Delta_{2k}}$ (for each $k \in \mathbb{N},\, k \geq 1$) and called the Dirac spinor torus. It is an Abelian variety of complex dimension $2^{k}$  whose covering space is the space of Dirac spinors, $\Delta_{2k}$, for the Clifford algebra $Cl(\C^{2k})$ associated with the vector space $\C^{2k}$. Fixing an isomorphism $\rho:Cl(\C^{2k})\rightarrow End (\Delta_{2k})$, we define Clifford multiplication on $S_{\Delta_{2k}}$ as the actions of those endomorphisms in the image of $\rho$ that preserve the full rank lattice. We analyze the properties of that Clifford multiplication on the 2-torsion points of the Dirac spinor torus.
We identify the Clifford actions with permutation maps that represent all isomorphism classes of these actions on the group of 2-torsion points. We provide a structure theorem describing these isomorphism classes of Clifford actions in a way that is independent of the choice of representatives. We conclude by extending the scope of our analysis to the group of $n$-torsion points and analyzing the fixed points and translation constants of entry-permuting maps, a broader class of actions of which the Clifford actions on the 2-torsion points of $S_{\Delta_{2k}}$ is a subset.

\end{abstract}
\maketitle

\section{Introduction} 

This manuscript focuses on describing Clifford multiplication on the 2-torsion points of a special complex Abelian variety called the Dirac spinor Abelian variety (see \cite{RS}). 
Note that for any complex Abelian variety, the 2-torsion points form a finite group of order $4^{g}$ isomorphic to $(\mathbb{Z}/2\mathbb{Z})^{2g}$ under the operation of translation, where $g$ is the dimension of the Abelian variety. The 2-torsion points on a complex Abelian variety also carry geometric significance, in the sense that they parameterize the set of symmetric line bundles with the first Chern class  $c_{1}(L)$, where $L$ is the symmetric line bundle that defines the polarization of our complex Abelian variety (see \cite{BL}).

Clifford multiplication on Dirac spinors is a well-studied  concept in physics and spin geometry, as the Dirac spinor module is often  used to construct spinor bundles on spin manifolds.  We generally consider elements of a vector space acting via multiplication on the space of spinors by their suitable matrix representatives (see \cite{FR}). Here, we study Clifford multiplication on the Dirac spinor Abelian variety, which we denote by  $S_{\Delta_{2k}}$
(for $k \in \mathbb{N},\, k \geq 1$). This variety is a complex torus formed by taking the quotient of the space of Dirac spinors $\Delta_{2k}=\C^{2^{k}}$ associated to the complex Clifford algebra $Cl(\C^{2k})$ by the standard square lattice, $\Delta_{2k}^{\mathbb{Z}}=\mathbb{Z}^{2^k}\oplus i\mathbb{Z}^{2^k}$. In our case, we define Clifford multiplication as the actions of those endomorphisms in the image of $\rho:Cl(\C^{2k})\xrightarrow{\cong} End(\Delta_{2k})$ that preserve the standard square lattice on $\Delta_{2k}$ (where $\rho$ is a fixed isomorphism mapping each Clifford algebra element in $Cl(\C^{2k})$ to its matrix representation in $End(\Delta_{2k})$); hence we can view these actions as endomorphisms that descend to the complex torus $S_{\Delta_{2k}}$. Specifically, we focus here on Clifford multiplication actions on the 2-torsion points $J_2^{S_{\Delta_{2k}}}$ of our spinor variety. 
The general outline of the paper is as follows.

In Section \ref{section: background} we provide definitions, notation and terminology, and some basic facts related to spinor Abelian varieties.

In Section \ref{Properties of the multiplicative group ...}, we define an equivalence relation $\sim$ on the multiplicative group $\gh_{2k}$ of generators of the Clifford algebra $Cl(\mathbb{C}^{2k})$ and their negatives by declaring two actions to be equivalent if they act identically on the 2-torsion points $J_2^{S_{\Delta_{2k}}}$. In Theorem \ref{number of unique Clifford actions} we prove that the $2^{2k}$ generators in $\gh_{2k}$ yield only $2^{k+1}$ distinct actions on $J_2^{S_{\Delta_{2k}}}$, i.e.\ that the equivalence relation $\sim$ has $2^{k+1}$ equivalence classes. By analyzing the way in which the matrix representations of generators in $\gh_{2k}$ give rise to those of generators in $\gh_{2k+2}$, we are able to describe, in Theorem \ref{that corollary}, the structure of the equivalence classes of Clifford actions under this relation of indistinguishability modulo 2-torsion points.

In Section \ref{section: Clifford permutations}, we show how the actions of the generators of the Clifford algebra on 2-torsion points $\vec{v}$ can be represented as certain permutations of order 2 acting on the rows of the vector form of $\vec{v}$ (combined with scalar multiplication by $i$, depending on the form of the generator). These permutations, which we call induced Clifford permutations, provide a way of viewing the actions of $\gh_{2k}$ on $J_2^{S_{\Delta_{}2k}}$ in a way that is independent of the choice of representatives of the equivalence classes defined in Section \ref{Properties of the multiplicative group ...}. 

In Section \ref{section: the group of actions on 2-torsion points}, we show that the strictly real induced Clifford permutations defined in Section \ref{section: Clifford permutations}, i.e.\ those associated with generators $e_\mu \in \gh_{2k}$ whose matrix representations have all real entries, form an Abelian subgroup of the alternating group $Alt(2^k)$, which we call $\operatorname{Cliff}(Alt(2^k))$. This way we obtain an isomorphism between the group $\gh_{2k} / J_2^{S_{\Delta_{2k}}}$ of generators acting on 2-torsion points and the product $\mathbb{Z}_2 \times \operatorname{Cliff}(Alt(2^k))$.

In Section \ref{sec: fixed points and translation constants}, we prove in Proposition \ref{FP is Tr} that 
for all generators $e_\mu$ considered as actions on the 2-torsion points $J_2^{S_{\Delta_{}2k}}$, the set of fixed points of $e_\mu$ is equal to its set of translation constants. For this, we need to obtain permutations representing all types of generators in $\gh_{2k}$: not only those $e_\mu$ whose matrix representations $\rho(e_\mu)$ have all real entries, but also those $e_\mu$ for which the nonzero entries of $\rho(e_\mu)$ are purely imaginary. We accomplish this by representing 2-torsion points of $S_{\Delta_{2k}}$ as $2^k \times 2$ matrices with entries in $\{0, 1\}$. Whereas the strictly real induced Clifford permutations defined in Section \ref{section: Clifford permutations} are permutations of the $2^k$ rows of the vector form of a 2-torsion point $\vec{v}$, the strictly imaginary induced Clifford permutations are permutations of the $2^{k+1}$ entries in the matrix form of $\vec{v}$.

 In Section \ref{entry permuting maps} we conclude by extending our analysis to the actions of entry-permuting maps on the 
group of $n$-torsion points of $S_{\Delta_{2k}}$. We define entry-permuting maps as permutations that act on the matrix forms of $n$-torsion points by permuting their entries; and when 2-torsion points in $S_{\Delta_{2k}}$ are represented in matrix form as in Section \ref{sec: fixed points and translation constants}, the Clifford actions are entry-permuting maps. In Corollary \ref{Tr times FP is ...}, we find a general relationship between the number of fixed points and the number of translation constants of entry-permuting maps acting on $n$-torsion points, partially extending Proposition \ref{FP is Tr}.

\section{Background material}\label{section: background}

Our focus is on the space of Dirac spinors $\Delta_{2k}=\C^{2^{k}}$ for $Cl(\mathbb{C}^{2k})$, the Clifford algebra of the complex vector space $\mathbb{C}^{2k}$ (where $k \in \mathbb{N},\, k \geq 1$). Note that $\Delta_{2k}$ is a $Cl(\mathbb{C}^{2k})$ module. Hence the module multiplication on $\Delta_{2k}$, which we call \textbf{Clifford multiplication}, is given by $g\cdot x:=\rho(g)(x)$, where  $\rho(g)$ is the $2^k \times 2^k$ complex matrix defined  by the isomorphism $\rho:Cl(\C^{2k})\xrightarrow{\cong} End(\Delta_{2k})$ for any $g\in Cl(\C^{2k})$ and $x\in\Delta_{2k}$. (See \cite{FR}.)

Note that when we take the quotient of $\Delta_{2k}$ by the square lattice 
\[\Delta_{2k}^{\mathbb{Z}} = \mathbb{Z}^{2^{k}}\oplus i\cdot \mathbb{Z}^{2^{k}}=\{m+i\cdot n: \,m,n\in\mathbb{Z}^{2^{k}}\},\] we obtain a $2^{k}$-dimensional complex torus; this is our \textbf{Dirac spinor torus} $S_{\Delta_{2k}}$. 

\begin{remark}
The Dirac spinor torus $S_{\Delta_{2k}}$ has as its period matrix $i \cdot I_{2^k}$ (here $I_{2^k}$ denotes the identity matrix in dimension $2^k$). This period matrix lies in the Siegel upper half-space; as a consequence, our torus $S_{\Delta_{2k}}$ is a principally polarized Abelian variety, which we refer to as the Dirac spinor Abelian variety. (See \cite{BL}, \cite{GH}.)
\end{remark}

\begin{defn}
We define \textbf{lattice Clifford actions} on the Dirac spinor Abelian variety $S_{\Delta_{2k}}$ as the restriction  $\rho:C l(\C^{2k})_{\mathbb{Z}}\rightarrow End(S_{\Delta_{2k}})$ of the  isomorphism $\rho$ to the integral subring $C l(\C^{2k})_{\mathbb{Z}}$.
\end{defn}

This makes $S_{\Delta_{2k}}$ a  $Cl(\C^{2k})_{\mathbb{Z}}$-module, since only integral combinations of generators $e_{\mu}$ of our Clifford algebra preserve the square lattice.

\begin{defn}
	The Clifford group $\Gamma_{2k}$ is defined as the subgroup of invertible elements $g$ in the Clifford algebra $C l(\C^{2k})$ that satisfy the property $g\cdot v\cdot g^{-1}\in \C^{2k}$ for all $v\in \C^{2k}$. We denote by $\hat{\Gamma}_{2k}$ the subset of $\Gamma_{2k}$ consisting of the multiplicative generators of the Clifford algebra and their negatives, and we refer to $\gh_{2k}$ as the \textbf{multiplicative group of generators} of $Cl(\mathbb{C}^{2k})$.
\end{defn}

(This terminology for $\gh_{2k}$ is justified by Lemma \ref{size of Gamma hat 2k} below.) If $\mu = \{i_1, \ldots, i_l\}$ is a sequence of numbers from $\{1, \ldots, 2k\}$, then to simplify the notation we write
\[
e_\mu = e_{i_1, \ldots, i_l}:= e_{i_1} \cdots e_{i_l},
\]
where $e_1, \ldots, e_{2k}$ are the vector generators of the vector space $\mathbb{C}^{2k}$ viewed as a vector subspace of the Clifford algebra $Cl(\mathbb{C}^{2k})$. Hence, $\gh_{2k}$ consists of elements $e_\mu$ where $\mu$ is such a sequence. Note now that  elements of $\gh_{2k}$ are of two kinds: the generators of the Clifford algebra, which are those $e_\mu$ for which $\mu$ is (or can be written as) an increasing sequence in $\{1, \ldots, 2k\}$; and their negatives, which are those $e_\mu$ for which $\mu$ cannot be written as an increasing sequence in $\{1, \ldots, 2k\}$. This follows from  basic facts on Clifford algebras; see \cite{EM} and \cite{LM}.

\begin{lem}\label{size of Gamma hat 2k}
	Let $\hat{\Gamma}_{2k}$ be defined as above. Then $\hat{\Gamma}_{2k}$ is a group of order $2^{2k+1}$.
	
\end{lem}

\begin{proof}
	The Clifford algebra $Cl(\C^{2k})$ has a total of $2^{2k}$ generators $e_{\mu}$. 
	The generators form a multiplicative group in the sense that the product of any two generators is either a generator of the Clifford algebra or its negative. 
	Now all negatives of the generators, except for $-1$, can be recovered from the anticommutativity property of the Clifford algebra; and $-1$ is in the group because of the negative definite generators: $e_{j}^{2}=-1$ for any one vector $e_j \in \{e_1, \ldots ,e_{2k}\}$. 
	Hence $|\hat{\Gamma}_{2k}|=2\cdot 2^{2k}=2^{2k+1}$. 

\end{proof}

\begin{defn}
For $k \in \mathbb{N}, \,k \geq 1$, we define the $2$-\textbf{torsion points} of the Abelian variety $S_{\Delta_{2k}}$ as $J_2^{S_{\Delta_{2k}}}=\{x\in S_{\Delta_{2k}}:2\cdot x=0\}\subset S_{\Delta_{2k}}$.
\end{defn}

\begin{lem}\label{number of points in J_2^{2^k}}
	The set of $2$-torsion points of our spinor Abelian variety $S_{\Delta_{2k}}$ is of order  $2^{(2^{k+1})}$.
	
\end{lem}

\begin{proof}
	Any element of $J_2^{S_{\Delta_{2k}}}$ is represented as a $2^k$-vector of points in the 4-element set of the 2-torsion points on the square elliptic curve $\dfrac{\C}{\mathbb{Z}\oplus i \mathbb{Z}}$, which we denote here as 
\[ 
J_2^{S_{\Delta_0}} = \left\{0, \dfrac{1}{2}, \dfrac{i}{2}, \dfrac{1+i}{2}\right\} \subset S_{\Delta_0}.
\]
Then to specify an element $\vec{v} \in J_2^{S_{\Delta_{2k}}}$ in the form of a $2^k$-vector, we must choose from among the 4 elements of $J_2^{S_{\Delta_0}}$ for each of the $2^k$ components of $\vec{v}$. Thus we have $$|J_2^{S_{\Delta_{2k}}}| = 4^{(2^k)} = (2^2)^{2^k} = 2^{2 \cdot 2^k} = 2^{(2^{k+1})}.$$ 
\end{proof}

\section{Properties of the multiplicative group of generators acting on the 2-torsion points of the Dirac spinor Abelian variety}\label{Properties of the multiplicative group ...}

 In this section we show that not all of the $2^{2k+1}$ actions in $\hat{\Gamma}_{2k}$ are distinct when considered as actions on the 2-torsion points $J_2^{S_{\Delta_{2k}}}$. In particular, in Theorem \ref{number of unique Clifford actions} we prove that the $2^{2k}$ generators of the Clifford algebra $Cl(\mathbb{C}^{2k})$ yield only $2^{k+1}$ unique actions on the 2-torsion points $J_2^{S_{\Delta_{2k}}}$.
As is well known, for each endomorphism associated to a generator of the Clifford algebra, we have an associated unitary matrix representation given by  a complex unitary  $2^{k}\times 2^{k}$ matrix.  The matrix representations $\rho(e_\mu)$ of generators $e_\mu \in \gh_{2k}$ come from the isomorphism $Cl(\C^{2k})\cong \underbrace{\C(2)\otimes_{\R} \cdots \otimes_{\R} \C(2)}_{k \textrm{ times}}\cong \C(2^{k})$ (see [Fr]).  The isomorphism stems from an inductive process generated by the isomorphism $Cl(\C^{2})\cong \C(2)$, given by the associations 
\begin{align*}
& e_1\cong E_1  = \left[\begin{array}{cc} i & 0 \\ 0 & -i\end{array}\right],\quad
e_2\cong E_2  = \left[\begin{array}{cc} 0 & i \\ i& 0\end{array}\right],  \textrm{ and } \\
& e_{12}\cong E_1E_2=E_{12}  = \left[\begin{array}{cc} 0& -1 \\ 1& 0\end{array}\right]. 
\end{align*}
With the representative matrices  $E_1, E_2$, and \[iE_{12}=:B = \left[\begin{array}{cc} 0& -i \\ i& 0\end{array}\right],\] along with the $2\times 2$ identity matrix $I_2$, it is possible to construct matrix representations for all generators of the complex Clifford algebras  $Cl(\C^{2k})$. (Note that $\mu = \emptyset$ counts as an increasing sequence, and that $\rho(e_\emptyset)$ is always the identity matrix  in the appropriate dimension.)

\begin{prop}[See \cite{FR}]\label{Dirac representations}
Let $I_2,E_1,E_2,$ and $B$ be as above. For all $k \in \mathbb{N},\, k \geq 1$, we have an isomorphism from $Cl(\C^{2k})$ to $\C(2^{k})$ given explicitly by the following $k$-Kronecker product identification:

\[
\begin{array}{llll}
	 e_{2j-1} & \xrightarrow{\cong} & I_2^{\otimes k-j  }\otimes E_1\otimes B^{\otimes j-1} & \textrm{ for } j=1, \ldots ,k, \textrm{ and } \\ 
	 e_{2j} & \xrightarrow{\cong} & I_2^{\otimes k-j  }\otimes E_2\otimes B^{\otimes j-1} & \textrm{ for } j=1, \ldots ,k.
	 \end{array}
\] \qed

\end{prop}

We use the following two lemmas on Kronecker products in order to prove Theorem \ref{number of unique Clifford actions}.

\begin{lem}\label{tensor lemma 1}
	
Let $F_1, G_1 \in \mathbb{C}(r_1), \, F_2, G_2 \in \mathbb{C}(r_2), \ldots,$ and  $F_m, G_m \in \mathbb{C}(r_m)$ for some $m \in \mathbb{N}$ and $r_1, \ldots, r_m \in \mathbb{N}$ with each $r_i > 0$. Then  \[(F_1\otimes \cdots \otimes F_{m})\cdot (G_1\otimes \cdots \otimes G_{m})=(F_1\cdot G_1)\otimes \cdots \otimes (F_{m}\cdot G_{m}).\]

\end{lem}

\begin{proof}
Begin with a product $(F_1\otimes \cdots \otimes F_{m})\cdot (G_1\otimes \cdots \otimes G_{m})$. Note that this product is a $(r_1 \cdots r_m) \times (r_1 \cdots r_m)$ matrix. By the associative property of Kronecker products, as well as the component product property of the Kronecker product of two matrices, we have the following:
\[
\begin{array}{lll}
(F_1\otimes \cdots \otimes F_{m}) \cdot (G_1\otimes \cdots \otimes G_{m}) & = & (F_1\otimes(F_2\otimes \cdots \otimes F_{m})) \cdot (G_1\otimes( G_2\otimes \cdots \otimes G_{m}))\\
& = &(F_1\cdot G_1)\otimes ((F_2\otimes \cdots \otimes F_m)\cdot (G_2\otimes \cdots \otimes G_m)).
\end{array}
\]
Repeating this process for  $(F_2\otimes \cdots \otimes F_m)\cdot (G_2\otimes \cdots \otimes G_m)$, we have \[ (F_2\otimes \cdots \otimes F_m)\cdot (G_2\otimes \cdots \otimes G_m)=(F_2\cdot G_2)\otimes ((F_3\otimes \cdots \otimes F_m)\cdot (G_3\otimes \cdots \otimes G_m)).\] Continuing in this way to the last terms, we obtain the desired result: \[ (F_1\otimes \cdots \otimes F_{m})\cdot (G_1\otimes \cdots \otimes G_{m})=(F_1\cdot G_1)\otimes  \cdots \otimes (F_{m}\cdot G_{m}).\]

\end{proof}

\begin{cor}\label{tensor corollary 1}
Consider  $l, m \in \mathbb{N}$ and $r_1, \ldots, r_m \in \mathbb{N}$ with $r_i \geq 1$ for each $i$; and consider complex matrices 
$F^{1}_1, \ldots ,F^{1}_{m},\, F^{2}_1, \ldots ,F^{2}_{m},\, \ldots, \,$ and  $F^{l}_1, \ldots ,F^{l}_{m}$ with $F^1_1, F^2_1, \ldots, F^l_1 \in \mathbb{C}(r_1)$,  $F^1_2, F^2_2, \ldots, F^l_2 \in \mathbb{C}(r_2), \ldots,$ and $ F^1_m, F^2_m,\  \ldots \ , F^l_m \in \mathbb{C}(r_m)$. Then the $l$ matrix product of the Kronecker products of the matrices $F^i_j$ is of the form 
\[
(F_1^{1}\otimes  \cdots \otimes F^{1}_m)\cdot(F^{2}_1\otimes  \cdots \otimes F^{2}_m) \cdots (F^{l}_1\otimes \cdots \otimes F^{l}_m)=(F^{1}_1 \cdot F^2_1 \cdots  F^{l}_1)\otimes \cdots \otimes (F^{1}_m \cdot F^2_m \cdots  F^{l}_m).
\]
\end{cor}
\begin{proof}
Starting with $(F_1^{1}\otimes \cdots \otimes F^{1}_m)\cdot(F^{2}_1\otimes \cdots \otimes F^{2}_m) \cdots\ (F^{l}_1\otimes \cdots \otimes F^{l}_m)$, we use Lemma \ref{tensor lemma 1} on the first two arguments to get $(F_1^{1}\otimes \cdots \otimes F^{1}_m)\cdot(F^{2}_1\otimes \cdots \otimes F^{2}_m) \cdots (F^{l}_1\otimes \cdots \otimes F^{l}_m)=(F^{1}_1\cdot F^{2}_1\otimes \cdots \otimes F^{1}_m\cdot F^{2}_m)\cdot (F^3_{1}\otimes \cdots \otimes F^{3}_m)\cdot(F^{4}_1\otimes \cdots \otimes F^{4}_m) \cdots\ (F^{l}_1\otimes \cdots \otimes F^{l}_m)$. Continuing this pairwise, via Lemma \ref{tensor lemma 1} we obtain the desired result: \[(F_1^{1}\otimes \cdots \otimes F^{1}_m)\cdot(F^{2}_1\otimes \cdots \otimes F^{2}_m) \cdots\ (F^{l}_1\otimes \cdots \otimes F^{l}_m)=(F^{1}_1 \cdots\  F^{l}_1)\otimes \cdots \otimes (F^{1}_m \cdots\ F^{l}_m).\]
\end{proof}

\begin{remark}\label{dim 1 2-torsion points}
The reason that Clifford multiplication by the matrices $\rho(e_\mu)$, for $e_\mu \in \gh_{2k}$, does not always give us different automorphisms on $J_2^{S_{\Delta_{2k}}}$, is as follows: as we saw in the proof of Lemma \ref{number of points in J_2^{2^k}}, each of the $2^k$ entries of a $2$-torsion point on the Dirac spinor Abelian variety $S_{\Delta_{2k}}$ is one of the four elements $v_0 :=0, \, v_1 :=\frac{1}{2}, \, v_2 :=\frac{i}{2}, \, v_3 :=\frac{1+i}{2}\in J_2^{S_{\Delta_0}}$. These four $2$-torsion points obey the following relations:

\begin{enumerate}
\item $2\cdot v_k=v_0$ for $k=0,1,2,3$
\item $v_0+v_k=v_k$ for $k=1,2,3$
\item $v_1+v_2=v_3$
\item $v_1+v_3=v_2$
\item $v_2+v_3=v_1$
\item $-v_k=v_k,-i\cdot v_k=i\cdot v_k$ for $k=0,1,2,3$
\item $i\cdot i\cdot v_k=v_k$.
\end{enumerate}
\end{remark}

Looking at the components of our 2-torsion points, it is immediately clear that not all Clifford actions are unique when we apply them to $J_2^{S_{\Delta_{2k}}}$, since multiplication by $-1$ on each component is the same as multiplication by $1$, and multiplication by $-i$ on each component is the same as multiplication by $i$.  

Now as a consequence of the relations above we have the following lemma. 

\begin{lem}\label{Clifford multiplication descends to ...}
On the $2$-torsion points $J_2^{S_{\Delta_{2k}}}\subset S_{\Delta_{2k}}$, Clifford multiplication descends to $Cl(\C^{2k})_{\mathbb{F}_2}$ multiplication, where   
\[
Cl(\C^{2k})_{\mathbb{F}_2}= \Bigg\{\sum_{\mu\subset \{1, \ldots, n\}} a_{\mu}e_{\mu}:a_{\mu}\in\{0,1\}\Bigg\};
\]
that is, where the integral scalars on a linear combination
of generators take the values of either 0 or 1.
\end{lem}

\begin{proof}
Multiplication by $i$ on $J_2^{S_{\Delta_0}}\subset \dfrac{\C}{\mathbb{Z}\oplus i \mathbb{Z}}$ is clearly an involution that fixes $v_0$ and $v_3$. Thus, viewing $J_2^{S_{\Delta_{2k}}}$ as 
$J_2^{S_{\Delta_{2k}}}=\{\left( \begin{array}{c} v_{a_1}\\ \vdots \\v_{a_{2^k}}\\ \end{array} \right) : a_l \in \{0, 1, 2, 3\} \textrm{ for } 1 \leq l \leq 2^k\}\}$, we see that multiplication by $i$ gives us $i\cdot i\cdot v_k=v_k$ on each of the $2^{k}$ components for our 2-torsion points on $S_{\Delta_{2k}}$. Moreover, integral multiplication on $J_2^{S_{\Delta_{2k}}}$ 
reduces to $\mathbb{F}_2$ multiplication, since $2m\cdot v_k=0$ and $(2m+1)\cdot v_k=v_k$ for  $m\in\mathbb{N}$ and $k=0,1,2,3$, from the symmetry relations on the 2-torsion points on $ \dfrac{\C}{\mathbb{Z}\oplus i \mathbb{Z}}$. Therefore Clifford multiplication on our set of 2-torsion points descends to $\mathbb{F}_2$ linear combinations of elements in the standard  basis. Thus multiplication by $Cl(\C^{2k})_{\mathbb{Z}}$ on $J_2^{S_{\Delta_{2k}}}$ is equivalent to  $Cl(\C^{2k})_{\mathbb{F}_2}$ multiplication.

\end{proof}

We are now ready to compute the number of unique Clifford actions given by the generators of the Clifford algebra on our 2-torsion points $J_2^{S_{\Delta_{2k}}}$.

\begin{thm}\label{number of unique Clifford actions}
The $2^{2k}$ generators of the Clifford algebra $Cl(\C^{2k})$  give us a total of $2^{k+1}$ unique involutions on $J_2^{S_{\Delta_{2k}}}$.

\end{thm}
\begin{proof}
From Lemma \ref{Clifford multiplication descends to ...}, it follows that Clifford multiplication by a generic element in $Cl(\C^{2k})_{\mathbb{Z}}$ descends to $\mathbb{F}_2$ linear combination of  generators $e_\mu$ and $i e_\mu$ (where $\mu$ is an increasing subsequence of $\{1, \ldots, 2k \}$). From Proposition \ref{Dirac representations} we see that the vector generators $e_1, \ldots, e_{2k}$ of our Clifford algebra are given by matrix representations which can be constructed by taking $k$-Kronecker products of combinations of $E_1,E_2,B,I_2\in \C(2)$. By taking products of the $k$-Kronecker product combinations of these matrices, we obtain the $2^{2k}$ generators $e_\mu$ of $Cl(\mathbb{C}^{2k})$.

Using the notation $T_{a_1, \ldots ,a_{2^{k}}}=\left( \begin{array}{c} v_{a_1}\\ \vdots \\v_{a_{2^{k}}}\ \end{array} \right)$ to denote an element of $J_2^{S_{\Delta_{2k}}}$, where $v_{a_{l}}\in J_2^{S_{\Delta_0}}$ for $1 \leq l \leq 2^k$, we get the following relations  on $J_2^{S_{\Delta_2}}$ by our generating matrices $E_1,E_2,B,I_2$. 

\[
\begin{array}{lcccr}
I_2\cdot T_{ab} & = & i \cdot E_1\cdot T_{ab} & = & T_{ab} \\
i \cdot E_2\cdot T_{ab} & = & i \cdot B\cdot T_{ab} & = & T_{ba}
\end{array}
\]

Defining an equivalence relation $\sim$ on $\hat{\Gamma}_2$ by $A\sim B$ if and only if $A\cdot \vec{v}=B\cdot \vec{v}$ for all $\vec{v}\in J_2^{S_{\Delta_2}}$, we get the following equivalences from our generating matrices: 

\[
\begin{array}{rcl}
I_2 & \sim & i \cdot E_1 \\
i \cdot E_2 & \sim & i \cdot B \\ 
i \cdot I_2 & \sim & E_1 \\
E_2 & \sim & B 
\end{array}
\]

Thus the representation matrices that act uniquely on $J_2^{S_{\Delta_2}}$ are generated by $I_2$ and $E_2$, as well as by $i \cdot I_2$ and $i \cdot E_2$. From these equivalences we conclude that for the even vector generators, given by $e_{2j}\xrightarrow{\cong} I_2^{\otimes k-j  }\otimes E_2\otimes B^{\otimes j-1} $ for $j=1, \ldots, k$, the representative matrices as they pertain to acting on $J_2^{S_{\Delta_{2k}}}$ are equivalent to
$I_2^{\otimes k-j} \otimes B^{\otimes j}$. For the odd vector generators, given by $e_{2j-1}\xrightarrow{\cong} I_2^{\otimes k-j  }\otimes E_1\otimes B^{\otimes j-1} $ for $j=1, \ldots, k$, we have the equivalent matrices $i\cdot (I_2^{\otimes k-j+1}\otimes B^{\otimes j-1})$ on $J_2^{S_{\Delta_{2k}}}$. By this construction, we see that all $2k$ vector generators $e_1, \ldots, e_{2k}$ of our Clifford algebra are unique. Now since the rest of the $2^{2k}$ representations of the  generators of the Clifford algebra $Cl(\C^{2k})$ are products of the $2k$ vector generators $e_1, \ldots, e_{2k}$, their representative matrices are products of $I_2^{\otimes k-j} \otimes B^{\otimes j}$ and  $i\cdot (I_2^{\otimes k-j+1}\otimes B^{\otimes j-1})$ on $J_2^{S_{\Delta_{2k}}}$. According to Lemma \ref{tensor lemma 1} and Corollary \ref{tensor corollary 1}, the products of the generators as they act on $J_2^{S_{\Delta_{2k}}}$ are of the form $C_1\otimes \cdots \otimes C_{k}$ or  $i \cdot (C_1\otimes \cdots \otimes C_{k})$, where each $C_{j}$ is a string of matrix  products of $I_2$s or $B$s. Noting that  $B^{2}=I_2$ acting on $J_2^{S_{\Delta_{2k}}}$, we have that each component $C_j$ is one of two options: $I_2$ or $B$. Hence there are a total of $2^{k}$ resulting products of the form $ C_1\otimes \cdots \otimes C_{k}$, and $2^{k}$ resulting products of the form $i \cdot (C_1\otimes \cdots \otimes C_{k}).$ Hence on $J_2^{S_{\Delta_{2k}}}$ we have a total of $2^{k}+2^{k}=2\cdot 2^{k}=2^{k+1}$ unique involutions acting on the 2-torsion points (where we include the identity in this count) induced from our $2^{2k}$ generators of our Clifford algebra $Cl(\C^{2k})$.

\end{proof}

Next we analyze the relationship between matrix representations of generators in $\hat{\Gamma}_{2k}$ and of generators in $\hat{\Gamma}_{2k+2}$. For this we introduce the following notation. 
\begin{itemize}
	\item Since, for a given $k$, a sequence $\mu$ might be an increasing subsequence of both $\{1, \ldots, 2k \}$ and $\{ 1, \ldots, 2k+2 \}$, we  write $\ek_\mu$ for the associated element of $\hat{\Gamma}_{2k}$, and $\ekp_\mu$ for the associated element of $\gh_{2k+2}$. 
 \item For ease of notation, we usually omit the set brackets and commas from sequences, so that, for example, the sequence $\mu = \{1, 3, 4\}$ is written as simply $\mu=134$. 
	\item If $\mu$ and $\nu$ are sequences from $\{1, \ldots, 2k \}$, we denote the concatenation of $\mu$ and $\nu$ (that is, $\mu$ followed by $\nu$) by $\mu \upSmallFrown \nu$. For example, if $\mu = 2$ and $\nu = 467$, then $\mu \upSmallFrown \nu = 2467$.
	
	\item If $\mu$ is a sequence in $\{ 1, \ldots, 2k \}$ and $n \in \mathbb{N}$, we denote by $\mu + n$ (respectively $\mu - n$) the sequence formed by replacing each $l \in \mu$ with $l+n$ (respectively $l-n$). 
\end{itemize}

The following lemma provides a description of how the representations of vector generators in $\gh_{2k+2}$ arise from those of vector generators in $\gh_{2k}$ via Kronecker products.

\begin{lem}\label{new vector generators in terms of old}
	The representations of the vector generators $\ekp_1, \ldots, \ekp_{2k+2} \in \gh_{2k+2}$ are formed from the representations of the vector generators $\ek_1, \ldots, \ek_{2k} \in \gh_{2k}$ as follows:
	\[
	\rho(\ekp_i) = \left\{
	\begin{array}{ll}
		I_{2^k} \otimes E_i & \textrm{ if i = $1$ or $2$} \\
		\rho({\ek_{i-2}}) \otimes B & \textrm{ if $i=3, 4, \ldots, $ or $2k+2$}
	\end{array}
	\right.
	\]
	where $I_{2^k}$ denotes the $2^k \times 2^k$ identity matrix.
\end{lem}	

\begin{proof}
	By Proposition \ref{Dirac representations}, we have 
	\[
	\begin{array}{lcl}
		
        \rho(\ekp_1) & = & I_2^{\otimes (k+1)-1} \otimes E_1 \otimes B^{\otimes 1-1} = I_2^{\otimes k} \otimes E_1 = I_{2^k} \otimes E_1 \\
		
		\rho(\ekp_2) & = & I_2^{\otimes (k+1)-1} \otimes E_2 \otimes B^{\otimes 1-1} = I_2^{\otimes k} \otimes E_2 = I_{2^k} \otimes E_2 \\
		\ \\
		
		\rho(\ekp_3) & = & I_2^{\otimes (k+1)-2} \otimes E_1 \otimes B^{\otimes 2-1} = (I_2^{\otimes k-1} \otimes E_1) \otimes B = \rho(\ek_1) \otimes B\\
		
		\rho(\ekp_4) & = & I_2^{\otimes (k+1)-2} \otimes E_2 \otimes B^{\otimes 2-1} = (I_2^{\otimes k-1} \otimes E_2) \otimes B = \rho(\ek_2) \otimes B\\
		\ \\
		
		\rho(\ekp_5) & = & I_2^{\otimes (k+1)-3} \otimes E_1 \otimes B^{\otimes 3-1} = (I_2^{\otimes k-2} \otimes E_1 \otimes B) \otimes B = \rho(\ek_3) \otimes B\\
		
		\rho(\ekp_6) & = & I_2^{\otimes (k+1)-3} \otimes E_2 \otimes B^{\otimes 3-1} = (I_2^{\otimes k-2} \otimes E_2 \otimes B) \otimes B = \rho(\ek_4) \otimes B\\
		
		& \vdots & \\
		
		\rho(\ekp_{2j-1}) & = & I_2^{\otimes k+1-j} \otimes E_1 \otimes B^{\otimes j-1} = (I_2^{\otimes k+1-j} \otimes E_1 \otimes B^{\otimes j-2}) \otimes B \\
  & = & \rho(\ek_{2j-3}) \otimes B\\
 
		\rho(\ekp_{2j}) & = & I_2^{\otimes k+1-j} \otimes E_2 \otimes B^{\otimes j-1} = (I_2^{\otimes k+1-j} \otimes E_2 \otimes B^{\otimes j-2}) \otimes B \\
  & = & \rho(\ek_{2j-2}) \otimes B\\
		
		& \vdots & \\
	
		\rho(\ekp_{2(k+1)-1}) & = & I_2^{\otimes k+1-(k+1)} \otimes E_1 \otimes B^{\otimes k+1-1} = E_1 \otimes B^{\otimes k} = (E_1 \otimes B^{\otimes k-1}) \otimes B \\
  & = & \rho(\ek_{2k-1}) \otimes B\\
		
		\rho(\ekp_{2(k+1)}) & = & I_2^{\otimes k+1-(k+1)} \otimes E_2 \otimes B^{\otimes k+1-1} = E_2 \otimes B^{\otimes k} = (E_2 \otimes B^{\otimes k-1}) \otimes B \\
  & = & \rho(\ek_{2k}) \otimes B\\
		
	\end{array}
	\]
\end{proof}

We next generalize the equivalence relation $\sim$, defined on $\gh_2$ in the proof of Theorem \ref{number of unique Clifford actions}, to $\gh_{2k}$ for any $k$.

\begin{defn}\label{equivalence mod two torsion points}
	Let $k \in \mathbb{N},\, k \geq 1$. For $\ek_{\mu}, \ek_{\eta} \in \gh_{2k}$, define the equivalence relation $\sim$ by $\ek_{\mu} \sim \ek_{\eta}$ if for all $\vec{v} \in J_2^{S_{\Delta_{2k}}}$, $\ek_{\mu} \cdot \vec{v} = \ek_{\eta} \cdot \vec{v}$. 
\end{defn}

One easily verifies that $\sim$ is an equivalence relation on $\gh_{2k}$. If $\ek_{\mu} \in \gh_{2k}$, we denote by $[\ek_{\mu}]$ the equivalence class of $\ek_{\mu}$ under the relation $\sim$; and we denote by $\faktor{\gh_{2k}}{\sim}$ the equivalence classes under $\sim$ considered as a group. Note that by Theorem \ref{number of unique Clifford actions}, this group has a total of $2^{k+1}$ classes. 

What we must remark  here is that the generators of the group $\faktor{\gh_{2k}}{\sim}$ are being viewed as operators on $J_2^{S_{\Delta_{2k}}}$, and not necessarily as multiplicative generators from a Clifford algebra setting. Moreover, this group is commutative, since all negatives are equivalent  to their positives when quotiented-out by our relation $\sim$ on $J_2^{S_{\Delta_{2k}}}$.  

\begin{lem}\label{new e mu in terms of old}
	Suppose $\mu = \{ i_1, \ldots, i_p \}$ is an increasing subsequence of $\{ 1, \ldots, 2k+2 \}$. Set $\underline{\mu} = \mu \setminus \{1, 2\}$, and denote by $\underline{\mu}-2$ the set obtained by subtracting $2$ from every element in the increasing subsequence $\underline{\mu}$. That is, if $ \mu \setminus \{1, 2\}=(i_{l_1}, \ldots ,i_{l_r})$, where $3\leq i_{l_1} \leq \cdots \leq i_{l_r}\leq 2k+2$, then  $ \underline{\mu}-2=(i_{l_1}-2,\ldots,i_{l_r}-2)$. (Note that $\mu$ or $\underline{\mu}$, and hence $\underline{\mu}-2$, could be the empty sequence.) Then we have the following: 
	
	\begin{enumerate}
		\item If $1, 2 \not\in \mu$, then $\rho(\ekp_\mu) = \left\{ \begin{array}{ll}
			\rho(\ek_{\underline{\mu} - 2}) \otimes I_2 & \textrm{ if $|\mu|$ is even} \\
			\rho(\ek_{\underline{\mu} - 2}) \otimes B & \textrm{ if $|\mu|$ is odd}
		\end{array} \right.$
		
		\item If $1 \in \mu$ and $2 \not\in \mu$, then $\rho(\ekp_\mu) = \left\{ \begin{array}{ll}
			\rho(\ekp_1) \cdot (\rho(\ek_{\underline{\mu} - 2}) \otimes B) & \textrm{ if $|\mu|$ is even} \\
			\rho(\ekp_1) \cdot (\rho(\ek_{\underline{\mu} - 2}) \otimes I_2) & \textrm{ if $|\mu|$ is odd}
		\end{array} \right.$
		
		\item If $1 \not\in \mu$ and $2 \in \mu$, then $\rho(\ekp_\mu) = \left\{ \begin{array}{ll}
			\rho(\ekp_2) \cdot (\rho(\ek_{\underline{\mu} - 2}) \otimes B) & \textrm{ if $|\mu|$ is even} \\
			\rho(\ekp_2) \cdot (\rho(\ek_{\underline{\mu} - 2}) \otimes I_2) & \textrm{ if $|\mu|$ is odd}
		\end{array} \right.$
		
		\item If $1, 2 \in \mu$, then $\rho(\ekp_\mu) = \left\{ \begin{array}{ll}
			\rho(\ekp_{12}) \cdot (\rho(\ek_{\underline{\mu} - 2}) \otimes I_2) & \textrm{ if $|\mu|$ is even} \\
			\rho(\ekp_{12}) \cdot (\rho(\ek_{\underline{\mu} - 2}) \otimes B) & \textrm{ if $|\mu|$ is odd}
		\end{array} \right.$
		
	\end{enumerate}
\end{lem}

\begin{proof}
	For (1): suppose $1, 2 \not\in \mu$. Then 
	\[
	\begin{array}{lll}
		\rho(\ekp_\mu) & = & \rho(\ekp_{i_1}) \cdots \rho(\ekp_{i_p}) \\
		& = & (\rho(\ek_{i_1-2}) \otimes B) \cdots  (\rho(\ek_{i_p - 2}) \otimes B) \textrm{ (by Lemma \ref{new vector generators in terms of old}, as $i_1 \geq 3$)}\\
		& = & (\rho(\ek_{i_1 - 2}) \cdots \rho(\ek_{i_p-2})) \otimes B^p \textrm{ (by Corollary \ref{tensor corollary 1}) }\\
		& = & \rho(\ek_{\mu - 2}) \otimes B^{|\mu|} \\
		& = & \rho(\ek_{\underline{\mu} - 2}) \otimes B^{|\mu|} \textrm{ (as $1, 2 \not\in \{i_1, \ldots, i_p\}$) }\\
		& = & \left\{\begin{array}{ll}
			\rho(\ek_{\underline{\mu}-2}) \otimes I_2 & \textrm{ if $|\mu|$ is even} \\
			\rho(\ek_{\underline{\mu}-2}) \otimes B & \textrm{ if $|\mu|$ is odd}
		\end{array} \right.
	\end{array}
	\]
	For (2): suppose $1 \in\mu$ and $2 \not\in \mu$. Then
	\[
	\begin{array}{lll}
		\rho(\ekp_\mu) & = & \rho(\ekp_1) \cdot \rho(\ekp_{i_2}) \cdots \rho(\ekp_{i_p}) \\
		& = & \rho(\ekp_1) \cdot (\rho(\ek_{i_2-2}) \otimes B) \cdots (\rho(\ek_{i_p-2}) \otimes B) \textrm{ (by Lemma \ref{new vector generators in terms of old}, as $i_2 \geq 3$)}\\
		& = & \rho(\ekp_1) \cdot [(\rho(\ek_{i_2-2}) \cdots \rho(\ek_{i_p-2})) \otimes B^{p-1}] \textrm{ (by Corollary \ref{tensor corollary 1})}\\
		& = & \left\{\begin{array}{ll}
			\rho(\ekp_1) \cdot (\rho(\ek_{\underline{\mu}-2}) \otimes B) & \textrm{ if $|\mu|$ is even} \\
			\rho(\ekp_1) \cdot (\rho(\ek_{\underline{\mu}-2}) \otimes I_2) & \textrm{ if $|\mu|$ is odd}
		\end{array} \right.
	\end{array}
	\]
	
	For (3): suppose $1 \not\in\mu$ and $2 \in \mu$. Then
	\[
	\begin{array}{lll}
		\rho(\ekp_\mu) & = & \rho(\ekp_2) \cdot \rho(\ekp_{i_2}) \cdots \rho(\ekp_{i_p}) \\
		& = & \rho(\ekp_2) \cdot (\rho(\ek_{i_2-2}) \otimes B) \cdots (\rho(\ek_{i_p-2}) \otimes B) \textrm{ (by Lemma \ref{new vector generators in terms of old}, as $i_2 \geq 3$)}\\
		& = & \rho(\ekp_2) \cdot [(\rho(\ek_{i_2-2}) \cdots \rho(\ek_{i_p-2})) \otimes B^{p-1}] \textrm{ (by Corollary \ref{tensor corollary 1})}\\
		& = & \left\{\begin{array}{ll}
			\rho(\ekp_2) \cdot (\rho(\ek_{\underline{\mu}-2}) \otimes B) & \textrm{ if $|\mu|$ is even} \\
			\rho(\ekp_2) \cdot (\rho(\ek_{\underline{\mu}-2}) \otimes I_2) & \textrm{ if $|\mu|$ is odd}
		\end{array} \right.
	\end{array}
	\]
	
	For (4): suppose $1, 2 \in \mu$. Then
	\[
	\begin{array}{lll}
		\rho(\ekp_\mu) & = & \rho(\ekp_{12}) \cdot \rho(\ekp_{i_3}) \cdots \rho(\ekp_{i_p}) \\
		& = & \rho(\ekp_{12}) \cdot (\rho(\ek_{i_3-2}) \otimes B) \cdots (\rho(\ek_{i_p-2}) \otimes B) \textrm{ (by Lemma \ref{new vector generators in terms of old}, as $i_3 \geq 3$)}\\
		& = & \rho(\ekp_{12}) \cdot [(\rho(\ek_{i_3-2}) \cdots \rho(\ek_{i_p-2})) \otimes B^{p-2}] \textrm{ (by Corollary \ref{tensor corollary 1})}\\
		& = & \left\{\begin{array}{ll}
			\rho(\ekp_{12}) \cdot (\rho(\ek_{\underline{\mu}-2}) \otimes I_2) & \textrm{ if $|\mu|$ is even} \\
			\rho(\ekp_{12}) \cdot (\rho(\ek_{\underline{\mu}-2}) \otimes B) & \textrm{ if $|\mu|$ is odd}
		\end{array} \right.
	\end{array}
	\]
\end{proof}

With a general understanding of what these matrix representations look like and how their negatives are quotiented-away when acting on 2-torsion points, we turn our interest toward the general shape of the matrices. 

\begin{defn}\label{matrix shape}
	For $n \in \mathbb{N}, \, n \geq 1$ and $M = (m_{ij}) \in \mathbb{C}(n)$, define the \textbf{shape} of $M$ to be $\textrm{Sh}(M)=(s_{ij})$ where for $1 \leq i, j \leq n$, \[s_{ij} = \left\{\begin{array}{ll}
		1 & \textrm{ if } m_{ij} \neq 0\\
		0 & \textrm{ if } m_{ij} = 0
	\end{array} \right.
	\]
	That is, for each $M$, $\textrm{Sh}(M)$ is the matrix obtained from $M$ by replacing each of its nonzero entries with a {\rm 1}.
\end{defn}

Now recall that a \textit{permutation matrix} is an $n \times n$ matrix (for some $n \in \mathbb{N}, n > 0$) with exactly one 1 in each row and each column, and zeros elsewhere. Any permutation matrix $P \in \mathbb{C}(n)$ is the result of switching the rows of the $n \times n$ identity matrix $I_n$ according to some permutation $\sigma$ on $\{1, \ldots, n\}$. The result of applying $P$ to an $n$-vector $\vec{v} \in \mathbb{C}^n$ is a permutation of the entries of $\vec{v}$ according to  $\sigma$. The product of permutation matrices, being equivalent to the composition of permutations on $\{1, \ldots, n\}$, is another permutation matrix. Also, observe that the Kronecker product of permutation matrices is again a permutation matrix. One can check that if both $\operatorname{Sh}(M)$ and $\operatorname{Sh}(N)$ are permutation matrices, then 

(i) $\textrm{Sh}(M \cdot N) = \textrm{Sh}(M) \cdot \textrm{Sh}(N)$ if $M, N \in \mathbb{C}(n)$, 

and 

(ii) $\textrm{Sh}(M \otimes N) = \textrm{Sh}(M) \otimes \textrm{Sh}(N)$.

\begin{lem}\label{e1, e2, and e12 have permutation shapes}
	For any $k \in \mathbb{N},\, k \geq 1$, $\textrm{Sh}(\rho(\ek_1)), \textrm{Sh}(\rho(\ek_2)), $ and $\textrm{Sh}(\rho(\ek_{12}))$ are $2^k \times 2^k$ permutation matrices. 
\end{lem}

\begin{proof} 
	Fix $k \in \mathbb{N},\, k \geq 1$. By inspection, $\textrm{Sh}(E_1), \textrm{Sh}(E_2), $ and $\textrm{Sh}(E_{12})$ are $2 \times 2$ permutation matrices, and $\textrm{Sh}(I_{2^{k-1}})$ is a $2^{k-1} \times 2^{k-1}$ permutation matrix. By Lemma \ref{new vector generators in terms of old}, $\rho(\ek_1) = I_{2^{k-1}} \otimes E_1$, and hence $\textrm{Sh}(\rho(\ek_1)) = \textrm{Sh}(I_{2^{k-1}} \otimes E_1) = \textrm{Sh}(I_{2^{k-1}}) \otimes \textrm{Sh}(E_1)$ is a $2^k \times 2^k$ permutation matrix. Similarly, $\textrm{Sh}(\rho(\ek_2)) = \textrm{Sh}(I_{2^{k-1}} \otimes E_2) = \textrm{Sh}(I_{2^{k-1}}) \otimes \textrm{Sh}(E_2)$ is a $2^k \times 2^k$ permutation matrix. Then $\textrm{Sh}(\rho(\ek_{12})) = \textrm{Sh}(\rho(\ek_1) \cdot \rho(\ek_2)) = \textrm{Sh}(\rho(\ek_1)) \cdot \textrm{Sh}(\rho(\ek_2))$ is a $2^k \times 2^k$ permutation matrix.
\end{proof}

\begin{prop}\label{e mu have permutation shapes} For all $k \in \mathbb{N}, \, k \geq 1$, and for all generators $e_\mu \in \gh_{2k}$, $\textrm{Sh}(\rho(e_\mu))$ is a $2^k \times 2^k$ permutation matrix. 
\end{prop}
\begin{proof}
	It is sufficient to prove the proposition for all positive $\ek_\mu \in \gh_{2k}$, i.e.\ all $\ek_\mu$ for which $\mu$ is an increasing subsequence of $\{ 1, \ldots, 2k \}$, as it is clear that $\textrm{Sh}(\rho(\ek_\mu)) = \textrm{Sh}(\rho(-\ek_\mu))$.
	
	By inspection, the claim holds for each $\stackrel{1}{e}_\mu \in \gh_2$. Suppose it holds for some $k \geq 1$, and let $\ekp_\mu \in \gh_{2k+2}$. Let $\underline{\mu}$ denote $\mu \setminus \{1,2\}$. We know from Lemma \ref{new e mu in terms of old} that $\rho(\ekp_\mu)$ has been formed from the element $\rho(\ek_{\underline{\mu}-2})$ of $\gh_{2k}$ in one of eight ways, by tensoring $\rho(\ek_{\underline{\mu}-2})$ by either $I_2$ or $B$  and then possibly matrix-multiplying on the left by $\rho(\ekp_1), \rho(\ekp_2), $ or $\rho(\ekp_{12})$. Both $I_2$ and $B$ have the shapes of permutation matrices; $\textrm{Sh}(\rho(\ek_{\underline{\mu}-2}))$ is a permutation matrix by the induction hypothesis; and $\rho(\ekp_1), \rho(\ekp_2)$, and $\rho(\ekp_{12})$ have the shapes of permutation matrices by Lemma \ref{e1, e2, and e12 have permutation shapes}. In all cases, $\rho(\ekp_\mu)$ has the same dimensions as either $\rho(\ek_{\underline{\mu}-2}) \otimes I_2$ or $\rho(\ek_{\underline{\mu}-2}) \otimes B$, namely $(2^k \cdot 2) \times (2^k \cdot 2) = 2^{k+1} \times 2^{k+1}$, and  $\textrm{Sh}(\rho(\ekp_\mu))$ is a permutation matrix.
\end{proof}

The next three lemmas describe the effect of matrix multiplying on the left by the representation of either $\ek_1$, $\ek_2$, or $\ek_{12}$ (for some $k \in \mathbb{N}, k \geq 1$). Recall that the presence of negative signs in matrices acting on 2-torsion points has no effect on the action. That is, if $M \in \mathbb{C}^{2^k}$ is a complex matrix acting on $J_2^{S_{\Delta_{2k}}}$ and $M'$ is a matrix obtained by replacing some or all of the entries $m_{ij}$ in $M$ with $-m_{ij}$, then $M \cdot \vec{v} = M' \cdot \vec{v}$ for all $\vec{v} \in J_2^{S_{\Delta_{2k}}}$.

\begin{lem}\label{what multiplication by e1 does}
	Let $M$ be any $2^k \times 2^k$ matrix. The effect of matrix multiplying $M$ on the left by $\rho(\ek_1)$ is to replace each entry $m_{ij}$ of $M$ by either $i m_{ij}$ or $-i m_{ij}$. 
	
\end{lem}

\begin{proof}
	Fix $k \in \mathbb{N}$, $k \geq 1$, and let $M \in \mathbb{C}(2^k)$. Note that it is sufficient to show that $\rho(\ek_1)$ has the following form:
	\[
	(\star) \hspace{5mm}\left[
	\begin{array}{cccc}
		\pm i & & &  \\
		& \pm i & &  \\
		& & \ddots  & \\
		& & & \pm i
	\end{array}
	\right]
	\] 
	(with zeros off the main diagonal). By Lemma \ref{new vector generators in terms of old}, $\rho(\ek_1) = I_{2^{k-1}} \otimes E_1 = I_{2^{k-1}} \otimes \left[ \begin{array}{ll}
		i & 0 \\
		0 & -i \\
	\end{array} \right]$, which clearly is a $2^k \times 2^k$ matrix of the form $(\star)$.
	
\end{proof}

\begin{lem}\label{what e2 does}
	Let $M$ be any $2^k \times 2^k$ matrix. The effect of matrix multiplying $M$ on the left by $\rho(\ek_2)$ is to replace each entry $m_{ij}$ of $M$ by $im_{ij}$, and then to interchange rows $1$ with $2$, $3$ with $4, \ldots, $ and $2^k - 1$ with $2^k$. 
\end{lem}

\begin{proof}
	Fix $k \in \mathbb{N}$, $k \geq 1$, and let $M \in \mathbb{C}(2^k)$. Note that it is sufficient to show that $\rho(\ek_2)$ has the following form:
	\[
	(\star) \hspace{5mm}\left[
	\begin{array}{ccccccccc}
		0 &  i&  & & & & & \\
		i & 0 & & & & & &\\
		& & 0 &  i & & & &\\
		& &  i & 0 & & & & \\
		& & & & \ddots & & & & \\
		& & & & & & 0 &  i & \\
		& & & & & &  i & 0 & 
	\end{array}
	\right]
	\] 
	(with zeros other than the entries shown). By Lemma \ref{new vector generators in terms of old}, $\rho(\ek_2) = I_{2^{k-1}} \otimes E_2 = I_{2^{k-1}} \otimes \left[ \begin{array}{ll}
		0 & i \\
		i & 0 \\
	\end{array} \right]$, which clearly is a $2^k \times 2^k$ matrix of the form $(\star)$.
\end{proof}

\begin{lem}\label{what e12 does}
	Let $M$ be any $2^k \times 2^k$ matrix. The effect of multiplying $M$ on the left by $\rho(\ek_{12})$ is to replace each entry $m_{ij}$ of $M$ by $m_{ij}$ or $-m_{ij}$, and then to interchange rows $1$ with $2$, $3$ with $4, \ldots, $ and $2^k - 1$ with $2^k$.
\end{lem}

\begin{proof}
	Fix $k \in \mathbb{N}$, $k \geq 1$, and let $M \in \mathbb{C}(2^k)$. Note that it is sufficient to show that $\rho(\ek_{12})$ has the following form:
	\[
	(\star) \hspace{5mm}\left[
	\begin{array}{ccccccccc}
		0 & \pm 1 &  & & & & & \\
		\pm 1 & 0 & & & & & &\\
		& & 0 & \pm 1 & & & &\\
		& &  \pm 1 & 0 & & & & \\
		& & & & \ddots & & & & \\
		& & & & & & 0 &  \pm 1 & \\
		& & & & & & \pm 1 & 0 & 
	\end{array}
	\right]
	\]
	(with zeros other than the entries shown).  Using  Lemmas \ref{tensor corollary 1} and \ref{new vector generators in terms of old}, we have 
 \begin{align*}
\rho(\ek_{12}) &= \rho(\ek_1) \cdot \rho(\ek_2) \\
&= (I_{2^{k-1}} \otimes E_1) \cdot (I_{2^{k-1}} \otimes E_2) \\ &= (I_{2^{k-1}} \cdot I_{2^{k-1}})\otimes (E_1 \cdot E_2) \\ &= I_{2^{k-1}} \otimes \left[ \begin{array}{cc}
		0 & -1 \\
		1 & 0 \\
	\end{array} \right],
 \end{align*}
 which clearly is a $2^k \times 2^k$ matrix of the form $(\star)$.
\end{proof}

In the following lemma, we show that the generators $e_\mu \in \gh_{2k}$ are divided into two different \textit{types}: real (with all nonzero entries in $\rho(e_\mu)$ being $\pm 1$), and imaginary (with all nonzero entries in $\rho(e_\mu)$ being $\pm i$).

\begin{lem}\label{flavors: real or complex}
	Let $e_\mu$ be any generator in $\gh_{2k}$,  for some $k \in \mathbb{N}, \, k \geq 1$. Then either each nonzero entry in $\rho(e_\mu)$ is in $\{-1,1\}$, or each nonzero entry in $\rho(e_\mu)$ is in $\{i, -i\}$.
\end{lem}

\begin{proof}
	We prove this by induction on $k$. The base case ($k=1$) holds by inspection. Suppose it holds for some $k \geq 1$. Let $\ekp_\mu \in \gh_{2k+2}$. As in Lemma \ref{new e mu in terms of old}, we set $\underline{\mu} = \mu \setminus \{1, 2\}$. By induction hypothesis, $\rho(\ek_{\underline{\mu}-2})$ either has all nonzero entries in $\{-1,1\}$ or has all nonzero entries in $\{-i, i\}$. Also, $|\mu|$ is either even or odd; so there are four main cases to consider. We present here the case for when $\rho(\ek_{\underline{\mu}-2})$ has all nonzero entries in $\{-1,1\}$ and $|\mu|$ is even (the proofs of the other three cases follow similar steps). By Lemma \ref{new e mu in terms of old}, within this case there are four subcases, depending on which of the numbers 1 and/or 2 are elements of the permutation $\mu$. 
	
	If $1, 2 \not\in \mu$, then $\rho(\ekp_\mu) = \rho(\ek_{\underline{\mu}-2}) \otimes I_2$; and Kronecker multiplying $\rho(\ek_{\underline{\mu}-2})$ by $I_2$ results in a matrix all of whose nonzero entries are still in $\{-1,1\}$.
	
	If $1 \in \mu$ and $2 \not\in \mu$, then $\rho(\ekp_\mu) = \rho(\ekp_1) \cdot  (\rho(\ek_{\underline{\mu}-2}) \otimes B)$. Kronecker multiplying $\rho(\ek_{\underline{\mu}-2})$ by $B$ results in a matrix all of whose nonzero entries are in $\{-i,i\}$. By Lemma \ref{what multiplication by e1 does}, $\rho(\ekp_1) \cdot  (\rho(\ek_{\underline{\mu}-2}) \otimes B)$ is a matrix all of whose nonzero entries are in $\{-1,1\}$.
	
	If $1 \not\in \mu$ and $2 \in \mu$, then $\rho(\ekp_\mu) = \rho(\ekp_2) \cdot  (\rho(\ek_{\underline{\mu}-2}) \otimes B)$. Kronecker multiplying $\rho(\ek_{\underline{\mu}-2})$ by $B$ results in a matrix all of whose nonzero entries are in $\{-i,i\}$. By Lemma \ref{what e2 does}, $\rho(\ekp_2) \cdot  (\rho(\ek_{\underline{\mu}-2}) \otimes B)$ is a matrix all of whose nonzero entries are in $\{-1,1\}$.
	
	If $1, 2 \in \mu$, then $\rho(\ekp_\mu) = \rho(\ekp_{12}) \cdot  (\rho(\ek_{\underline{\mu}-2}) \otimes I_2)$. Kronecker multiplying $\rho(\ek_{\underline{\mu}-2})$ by $I_2$ results in a matrix all of whose nonzero entries are in $\{-1,1\}$. By Lemma \ref{what e12 does}, $\rho(\ekp_{12}) \cdot  (\rho(\ek_{\underline{\mu}-2}) \otimes I_2)$ is a matrix all of whose nonzero entries are in $\{-1,1\}$.
	
	Then we have that when $\rho(\ek_{\underline{\mu}-2})$ has all nonzero entries in $\{-1,1\}$ and $|\mu|$ is even, all of the nonzero entries in $\rho(\ekp_\mu)$ are in $\{-1,1\}$. 
	
	Hence by similar induction steps in all cases, either $\rho(\ekp_\mu)$ has all of its nonzero entries in $\{-1,1\}$, or it has all of its nonzero entries in $\{-i,i\}$.
\end{proof}

Observe that if $P, P'$ are distinct $2^k \times 2^k$ permutation matrices, then there is a $\vec{v} \in J_2^{S_{\Delta_{2k}}}$ such that $P \cdot \vec{v} \neq P' \cdot \vec{v}$. By this and Lemma \ref{flavors: real or complex}, we have:

\begin{lem}\label{things in same class have same shape}
	Any two elements of a given equivalence class mod $\sim$ have the same shape. \hfill $\square$
\end{lem}

Next we show that all representations in a given equivalence class have the same type, either real or imaginary.

\begin{lem}\label{things in same class have same flavor}
	Any two elements of a given equivalence class mod $\sim$ have the same kind of nonzero entries: either all matrices in the class have nonzero entries in $\{1, -1\}$, or all matrices in the class have nonzero entries in $\{i, -i\}$. That is, within a given class, either all $\rho(e_\mu)$ have real type, or all $\rho(e_\mu)$ have imaginary type.
\end{lem}

\begin{proof}
	Suppose by contradiction that for some $\ek_\mu, \ek_\eta \in \gh_{2k}$ with $\ek_\mu \sim \ek_\eta$, $\rho(\ek_\mu)$ had real type while $\rho(\ek_\eta)$ had imaginary type. Set $\vec{v}$ to be the constant $v_1=\frac{1}{2}$ vector in $J_2^{S_{\Delta_{2k}}}$. Then $\rho(\ek_\mu) \cdot \vec{v} = \vec{v}$, but $\rho(\ek_\eta) \cdot \vec{v} = i \textrm{Sh}(\rho(\ek_\eta)) \cdot \vec{v} = i \textrm{Sh}(\rho(\ek_\mu)) \cdot \vec{v}= i \vec{v}$, which is the constant $v_2=\frac{i}{2}$ vector; but this is a contradiction since $\rho(\ek_\mu)$ and $\rho(\ek_\eta)$ should act identically on $\vec{v}$.
\end{proof}

Observe that each strictly increasing subsequence $\mu \subseteq \{ 1, \ldots, 2k+2 \}$ can be obtained from a strictly increasing subsequence $\nu \subseteq \{ 3, \ldots, 2k+2 \}$ by prepending either $\emptyset$, $\{1\}$, $\{2\}$, or $\{1,2\}$ to $\nu$. (We speak of prepending $\emptyset$ to a sequence because even though doing so does not change the sequence itself, when $\ek_\mu \in \gh_{2k}$, $\emptyset \upSmallFrown (\mu+2)$ will be one of four sequences obtained from $\mu$ and used in Lemma \ref{e mu hats split into 4 new e mu} to construct new generators in $\gh_{2k+2}$.) Each such sequence $\nu$, in turn, can be obtained from a strictly increasing subsequence $\eta \subseteq \{ 1, \ldots, 2k \}$ by adding 2 to each element of $\eta$. 

This means that we have the following four bijections:

\begin{enumerate}
	\item A bijection between strictly increasing subsequences of $\{ 1, \ldots, 2k \}$ and strictly increasing subsequences of $\{ 1, \ldots, 2k+2 \}$ that contain neither 1 nor 2;
	\item A bijection between strictly increasing subsequences of $\{ 1, \ldots, 2k \}$ and strictly increasing subsequences of $\{ 1, \ldots, 2k+2 \}$ that contain 1 but not 2;
	\item A bijection between strictly increasing subsequences of $\{ 1, \ldots, 2k \}$ and strictly increasing subsequences of $\{ 1, \ldots, 2k+2 \}$ that contain 2 but not 1; and 
	\item A bijection between strictly increasing subsequences of $\{ 1, \ldots, 2k \}$ and strictly increasing subsequences of $\{ 1, \ldots, 2k+2 \}$ that contain both 1 and 2.
\end{enumerate}

Therefore we have the corresponding four bijections between $\gh_{2k}$ and subsets of $\gh_{2k+2}$:

\begin{enumerate}
	\item A bijection between $\gh_{2k}$ and $\{e_\mu \in \gh_{2k+2}: 1, 2 \not\in \mu\}$;
	\item A bijection between $\gh_{2k}$ and $\{e_\mu \in \gh_{2k+2}: 1 \in \mu, 2 \not\in \mu\}$;
	\item A bijection between $\gh_{2k}$ and $\{e_\mu \in \gh_{2k+2}: 1 \not\in \mu, 2 \in \mu\}$; and 
	\item A bijection between $\gh_{2k}$ and $\{e_\mu \in \gh_{2k+2}: 1, 2 \in \mu\}$.
\end{enumerate}

Since these four subsets of $\gh_{2k+2}$ are all disjoint and since their union is all of $\gh_{2k+2}$, we have (again) that $|\gh_{2k+2}| = 4 |\gh_{2k}|$.

\begin{exmp}\label{ex: 4 new elements}
Consider the sequence $\mu = 46 \subseteq \{ 1, 2, 3, 4, 5, 6 \} $ and the corresponding generator $\stackrel{3}{e}_{46}$ in $\gh_6$. (Recall that for ease of notation, we are writing sequences without  brackets or commas.) The four subsequences of $\{ 1, 2, 3, 4, 5, 6, 7, 8 \}$  that are formed from $\mu$ by first adding $2$ to each element in $\mu$ and then prepending either $\emptyset, \{1\}, \{2\}, $ or $\{12\}$ are $\emptyset \upSmallFrown 68 = 68, \, 1 \upSmallFrown68 = 168, \, 2 \upSmallFrown 68 = 268$, and $12 \upSmallFrown 68 = 1268$. The corresponding four new elements of $\gh_8$ that are formed from $\stackrel{3}{e}_{46} \in \gh_6$ are $\stackrel{4}{e}_{68}, \, \stackrel{4}{e}_{168}, \, \stackrel{4}{e}_{268}$, and $\stackrel{4}{e}_{1268}$.
\end{exmp}

\begin{lem}\label{e mu hats split into 4 new e mu}
	Let $\ek_{\mu}$ be a generator in $\gh_{2k}$, and set $\mu' = \mu+2$.  There are exactly four elements of $\gh_{2k+2}$ that correspond to the increasing sequences $\emptyset \upSmallFrown \mu'$, $1 \upSmallFrown \mu'$, $2 \upSmallFrown \mu'$, and $12 \upSmallFrown \mu'$:
	\[ \ekp_{\emptyset \upSmallFrown \mu' } = \ekp_\emptyset  \ekp_{\mu'}, \ \ekp_{1 \upSmallFrown \mu' } = \ekp_1  \ekp_{\mu'}, \]
	
	\[\ekp_{2 \upSmallFrown \mu' } = \ekp_2  \ekp_{\mu'}, \textrm{ and } \ekp_{12 \upSmallFrown \mu' } = \ekp_{12}  \ekp_{\mu'}. \]
	
	Also, $  \rho(\ekp_{\mu'}) = \left\{ 
	\begin{array}{ll}
		\rho(\ek_{\mu}) \otimes I_2 & \textrm{ if } |\mu| \textrm{ is even} \\
		\rho(\ek_{\mu}) \otimes B & \textrm{ if } |\mu| \textrm{ is odd} 
	\end{array}
	\right. $
	
\end{lem}

\begin{proof}
	The first claim follows from the preceding discussion. For the second claim: we have defined $\mu'$ as $\mu+2$, so $1, 2 \not\in \mu'$. Then in this case $(\mu'\setminus \{1,2\})-2 = \mu'-2 = (\mu+2)-2 = \mu$. Then by Lemma \ref{new e mu in terms of old}, 
	\[
	\begin{array}{lll}
		\rho(\ekp_{\mu'}) & = & \left\{ \begin{array}{ll}
			\rho(\ek_{(\mu' \setminus \{1,2\})-2}) \otimes I_2 & \textrm{ if } |\mu'| \textrm{ is even} \\
			\rho(\ek_{(\mu' \setminus \{1,2\})-2}) \otimes B & \textrm{ if } |\mu'| \textrm{ is odd}
		\end{array}
		\right. \\
		
		& = & \left\{ \begin{array}{ll}
			\rho(\ek_{\mu}) \otimes I_2 & \textrm{ if } |\mu'| \textrm{ is even} \\
			\rho(\ek_{\mu}) \otimes B & \textrm{ if } |\mu'| \textrm{ is odd}
		\end{array}
		\right. \\
		
		& = & \left\{ \begin{array}{ll}
			\rho(\ek_{\mu}) \otimes I_2 & \textrm{ if } |\mu| \textrm{ is even} \\
			\rho(\ek_{\mu}) \otimes B & \textrm{ if } |\mu| \textrm{ is odd}
		\end{array}
		\right. 	
	\end{array}
	\]
\end{proof}

\begin{lem}\label{2 shapes, in 2 flavors each}
Let $k \in \mathbb{N}, \, k \geq 1$.	
 Let $\ek_{\mu} \in \gh_{2k}$ be a generator, and set $\mu' = \mu + 2$. The four generators  $\ekp_{\emptyset \upSmallFrown {\mu}'}$, $\ekp_{1 \upSmallFrown {\mu}'}$, $\ekp_{2 \upSmallFrown {\mu}'}$, and $\ekp_{12 \upSmallFrown {\mu}'}$ have matrix representations that occur in  two new shapes in $\gh_{2k+2}$, each occurring in real and imaginary types.
\end{lem} 

\begin{proof}
	Let $\ek_{\mu} \in \gh_{2k}$. Suppose that each nonzero entry in $\rho(\ek_{\mu})$ is in $\{1, -1\}$ -- that is, that $\rho(\ek_{\mu})$ is of real type -- and that $|\mu|$ is even. This is the first of four cases; the remaining three (depending on whether the type of $\rho(\ek_{\mu})$ is real or imaginary, and whether the length of the sequence $\mu$ is even or odd) are similar to the first.
	
	Since we have assumed $|\mu|$ is even and each nonzero entry in $\rho(\ek_{\mu})$ is in $\{1, -1\}$, by Lemma \ref{e mu hats split into 4 new e mu}, $\ekp_{\emptyset \upSmallFrown \mu'} = \rho(\ek_{\mu}) \otimes I_2$; and this is a matrix in which each nonzero entry of $\rho(\ek_{\mu})$ has been replaced by a matrix equivalent mod $\sim$ to $\left( \begin{array}{cc} 1 & 0 \\ 0 & 1 \end{array} \right)$ (while each zero entry has been replaced by a $2 \times 2$ block of zeros). Also by Lemma \ref{e mu hats split into 4 new e mu}, $\rho(\ekp_{1 \upSmallFrown \mu'}) = \rho(\ekp_1) \cdot (\rho(\ek_{\mu}) \otimes I_2)$. By Lemma \ref{what multiplication by e1 does}, the resulting matrix is one in which each nonzero entry of $\rho(\ek_{\mu})$ has been replaced by a matrix equivalent mod $\sim$ to $\left( \begin{array}{cc} i & 0 \\ 0 & -i \end{array} \right)$. Therefore $\rho(\ekp_{\emptyset \upSmallFrown \mu'})$ has the same shape as $\rho(\ekp_{1 \upSmallFrown \mu'})$, but the first of these matrices has real type while the second has imaginary type. Next: by Lemma \ref{e mu hats split into 4 new e mu}, $\rho(\ekp_{2\upSmallFrown \mu'}) = \rho(\ekp_2) \cdot (\rho(\ek_{\mu}) \otimes I_2)$. By Lemma \ref{what e2 does}, the resulting matrix is one in which every nonzero entry of $\rho(\ek_{\mu})$ has been replaced by a matrix equivalent mod $\sim$ to $\left( \begin{array}{cc} 0 & i \\ i & 0 \end{array} \right)$. Finally, by Lemma \ref{e mu hats split into 4 new e mu}, $\rho(\ekp_{12\upSmallFrown \mu'}) = \rho(\ekp_{12}) \cdot (\rho(\ek_{\mu}) \otimes I_2)$. By Lemma \ref{what e12 does}, the resulting matrix is one in which every nonzero entry of $\rho(\ek_{\mu})$ has been replaced by a matrix equivalent mod $\sim$ to $\left( \begin{array}{cc} 0 & -1 \\ 1 & 0 \end{array} \right)$. Therefore $\rho(\ekp_{2 \upSmallFrown \mu'})$ has the same shape as $\rho(\ekp_{12 \upSmallFrown \mu'})$, but the first of these matrices has imaginary type while the second has real type. Thus the four representative matrices of generators in $\gh_{2k+2}$ that arise from those of the generator $\ek_\mu \in \gh_{2k}$ come in two shapes, and each of those shapes comes in one real and one imaginary type.
	
	The three remaining cases are handled similarly, using Lemmas \ref{e mu hats split into 4 new e mu}, \ref{what multiplication by e1 does}, \ref{what e2 does}, and \ref{what e12 does}.
\end{proof}

The following theorem proves that matrix representations of elements of $\gh_{2k}$ with the same shape provide us with the same equivalence classes in $\hat{\Gamma}_{2k+2}$ with respect to action on the 2-torsion points. 

\begin{thm}\label{same shapes give rise to same equivalence classes}
	Let $\ek_{\mu}, \ek_{\eta} \in \gh_{2k}$ for some $k \in \mathbb{N}, k \geq 1$. Then $\ek_{\mu}$ and $\ek_{\eta}$ give rise to the same four equivalence classes mod $\sim$ in $\gh_{2k+2}$ if and only if they have the same shape. That is,
	\[
	\begin{array}{ll}
		&\left\{ \left[ \ekp_{\emptyset \upSmallFrown \mu'} \right], \left[ \ekp_{1 \upSmallFrown \mu'} \right], \left[ \ekp_{2 \upSmallFrown \mu'} \right], \left[ \ekp_{12 \upSmallFrown \mu'} \right] \right\} \\[15pt]
		
		= & \left\{\left[ \ekp_{\emptyset \upSmallFrown \eta'} \right], \left[ \ekp_{1 \upSmallFrown \eta'} \right], \left[ \ekp_{2 \upSmallFrown \eta'} \right], \left[ \ekp_{12 \upSmallFrown \eta'} \right] \right\}
	\end{array}
	\]
	if and only if $\operatorname{Sh}(\rho(\ek_{\mu})) = \operatorname{Sh}(\rho(\ek_{\eta}))$.
	
	Moreover, the sets $\left\{ \left[ \ekp_{\emptyset \upSmallFrown \mu'} \right], \left[ \ekp_{1 \upSmallFrown \mu'} \right], \left[ \ekp_{2 \upSmallFrown \mu'} \right], \left[ \ekp_{12 \upSmallFrown \mu'} \right] \right\}$ and \\ $\left\{\left[ \ekp_{\emptyset \upSmallFrown \eta'} \right], \left[ \ekp_{1 \upSmallFrown \eta'} \right], \left[ \ekp_{2 \upSmallFrown \eta'} \right], \left[ \ekp_{12 \upSmallFrown \eta'} \right] \right\}$ are disjoint if $\operatorname{Sh}(\rho(\ek_{\mu})) \neq \operatorname{Sh}(\rho(\ek_{\eta}))$.
	
\end{thm}

\begin{proof}
Let $\ek_{\mu} \in \gh_{2k}$, and set $\mu' = \mu + 2$. We begin by making some observations about the matrix representations of the four elements  $\ekp_{\emptyset \upSmallFrown \mu'},$  $\ekp_{1 \upSmallFrown \mu'}, \ekp_{2 \upSmallFrown \mu'},$ and $\ekp_{12 \upSmallFrown \mu'}$ of $\gh_{2k+2}$ that arise from $\ek_{\mu}$ as in Lemma \ref{e mu hats split into 4 new e mu}.
	
	First consider $\ekp_{\emptyset \upSmallFrown \mu'}$. By Lemma \ref{e mu hats split into 4 new e mu}, 
	\[
	\rho(\ekp_{\emptyset \upSmallFrown \mu'}) = \rho(\ekp_{\mu'}) = \left\{
	\begin{array}{lll}
		\rho(\ek_{\mu}) \otimes I_2 & \textrm{ if } & |\mu| \textrm{ is even } \\
		\rho(\ek_{\mu}) \otimes B & \textrm{ if } & |\mu| \textrm{ is odd }
	\end{array}
	\right.
	\]
	Taking the Kronecker product of $\rho(\ek_{\mu})$ with $I_2$ on the right replaces each nonzero entry of $\rho(\ek_{\mu})$ with a block equivalent mod $\sim$ to either $\left[ \begin{array}{ll} 1 & 0 \\ 0 & 1  \end{array} \right]$ (if $\rho(\ek_{\mu})$ is of real type) or $\left[ \begin{array}{ll} i & 0 \\ 0 & i  \end{array} \right]$ (if $\rho(\ek_{\mu})$ is of imaginary type). Taking the Kronecker product of $\rho(\ek_{\mu})$ with $B$ on the right replaces each nonzero entry of $\rho(\ek_{\mu})$ with a block equivalent mod $\sim$ to either $\left[ \begin{array}{ll} 0 & i \\ i & 0  \end{array} \right]$ (if $\rho(\ek_{\mu})$ is of real type) or $\left[ \begin{array}{ll} 0 & 1 \\ 1 & 0  \end{array} \right]$ (if $\rho(\ek_{\mu})$ is of imaginary type). All of the zero entries of $\rho(\ek_{\mu})$ get replaced by $\left[ \begin{array}{ll} 0 & 0 \\ 0 & 0  \end{array} \right]$ in forming $\rho(\ekp_{\mu'})$, in either case.
	
	That is: in forming the matrix representation of $\ekp_{\emptyset \upSmallFrown \mu'} = \ekp_{\mu'}$ from that of $\ek_{\mu}$, all of the zero entries of $\rho(\ek_{\mu})$ get replaced by $\left[ \begin{array}{ll} 0 & 0 \\ 0 & 0  \end{array} \right]$, and all of the nonzero entries of $\rho(\ek_{\mu})$ get replaced by a block equivalent mod $\sim$ to:
	\[
	\left\{
	\begin{array}{l}
		\left[ \begin{array}{ll} 1 & 0 \\ 0 & 1  \end{array} \right], \textrm{ if $\rho(\ek_{\mu})$ is of real type and $|\mu|$ is even,} \\[7mm]
		\left[ \begin{array}{ll} i & 0 \\ 0 & i  \end{array} \right], \textrm{ if $\rho(\ek_{\mu})$ is of imaginary type and $|\mu|$ is even,} \\[7mm]
		\left[ \begin{array}{ll} 0 & i \\ i & 0  \end{array} \right], \textrm{ if $\rho(\ek_{\mu})$ is of real type and $|\mu|$ is odd,} \\[7mm]
		\left[ \begin{array}{ll} 0 & 1 \\ 1 & 0  \end{array} \right], \textrm{ if $\rho(\ek_{\mu})$ is of imaginary type and $|\mu|$ is odd.}
	\end{array}
	\right.
	\]
	
	Next consider $\ekp_{1 \upSmallFrown \mu'}$. By Lemma \ref{e mu hats split into 4 new e mu}, 
	\[
	\rho(\ekp_{1 \upSmallFrown \mu'}) = \rho(\ekp_1) \cdot  \rho(\ekp_{\mu'}) = \left\{
	\begin{array}{lll}
		\rho(\ekp_1) \cdot  (\rho(\ek_{\mu}) \otimes I_2) & \textrm{ if } & |\mu| \textrm{ is even } \\
		\rho(\ekp_1) \cdot (\rho(\ek_{\mu}) \otimes B) & \textrm{ if } & |\mu| \textrm{ is odd }
	\end{array}
	\right.
	\]
	By Lemma \ref{what multiplication by e1 does}, we have that in forming $\rho(\ekp_{1 \upSmallFrown \mu'})$ from $\rho(\ek_{\mu})$, all of the zero entries of $\rho(\ek_{\mu})$ get replaced by $\left[ \begin{array}{ll} 0 & 0 \\ 0 & 0  \end{array} \right]$, and all of the nonzero entries of $\rho(\ek_{\mu})$ get replaced by a block equivalent mod $\sim$ to:
	\[
	\left\{
	\begin{array}{l}
		\left[ \begin{array}{ll} i & 0 \\ 0 & i  \end{array} \right], \textrm{ if $\rho(\ek_{\mu})$ is of real type and $|\mu|$ is even,} \\[7mm]
		\left[ \begin{array}{ll} 1 & 0 \\ 0 & 1  \end{array} \right], \textrm{ if $\rho(\ek_{\mu})$ is of imaginary type and $|\mu|$ is even,} \\[7mm]
		\left[ \begin{array}{ll} 0 & 1 \\ 1 & 0  \end{array} \right], \textrm{ if $\rho(\ek_{\mu})$ is of real type and $|\mu|$ is odd,} \\[7mm]
		\left[ \begin{array}{ll} 0 & i \\ i & 0  \end{array} \right], \textrm{ if $\rho(\ek_{\mu})$ is of imaginary type and $|\mu|$ is odd.}
	\end{array}
	\right.
	\]
	
	Next consider $\ekp_{2 \upSmallFrown \mu'}$. By Lemma \ref{e mu hats split into 4 new e mu}, 
	\[
	\rho(\ekp_{2 \upSmallFrown \mu'}) = \rho(\ekp_2) \cdot  \rho(\ekp_{\mu'}) = \left\{
	\begin{array}{lll}
		\rho(\ekp_2) \cdot  (\rho(\ek_{\mu}) \otimes I_2) & \textrm{ if } & |\mu| \textrm{ is even } \\
		\rho(\ekp_2) \cdot (\rho(\ek_{\mu}) \otimes B) & \textrm{ if } & |\mu| \textrm{ is odd }
	\end{array}
	\right.
	\]
	By Lemma \ref{what e2 does}, we have that in forming $\rho(\ekp_{2 \upSmallFrown \mu'})$ from $\rho(\ek_{\mu})$, all of the zero entries of $\rho(\ek_{\mu})$ get replaced by $\left[ \begin{array}{ll} 0 & 0 \\ 0 & 0  \end{array} \right]$, and all of the nonzero entries of $\rho(\ek_{\mu})$ get replaced by a block equivalent mod $\sim$ to:
	\[
	\left\{
	\begin{array}{l}
		\left[ \begin{array}{ll} 0 & i \\ i & 0  \end{array} \right], \textrm{ if $\rho(\ek_{\mu})$ is of real type and $|\mu|$ is even,} \\[7mm]
		\left[ \begin{array}{ll} 0 & 1 \\ 1 & 0  \end{array} \right], \textrm{ if $\rho(\ek_{\mu})$ is of imaginary type and $|\mu|$ is even,} \\[7mm]
		\left[ \begin{array}{ll} 1 & 0 \\ 0 & 1  \end{array} \right], \textrm{ if $\rho(\ek_{\mu})$ is of real type and $|\mu|$ is odd,} \\[7mm]
		\left[ \begin{array}{ll} i & 0 \\ 0 & i  \end{array} \right], \textrm{ if $\rho(\ek_{\mu})$ is of imaginary type and $|\mu|$ is odd.}
	\end{array}
	\right.
	\]
	
	Next consider $\ekp_{12 \upSmallFrown \mu'}$. By Lemma \ref{e mu hats split into 4 new e mu}, 
	\[
	\rho(\ekp_{12 \upSmallFrown \mu'}) = \rho(\ekp_{12}) \cdot  \rho(\ekp_{\mu'}) = \left\{
	\begin{array}{lll}
		\rho(\ekp_{12}) \cdot  (\rho(\ek_{\mu}) \otimes I_2) & \textrm{ if } & |\mu| \textrm{ is even } \\
		\rho(\ekp_{12}) \cdot (\rho(\ek_{\mu}) \otimes B) & \textrm{ if } & |\mu| \textrm{ is odd }
	\end{array}
	\right.
	\]
	By Lemma \ref{what e12 does}, we have that in forming $\rho(\ekp_{12 \upSmallFrown \mu'})$ from $\rho(\ek_{\mu})$, all of the zero entries of $\rho(\ek_{\mu})$ get replaced by $\left[ \begin{array}{ll} 0 & 0 \\ 0 & 0  \end{array} \right]$, and all of the nonzero entries of $\rho(\ek_{\mu})$ get replaced by a block equivalent mod $\sim$ to:
	\[
	\left\{
	\begin{array}{l}
		\left[ \begin{array}{ll} 0 & 1 \\ 1 & 0  \end{array} \right], \textrm{ if $\rho(\ek_{\mu})$ is of real type and $|\mu|$ is even,} \\[7mm]
		\left[ \begin{array}{ll} 0 & i \\ i & 0  \end{array} \right], \textrm{ if $\rho(\ek_{\mu})$ is of imaginary type and $|\mu|$ is even,} \\[7mm]
		\left[ \begin{array}{ll} i & 0 \\ 0 & i  \end{array} \right], \textrm{ if $\rho(\ek_{\mu})$ is of real type and $|\mu|$ is odd,} \\[7mm]
		\left[ \begin{array}{ll} 1 & 0 \\ 0 & 1  \end{array} \right], \textrm{ if $\rho(\ek_{\mu})$ is of imaginary type and $|\mu|$ is odd.}
	\end{array}
	\right.
	\]
	Next, observe that if $\ek_{\mu}, \ek_{\eta} \in \gh_{2k}$ and their matrix representations have the same shape and the same type, then they act identically on elements of $J_2^{S_{\Delta_{2k}}}$ and so are equivalent mod $\sim$. 
	
	Now to prove the forward direction of the first claim: suppose $\ek_{\mu}, \ek_{\eta} \in \gh_{2k}$ and we have $\operatorname{Sh}(\rho(\ek_{\mu})) \neq \operatorname{Sh}(\rho(\ek_{\eta}))$. Then we can find some $1 \leq i, j \leq 2^k$ such that $\rho(\ek_{\mu})$ has a 0 in the $(i,j)$th entry, but $\rho(\ek_{\eta})$ has a nonzero value in the $(i,j)$th entry. Then we claim that $\ekp_{\emptyset \upSmallFrown \mu'}$ is not equivalent mod $\sim$ to any of $\ekp_{\emptyset \upSmallFrown \eta'}, \ekp_{1 \upSmallFrown \eta'}, \ekp_{2 \upSmallFrown \eta'}$, or $\ekp_{12 \upSmallFrown \eta'}$.
	
	To prove the claim: we note by Lemma \ref{e mu hats split into 4 new e mu} that in forming $\rho(\ekp_{\emptyset \upSmallFrown \mu'})$, we take the Kronecker product of $\rho(\ek_{\mu})$ with either $I_2$ or $B$. Also $\rho(\ekp_{\emptyset \upSmallFrown \eta'})$ is formed similarly from $\ek_\eta$. Using Lemmas \ref{what multiplication by e1 does}, \ref{what e2 does}, and \ref{what e12 does}, in forming any of  $\rho(\ekp_{1 \upSmallFrown \eta'}), \rho(\ekp_{2 \upSmallFrown \eta'})$, or $\rho(\ekp_{12 \upSmallFrown \eta'})$, we first take the Kronecker product with $I_2$ or $B$, and then scalar multiply by $i$ and/or interchange adjacent rows (that is, rows 1 and 2, 3 and 4, etc.). This means that $\rho(\ekp_{\emptyset \upSmallFrown \mu'})$ has a $2 \times 2$ block of zeros in the location corresponding to where  $\rho(\ekp_{\emptyset \upSmallFrown \eta'}), \rho(\ekp_{1 \upSmallFrown \eta'}), \rho(\ekp_{2 \upSmallFrown \eta'})$, and $\rho(\ekp_{12 \upSmallFrown \eta'})$ have a block equivalent mod $\sim$ to either $\left[ \begin{array}{ll} 1 & 0 \\ 0 & 1  \end{array} \right]$, $\left[ \begin{array}{ll} 0 & 1 \\ 1 & 0  \end{array} \right]$, $\left[ \begin{array}{ll} i & 0 \\ 0 & i  \end{array} \right]$, or $\left[ \begin{array}{ll} 0 & i \\ i & 0  \end{array} \right]$. Then $\operatorname{Sh}(\rho(\ekp_{\emptyset \upSmallFrown \mu'}))$ is different from any of the shapes of $\rho(\ekp_{\emptyset \upSmallFrown \eta'})$, $\rho(\ekp_{1 \upSmallFrown \eta'})$, $\rho(\ekp_{2 \upSmallFrown \eta'})$, and $\rho(\ekp_{12 \upSmallFrown \eta'})$, so that $\ekp_{\emptyset \upSmallFrown \mu'}$ is not  equivalent mod $\sim$ to any of the four elements arising from $\ek_{\eta}$.
	
	For the backward direction of the first claim: suppose $\ek_{\mu}, \ek_{\eta} \in \gh_{2k}$ with $\operatorname{Sh}(\rho(\ek_{\mu})) = \operatorname{Sh}(\rho(\ek_{\eta}))$. If $\rho(\ek_{\mu})$ and $\rho(\ek_{\eta})$ had the same type (real or imaginary) and $|\mu|$ and $|\eta|$ had the same parity (even or odd), then we would have $[\ek_{\mu}] = [\ek_{\eta}]$, so that the conclusion would hold by Lemma \ref{e mu hats split into 4 new e mu}. Then we need only consider cases where the types (real or imaginary) of $\rho(\ekp_{\mu})$ and $\rho(\ekp_{\eta})$ are different, and/or where the parities of $|\mu|$ and $|\eta|$ are different.
	
	\textit{Case 1}: Suppose $\rho(\ek_{\mu})$ is of real type, $\rho(\ek_{\eta})$ is of  imaginary type, and both $|\mu|$ and $|\eta|$ are even. We show that in this case 
\[
 \left[ \rho(\ekp_{\emptyset \upSmallFrown\mu'}) \right] = \left[ \rho(\ekp_{1 \upSmallFrown\eta'}) \right], \quad \left[ \rho(\ekp_{1 \upSmallFrown\mu'}) \right] = \left[ \rho(\ekp_{\emptyset \upSmallFrown\eta'}) \right],  \]   
 \[
 \left[ \rho(\ekp_{2 \upSmallFrown\mu'}) \right] = \left[ \rho(\ekp_{12 \upSmallFrown\eta'}) \right], \quad \textrm{ and } \quad  \left[ \rho(\ekp_{12 \upSmallFrown\mu'}) \right] = \left[ \rho(\ekp_{2 \upSmallFrown\eta'}) \right].
 \]
We use the observations about the matrix representations of these elements of $\gh_{2k+2}$ that we made at the beginning of the present proof.
	
	We have that $\rho(\ekp_{\emptyset \upSmallFrown \mu'})$ is a matrix in which all zero entries of $\rho(\ek_{\mu})$ have been replaced by $\left[ \begin{array}{ll} 0 & 0 \\ 0 & 0  \end{array} \right]$, and all nonzero entries have been replaced by a block equivalent mod $\sim$ to $\left[ \begin{array}{ll} 1 & 0 \\ 0 & 1  \end{array} \right]$. Also, $\rho(\ekp_{1 \upSmallFrown \eta'})$ is a matrix in which all nonzero entries of $\rho(\ek_{\eta})$ have been replaced by a block equivalent mod $\sim$ to $\left[ \begin{array}{ll} 1 & 0 \\ 0 & 1  \end{array} \right]$. But since $\operatorname{Sh}(\rho(\ek_{\mu})) = \operatorname{Sh}(\rho(\ek_{\eta}))$, the nonzero entries of $\rho(\ek_{\mu})$ and $\rho(\ek_{\eta})$ are in the same places, so that in fact  $\rho(\ekp_{1 \upSmallFrown \eta'})$ is a matrix in which all nonzero entries of $\rho(\ek_{\mu})$ have been replaced by a block equivalent mod $\sim$ to $\left[ \begin{array}{ll} 1 & 0 \\ 0 & 1  \end{array} \right]$. Therefore $\ekp_{\emptyset \upSmallFrown \mu'} \sim \ekp_{1 \upSmallFrown \eta'}$, and so $\left[\ekp_{\emptyset \upSmallFrown \mu'}\right] = \left[\ekp_{1 \upSmallFrown \eta'}\right]$.
	
	$\rho(\ekp_{1 \upSmallFrown\mu'})$ is a matrix in which all zeros of $\rho(\ek_{\mu})$ have been replaced by $2 \times 2$ zero blocks and each nonzero entry has been replaced by a block equivalent mod $\sim$ to $\left[ \begin{array}{ll} i & 0 \\ 0 & i  \end{array} \right]$. Also, $\rho(\ekp_{\emptyset \upSmallFrown\eta'})$ is a matrix in which zeros in $\rho(\ek_{\eta})$ have been replaced by $2 \times 2$ zero blocks and each nonzero entry has been replaced by a block equivalent mod $\sim$ to $\left[ \begin{array}{ll} i & 0 \\ 0 & i  \end{array} \right]$. Since $\operatorname{Sh}(\rho(\ek_{\mu})) = \operatorname{Sh}(\rho(\ek_{\eta}))$, we have $\ekp_{1 \upSmallFrown \mu'} \sim \ekp_{\emptyset \upSmallFrown \eta'}$, so that $\left[ \ekp_{1 \upSmallFrown \mu'} \right] = \left[ \ekp_{\emptyset \upSmallFrown \eta'} \right]$.
	
	$\rho(\ekp_{2 \upSmallFrown\mu'})$ is a matrix in which all zeros of $\rho(\ek_{\mu})$ have been replaced by $2 \times 2$ zero blocks and each nonzero entry has been replaced by a block equivalent mod $\sim$ to $\left[ \begin{array}{ll} 0 & i \\ i & 0  \end{array} \right]$. Also, $\rho(\ekp_{12 \upSmallFrown\eta'}\rho)$ is a matrix in which zeros in $\rho(\ek_{\eta})$ have been replaced by $2 \times 2$ zero blocks and each nonzero entry in $\rho(\ek_{\eta})$ has been replaced by a block equivalent mod $\sim$ to $\left[ \begin{array}{ll} 0 & i \\ i & 0  \end{array} \right]$. Since $\operatorname{Sh}(\rho(\ek_{\mu})) = \operatorname{Sh}(\rho(\ek_{\eta}))$, we have $\ekp_{2 \upSmallFrown \mu'} \sim \ekp_{12 \upSmallFrown \eta'}$, so that $\left[ \ekp_{2 \upSmallFrown \mu'} \right] = \left[ \ekp_{12 \upSmallFrown \eta'} \right]$.
	
	Finally, $\rho(\ekp_{12 \upSmallFrown\mu'})$ is a matrix in which all zeros of $\rho(\ek_{\mu})$ have been replaced by $2 \times 2$ zero blocks and each nonzero entry has been replaced by a block equivalent mod $\sim$ to $\left[ \begin{array}{ll} 0 & 1 \\ 1 & 0  \end{array} \right]$. Also, $\rho(\ekp_{2 \upSmallFrown\eta'})$ is a matrix in which zeros in $\rho(\ek_{\eta})$ have been replaced by $2 \times 2$ zero blocks and each nonzero entry in $\rho(\ek_{\eta})$ has been replaced by a block equivalent mod $\sim$ to $\left[ \begin{array}{ll} 0 & 1 \\ 1 & 0  \end{array} \right]$. Since $\operatorname{Sh}(\rho(\ek_{\mu})) = \operatorname{Sh}(\rho(\ek_{\eta}))$, we have $\ekp_{12 \upSmallFrown \mu'} \sim \ekp_{2 \upSmallFrown \eta'}$, so that $\left[ \ekp_{12 \upSmallFrown \mu'} \right] = \left[ \ekp_{2 \upSmallFrown \eta'} \right]$.
	
	This completes Case 1. 
\

\
	
	Below we show proofs of Cases 2 through 6, which are similar to Case 1 and which also follow from the characterization of representations of the elements $\ekp_{\emptyset \upSmallFrown \mu'}$, $\ekp_{1 \upSmallFrown \mu'}$, \\ $\ekp_{2 \upSmallFrown \mu'},$ and $\ekp_{12 \upSmallFrown \mu'}$ found at the beginning of the proof above.
	
	\textit{Case 2}: Suppose $\rho(\ek_{\mu})$ has real type, $\rho(\ek_{\eta})$ has imaginary type, $|\mu|$ is even, and $|\eta|$ is odd. In this case one can show:
	
	(i) $\ekp_{\emptyset \upSmallFrown \mu'} \sim \ekp_{12 \upSmallFrown \eta'}$, and the representations of both of these generators are the result of replacing each nonzero entry of $\rho(\ek_{\mu})$ with a block equivalent mod $\sim$ to $\left[ \begin{array}{ll} 1 & 0 \\ 0 & 1  \end{array} \right]$.
	
	(ii) $\ekp_{1 \upSmallFrown \mu'} \sim \ekp_{2 \upSmallFrown \eta'}$, and the representations of both of these generators are the result of replacing each nonzero entry of $\rho(\ek_{\mu})$ with a block equivalent mod $\sim$ to $\left[ \begin{array}{ll} i & 0 \\ 0 & i  \end{array} \right]$.   
	
	(iii) $\ekp_{2 \upSmallFrown \mu'} \sim \ekp_{1 \upSmallFrown \eta'}$, and the representations of both of these generators are the result of replacing each nonzero entry of $\rho(\ek_{\mu})$ with a block equivalent mod $\sim$ to $\left[
	\begin{array}{ll} 0 & i \\ i & 0
	\end{array}
	\right].$
	
	(iv) $\ekp_{12 \upSmallFrown \mu'} \sim \ekp_{\emptyset \upSmallFrown \eta'}$, and the representations of both of these generators are the result of replacing each nonzero entry of $\rho(\ek_{\mu})$ with a block equivalent mod $\sim$ to $\left[
	\begin{array}{ll} 0 & 1 \\ 1 & 0
	\end{array}
	\right].$
	
	\textit{Case 3}: Suppose $\rho(\ek_{\mu})$ has real type, $\rho(\ek_{\eta})$ has imaginary type, $|\mu|$ is odd, and $|\eta|$ is even. In this case one can show:
	
	(i) $\ekp_{\emptyset \upSmallFrown \mu'} \sim \ekp_{12 \upSmallFrown \eta'}$, and the representations of both of these generators are the result of replacing each nonzero entry of $\rho(\ek_{\mu})$ with a block equivalent mod $\sim$ to $\left[ \begin{array}{ll} 0 & i \\ i & 0  \end{array} \right]$.
	
	(ii) $\ekp_{1 \upSmallFrown \mu'} \sim \ekp_{2 \upSmallFrown \eta'}$, and the representations of both of these generators are the result of replacing each nonzero entry of $\rho(\ek_{\mu})$ with a block equivalent mod $\sim$ to $\left[ \begin{array}{ll} 0 & 1 \\ 1 & 0  \end{array} \right]$.   
	
	(iii) $\ekp_{2 \upSmallFrown \mu'} \sim \ekp_{1 \upSmallFrown \eta'}$, and the representations of both of these generators are the result of replacing each nonzero entry of $\rho(\ek_{\mu})$ with a block equivalent mod $\sim$ to $\left[
	\begin{array}{ll} 1 & 0 \\ 0 & 1
	\end{array}
	\right].$
	
	(iv) $\ekp_{12 \upSmallFrown \mu'} \sim \ekp_{\emptyset \upSmallFrown \eta'}$, and the representations of both of these generators are the result of replacing each nonzero entry of $\rho(\ek_{\mu})$ with a block equivalent mod $\sim$ to $\left[
	\begin{array}{ll} i & 0 \\ 0 & i
	\end{array}
	\right].$
	
	\textit{Case 4}: Suppose $\rho(\ek_{\mu})$ has real type, $\rho(\ek_{\eta})$ has imaginary type, and both $|\mu|$ and $|\eta|$ are odd. In this case one can show:
	
	(i) $\ekp_{\emptyset \upSmallFrown \mu'} \sim \ekp_{1 \upSmallFrown \eta'}$, and the representations of both of these generators are the result of replacing each nonzero entry of $\rho(\ek_{\mu})$ with a block equivalent mod $\sim$ to $\left[ \begin{array}{ll} 0 & i \\ i & 0  \end{array} \right]$.
	
	(ii) $\ekp_{1 \upSmallFrown \mu'} \sim \ekp_{\emptyset \upSmallFrown \eta'}$, and the representations of both of these generators are the result of replacing each nonzero entry of $\rho(\ek_{\mu})$ with a block equivalent mod $\sim$ to $\left[ \begin{array}{ll} 0 & 1 \\ 1 & 0  \end{array} \right]$.   
	
	(iii) $\ekp_{2 \upSmallFrown \mu'} \sim \ekp_{12 \upSmallFrown \eta'}$, and the representations of both of these generators are the result of replacing each nonzero entry of $\rho(\ek_{\mu})$ with a block equivalent mod $\sim$ to $\left[
	\begin{array}{ll} 1 & 0 \\ 0 & 1
	\end{array}
	\right].$
	
	(iv) $\ekp_{12 \upSmallFrown \mu'} \sim \ekp_{2 \upSmallFrown \eta'}$, and the representations of both of these generators are the result of replacing each nonzero entry of $\rho(\ek_{\mu})$ with a block equivalent mod $\sim$ to $\left[
	\begin{array}{ll} i & 0 \\ 0 & i
	\end{array}
	\right].$
	
	\textit{Case 5}: Suppose both $\rho(\ek_{\mu})$ and $\rho(\ek_{\eta})$ have real type, $|\mu|$ is even, and $|\eta|$ odd. In this case one can show:
	
	(i) $\ekp_{\emptyset \upSmallFrown \mu'} \sim \ekp_{2 \upSmallFrown \eta'}$, and the representations of both of these generators are the result of replacing each nonzero entry of $\rho(\ek_{\mu})$ with a block equivalent mod $\sim$ to $\left[ \begin{array}{ll} 1 & 0 \\ 0 & 1  \end{array} \right]$.
	
	(ii) $\ekp_{1 \upSmallFrown \mu'} \sim \ekp_{12 \upSmallFrown \eta'}$, and the representations of both of these generators are the result of replacing each nonzero entry of $\rho(\ek_{\mu})$ with a block equivalent mod $\sim$ to $\left[ \begin{array}{ll} i & 0 \\ 0 & i  \end{array} \right]$.   
	
	(iii) $\ekp_{2 \upSmallFrown \mu'} \sim \ekp_{\emptyset \upSmallFrown \eta'}$, and the representations of both of these generators are the result of replacing each nonzero entry of $\rho(\ek_{\mu})$ with a block equivalent mod $\sim$ to $\left[
	\begin{array}{ll} 0 & i \\ i & 0
	\end{array}
	\right].$
	
	(iv) $\ekp_{12 \upSmallFrown \mu'} \sim \ekp_{1 \upSmallFrown \eta'}$, and the representations of both of these generators are the result of replacing each nonzero entry of $\rho(\ek_{\mu})$ with a block equivalent mod $\sim$ to $\left[
	\begin{array}{ll} 0 & 1 \\ 1 & 0
	\end{array}
	\right].$
	
	\textit{Case 6}: Suppose both $\rho(\ek_{\mu})$ and $\rho(\ek_{\eta})$ have imaginary type, $|\mu|$ is even, and $|\eta|$ odd. In this case one can show:
	
	(i) $\ekp_{\emptyset \upSmallFrown \mu'} \sim \ekp_{2 \upSmallFrown \eta'}$, and the representations of both of these generators are the result of replacing each nonzero entry of $\rho(\ek_{\mu})$ with a block equivalent mod $\sim$ to $\left[ \begin{array}{ll} i & 0 \\ 0 & i  \end{array} \right]$.
	
	(ii) $\ekp_{1 \upSmallFrown \mu'} \sim \ekp_{12 \upSmallFrown \eta'}$, and the representations of both of these generators are the result of replacing each nonzero entry of $\rho(\ek_{\mu})$ with a block equivalent mod $\sim$ to $\left[ \begin{array}{ll} 1 & 0 \\ 0 & 1  \end{array} \right]$.   
	
	(iii) $\ekp_{2 \upSmallFrown \mu'} \sim \ekp_{\emptyset \upSmallFrown \eta'}$, and the representations of both of these generators are the result of replacing each nonzero entry of $\rho(\ek_{\mu})$ with a block equivalent mod $\sim$ to $\left[
	\begin{array}{ll} 0 & 1 \\ 1 & 0
	\end{array}
	\right].$
	
	(iv) $\ekp_{12 \upSmallFrown \mu'} \sim \ekp_{1 \upSmallFrown \eta'}$, and the representations of both of these generators are the result of replacing each nonzero entry of $\rho(\ek_{\mu})$ with a block equivalent mod $\sim$ to $\left[
	\begin{array}{ll} 0 & i \\ i & 0
	\end{array}
	\right].$
	
	Thus in all cases where  $\operatorname{Sh}(\rho(\ek_{\mu})) = \operatorname{Sh}(\rho(\ek_{\eta}))$, we have
	\[
	\begin{array}{ll}
		&\left\{ \left[ \ekp_{\emptyset \upSmallFrown \mu'} \right], \left[ \ekp_{1 \upSmallFrown \mu'} \right], \left[ \ekp_{2 \upSmallFrown \mu'} \right], \left[ \ekp_{12 \upSmallFrown \mu'} \right] \right\} \\[15pt]
		
		= & \left\{\left[ \ekp_{\emptyset \upSmallFrown \eta'} \right], \left[ \ekp_{1 \upSmallFrown \eta'} \right], \left[ \ekp_{2 \upSmallFrown \eta'} \right], \left[ \ekp_{12 \upSmallFrown \eta'} \right] \right\}.
	\end{array}
	\] 
	We now show that the sets $\mathcal{J}_\mu := \left\{ \left[ \ekp_{\emptyset \upSmallFrown \mu'} \right], \left[ \ekp_{1 \upSmallFrown \mu'} \right], \left[ \ekp_{2 \upSmallFrown \mu'} \right], \left[ \ekp_{12 \upSmallFrown \mu'} \right] \right\}$ and \\ $\mathcal{J}_\eta := \left\{\left[ \ekp_{\emptyset \upSmallFrown \eta'} \right], \left[ \ekp_{1 \upSmallFrown \eta'} \right], \left[ \ekp_{2 \upSmallFrown \eta'} \right], \left[ \ekp_{12 \upSmallFrown \eta'} \right] \right\}$ are disjoint if we have $\operatorname{Sh}(\rho(\ek_{\mu})) \neq \operatorname{Sh}(\rho(\ek_{\eta}))$. 
	
	Suppose $\ek_{\mu}, \ek_{\eta} \in \gh_{2k}$ and  $\operatorname{Sh}(\rho(\ek_{\mu})) \neq \operatorname{Sh}(\rho(\ek_{\eta}))$, and suppose by way of contradiction that $\mathcal{J}_\mu \cap \mathcal{J}_\eta \neq \emptyset$. We show that a contradiction arises if $\ekp_{\emptyset \upSmallFrown \mu'} \sim \ekp_{12 \upSmallFrown \eta'}$; the other cases are similar.
	\[
	\begin{array}{llll}
		& \ekp_{\emptyset \upSmallFrown \mu'} & \sim & \ekp_{12 \upSmallFrown \eta'} \\[3mm] \implies & \ekp_{\mu'} & \sim & \ekp_{12} \ekp_{\eta'} \\[3mm]
		\implies & \ekp_1 \ekp_{\mu'} & \sim & \ekp_1 \ekp_{12} \ekp_{\eta'} \textrm{ (multiplying on both sides by $\ekp_1$) } \\[3mm]
		\implies & \ekp_1 \ekp_{\mu'} & \sim & \ekp_1 \ekp_1 \ekp_2 \ekp_{\eta'} \\[3mm]
		\implies & \ekp_1 \ekp_{\mu'} & \sim & \ekp_2 \ekp_{\eta'} \\[3mm]
		\implies & \ekp_{1 \upSmallFrown \mu'} & \sim & \ekp_{2 \upSmallFrown \eta'}; \\[3mm]
	\end{array}
	\]
	also 
	\[
	\begin{array}{llll}
		& \ekp_{\mu'} & \sim & \ekp_{12} \ekp_{\eta'} \\[3mm]
		\implies & \ekp_{12} \ekp_{\mu'} & \sim & \ekp_{12} \ekp_{12} \ekp_{\eta'} \textrm{ (multiplying on both sides by $\ekp_{12}$) } \\[3mm]
		\implies & \ekp_{12} \ekp_{\mu'} & \sim & \ekp_{\eta'} \\[3mm]
		\implies & \ekp_{12 \upSmallFrown \mu'} & \sim & \ekp_{\eta'}; \\[3mm]
	\end{array}
	\]
	and finally 
	\[
	\begin{array}{llll}
		& \ekp_{\mu'} & \sim & \ekp_{12} \ekp_{\eta'} \\[3mm]
		\implies & \ekp_2 \ekp_{\mu'} & \sim & \ekp_2 \ekp_{12} \ekp_{\eta'} \textrm{ (multiplying on both sides by $\ekp_2$) } \\[3mm]
		\implies & \ekp_2 \ekp_{\mu'} & \sim & \ekp_1 \ekp_{\eta'} \\[3mm]
		\implies & \ekp_{2 \upSmallFrown \mu'} & \sim & \ekp_{1 \upSmallFrown \eta'}. \\[3mm]
	\end{array}
	\]
	But then we would have $\mathcal{J} = \mathcal{J'}$, which we have already shown is not the case when $\operatorname{Sh}(\rho(\ek_{\mu})) \neq \operatorname{Sh}(\rho(\ek_{\eta}))$.
\end{proof}

We are now ready to prove the structure theorem for Clifford multiplication on 2-torsion points on the Dirac spinor Abelian variety.

\begin{thm}\label{that corollary} The following hold, for all $k \in \mathbb{N}, \, k \geq 1$:
	\begin{enumerate}
		\item Each equivalence class mod $\sim$ in $\gh_{2k}$ has the same size: $2^{k-1}$. 
		\item There are a total of $2^{k+1}$ equivalence classes mod $\sim$: $2^k$ real equivalence classes and $2^k$ imaginary equivalence classes. 
		\item There are $2^k$  distinct shapes occurring among the equivalence classes mod $\sim$ in $\gh_{2k}$, and each one occurs in two classes: an imaginary class and a real class.
		\item There are exactly as many $\ek_{\mu}$ in a given class with even-length $\mu$ as with odd-length $\mu$. That is, for each equivalence class $[\ek_{\mu}]$ mod $\sim$ in $\gh_{2k}$, $|\{\ek_{\eta} \in [\ek_{\mu}]: |\eta| \textrm{ is even}\}| = |\{\ek_{\eta} \in [\ek_{\mu}]: |\eta| \textrm{ is odd}\}|$.
	\end{enumerate}
\end{thm}

\begin{proof}
	We prove the theorem by induction on $k$. By inspection, it is true for $k=1$. Suppose properties (1) through (4) above hold for some $k \geq 1$. Consider an equivalence class $E=[\ekp_\mu] \subseteq \gh_{2k+2}$, for some representative $\ekp_\mu \in \gh_{2k+2}$. 
	
	By Theorem \ref{same shapes give rise to same equivalence classes}, two elements of $\gh_{2k}$ give rise to the same four equivalence classes mod $\sim$ if and only if they have the same shape. So, to count the number of generators in $E$, we count the number of $\ek_{\eta} \in \gh_{2k}$ whose representations have the same shape as that of $\ek_{\hat{\mu}}$, where (by Lemma \ref{e mu hats split into 4 new e mu})  $\ek_{\hat{\mu}}$ is the element of $\gh_{2k}$ such that $\ekp_\mu$ is either $\ekp_{\emptyset \upSmallFrown \hat{\mu}'}$, $\ekp_{1 \upSmallFrown \hat{\mu}'}$, $\ekp_{2 \upSmallFrown \hat{\mu}'}$, or $\ekp_{12 \upSmallFrown \hat{\mu}'}$ (where $\hat{\mu}' = \hat{\mu} + 2$). By conditions (1) and (3) of the induction hypotheses, the number of $\ek_{\eta} \in \gh_{2k}$ for which $\rho(\ek_{\eta})$ has the same shape as $\rho(\ek_{\mu})$ is $2^{k-1} + 2^{k-1}$: $2^{k-1}$ generators that are in the same class mod $\sim$ as $e_{\mu}$ (and so whose representations have the same type of nonzero entries (purely real or purely imaginary)), and another $2^{k-1}$ generators in a class $E'$ mod $\sim$ in which every $e$ has a representation with the same shape as $\rho(e_{\mu})$ but the opposite kind of nonzero entries (purely real rather than purely imaginary, or vice versa). 

	So, we have that the size of the equivalence class $E \subseteq \gh_{2k}$ is $2^{k-1} + 2^{k-1} = 2^k = 2^{(k+1)-1}$. Thus condition (1) continues to hold in $\gh_{2k+2}$.
	
	By condition (2) of the induction hypothesis, there are $2^{k+1}$ equivalence classes mod $\sim$ in $\gh_{2k}$. By Lemma \ref{2 shapes, in 2 flavors each} and Theorem \ref{same shapes give rise to same equivalence classes}, from each equivalence class mod $\sim$ in $\gh_{2k}$ come four new equivalence classes -- a real and an  imaginary class for each of two new shapes. However, the real and imaginary classes of each of the two new shapes will generate the same new equivalence classes mod $\sim$ in $\gh_{2k+2}$. Thus to count the number of equivalence classes mod $\sim$ in $\gh_{2k+2}$, we take 4 times the number of shapes occurring in representations of generators in $\gh_{2k}$, which by condition (3) of the induction hypothesis was $2^k$; so the number of equivalence classes mod $\sim$ in $\gh_{2k+2}$ is $4 \cdot 2^k = 2^2 \cdot 2^k = 2^{2+k} = 2^{(k+1)+1}$. Thus condition (2) continues to hold in $\gh_{2k+2}$.
	
	Suppose $E = [\ek_{\mu}]$ and $E'= [\ek_{\eta}]$ are, respectively, the real and imaginary equivalence classes in $\gh_{2k}$ with some shape $P = \operatorname{Sh}(\rho(\ek_{\mu})) = \operatorname{Sh}(\rho(\ek_{\eta}))$. By Lemmas \ref{2 shapes, in 2 flavors each} and \ref{same shapes give rise to same equivalence classes}, the same four new equivalence classes, representing two new shapes (one real and one imaginary class for each), are obtained from $E$ and $E'$. Thus to count the number of shapes among equivalence classes in $\gh_{2k+2}$, we take 2 times the number of shapes occurring among classes in $\gh_{2k}$: $2 \cdot 2^k = 2^{k+1}$. Thus condition (3) continues to hold in $\gh_{2k+2}$.
	
	Finally, condition (4) continues to hold: let $\ek_{\mu} \in \gh_{2k}$. If $|\mu|$ was even, then $|\emptyset \upSmallFrown {\mu}'|$ and $|12 \upSmallFrown {\mu}'|$ are even while $|1 \upSmallFrown {\mu}'|$ and $|2 \upSmallFrown {\mu}'|$ are odd; while if $|\mu|$ was odd, then $|1 \upSmallFrown {\mu}'|$ and $|2 \upSmallFrown {\mu}'|$ are even while $|\emptyset \upSmallFrown {\mu}'|$ and $|12 \upSmallFrown {\mu}'|$ are odd.
	
\end{proof}

With this structure theorem, we know exactly what Clifford multiplication on 2-torsion points looks like. Now that we have an understanding of the unique operators on $J_2^{S_{\Delta_{2k}}}$, we can proceed to view $\hat{\Gamma}_{2k}$ actions on the 2-torsion points of the Dirac spinor Abelian varieties in an entirely different manner, abandoning the matrix multiplication action and replacing it with permutations of order $2$.

\section{Representing the actions of our Clifford  generators on $J_2^{S_{\Delta_{2k}}}$ as permutations of order $2$}\label{section: Clifford permutations}

In this section, we will again represent our 2-torsion points $\vec{v} \in J_2^{S_{\Delta_{2k}}}$ with the notation  $$T_{a_1, \ldots ,a_{2^{k}}}= \vec{v} = \left( \begin{array}{c} v_{a_1}\\ \vdots \\v_{a_{2^{k}}}\ \end{array} \right),$$ where $a_l \in \{0, 1, 2, 3\}$ for $1 \leq l \leq 2^k$. (Recall that we have denoted $v_0=0$, $v_1=\frac{1}{2}$, $v_2 = \frac{i}{2}$, and $v_3 = \frac{1+i}{2}$.) We define the group action of the symmetric group $\mathcal{S}_{2^{k}}$ on $J_2^{S_{\Delta_{2k}}}$ in the usual manner, by permuting the components of the vector $T_{a_1, \ldots, a_{2^k}}$ as indicated by an element of the symmetric group. Since we will be discussing permutations of order 2, such an action takes the following form:

\[
\sigma\cdot T_{a_1, \ldots ,a_{2^{k}}}=T_{a_{\sigma(1)}, \ldots , a_{\sigma(2^k)}}=\left( \begin{array}{c} v_{a_{\sigma(1)}}\\ \vdots \\v_{a_{\sigma(2^{k})}}\ \end{array} \right) 
\]

\noindent for any $\sigma\in \mathcal{S}_{2^{k}}$ of order 2. As we know, there are a total of $2^{k+1}$ unique Clifford actions on our 2-torsion points, and each non-identity action is an involution on $J_2^{S_{\Delta_{2k}}}$.  We can now define the actions on $J_2^{S_{\Delta_{2k}}}$ in terms of permutations on $\{1, \ldots, 2^k\}$. 

\begin{defn}\label{switch permutations}
We define the \textbf{switch permutations} $A_{2^{j}}\in \mathcal{S}_{2^{k}}$, for $j=0,\ldots,k-1$, as the  permutations that pairwise interchange blocks of size $2^{j}$ in $\{1, \ldots, 2^k\}$. More explicitly: fix $j \in \{0, \ldots, k-1\}$. Write the ordered set $\langle 1, \ldots, 2^k\rangle$ as the disjoint union of $2^{k-j}$ blocks of size $2^j$, preserving the usual order:
\[
\begin{array}{lrcl}
\textrm{First block:} & X_1 &= & \langle1, \ldots, 1 \cdot 2^j\rangle \\
\textrm{Second block:} & X_2 & = & \langle1 \cdot 2^j + 1, \ldots, 2 \cdot 2^j\rangle \\
\textrm{Third block:} & X_3 & = & \langle 2 \cdot 2^j + 1, \ldots, 3 \cdot 2^j\rangle \\
\ & \  & \vdots & \ \\
(2^{k-j}-1)\textrm{th block:} & X_{2^{k-j}-1} & = & \langle (2^{k-j}-2) \cdot 2^j + 1, \ldots, (2^{k-j}-1) \cdot 2^j\rangle \\
(2^{k-j})\textrm{th block:} & X_{2^{k-j}} & = & \langle (2^{k-j}-1) \cdot 2^j + 1, \ldots, (2^{k-j}) \cdot 2^j\rangle 
\end{array}
\]
The permutation $A_{2^j}$ interchanges the first and second blocks, the third and fourth blocks, \ldots, and the $(2^{k-j}-1)$th and $(2^{k-j})$th blocks of size $2^j$, while preserving the existing order of the numbers within each block. 
\end{defn}

We quickly remark that all of the switch permutations $A_{2^{j}}\in\mathcal{S}_{2^{k}}$ are of order 2, and that they are even permutations and hence belong to the alternating subgroup $Alt(2^{k})$ of the symmetry group $\mathcal{S}_{2^k}$. 

\begin{exmp}
The switch permutation $A_{2^1}$ on the ordered set $\langle 1, \ldots, 8\rangle$ pairwise interchanges blocks of size $2$. Thus it maps  $\langle 1, 2, 3, 4, 5, 6, 7, 8 \rangle$ to $\langle 3, 4, 1, 2, 7, 8, 5, 6  \rangle$; that is, $A_{2^1}$ acting on $\langle 1, \ldots, 8\rangle$ is the even permutation $(1 \ 3)(2 \ 4)(5 \ 7)(6 \ 8)$.
\end{exmp}

We now connect our switch permutations $A_{2^{j}}\in Alt(2^{k})$ and the identity permutation $(1)$ with all of our $2^{k+1}$ unique Clifford actions given by generators of the standard basis for the Clifford algebra $Cl(\C^{2k})$ acting on $J_2^{S_{\Delta_{2k}}}$ established in Section \ref{Properties of the multiplicative group ...}.

\begin{defn}
Let $\langle i \rangle=\{1, -1, i, -i\}$ be the group of fourth roots of unity. Define the $\langle i \rangle \times \mathcal{S}_{2^k}$  action on $J_2^{S_{\Delta_{2k}}}$ as follows: $(z, \sigma)\cdot T_{a_1, \ldots ,a_{2^{k}}}=z\cdot (\sigma \cdot T_{a_1, \ldots ,a_{2^{k}}})$; that is, the scalar multiplication of the vector $\sigma \cdot T_{a_1, \ldots ,a_{2^{k}}}$ by the fourth root of unity $z$.
\end{defn}

With this new adapted action on $J_2^{S_{\Delta_{2k}}}$, we prove that we can represent all unique Clifford actions on the 2-torsion points in terms of the switch permutations $A_{2^{j}}\in Alt(2^{k})\subset \mathcal{S}_{2^{k}}$ and scalar multiplication by $i\in\C$.

\begin{defn}\label{S maps}
For $k \in \mathbb{N},\, k \geq 1$, we define two kinds of actions, denoted $S_{1, \ldots, k}$ and $i \cdot S_{1, \ldots, k}$, on the $2$-torsion points $J_2^{S_{\Delta_{2k}}}$ as follows: $S_{1, \ldots, k} := S_k \circ \cdots \circ S_1$ where $S_l \in \{(1), A_{2^{l - 1}}\}$ for $1 \leq l \leq k$. That is, a permutation $S_{1, \ldots, k}$ acts on $\{1, \ldots, 2^k\}$ (and thus on the entries of a $2$-torsion point $\vec{v} \in J_2^{S_{\Delta_{2k}}}$) by either pairwise switching or leaving fixed all blocks of size $2^{l-1}$ for each $1 \leq l \leq k$.) We denote by $i \cdot S_{1, \ldots, k}$ the action of one of these $S_{1, \ldots, k}$ followed by a scalar multiplication by $i$.
\end{defn}

Observe that since there are two choices for each of the $k$ components of an action $S_{1, \ldots, k}$, there are $2^k$ such maps. Similarly, there are $2^k$ maps of the form $i \cdot S_{1, \ldots, k}$. 

\begin{thm}\label{representing Clifford multiplication via permutations}
Let $k \in \mathbb{N},\ k \geq 1$, and let $\ek_\mu \in \gh_{2k}$ be a generator. Then the action of $\ek_\mu$ on an element $\vec{v} \in J_2^{S_{\Delta_{2k}}}$ is identical to that of some action $S_{1,\ldots, k}$ or $i \cdot S_{1,\ldots, k}$ on the entries of $\vec{v}$.

Thus the set $\{[\ek_\mu]:\ \ek_\mu \in \gh_{2k}\}$ of equivalence classes mod $\sim$ is isomorphic, as a set of actions on $J_2^{S_{\Delta_{2k}}}$, to the set $\{S_{1, \ldots, k},\ i \cdot S_{1, \ldots, k}: S_{l} \in \{A_{2^{l-1}}, (1)\} \textrm{ for } 1 \leq l \leq k\}$.

\end{thm}

\begin{proof}

Say that a $2^k \times 2^k$ permutation matrix $P$ has Property ($\star$) if: 
\begin{itemize}
	\item $P$ is of form $\left[\begin{array}{c|c}
		C_{k-1} & 0 \\
		\hline 
		0 & C_{k-1}
	\end{array}\right]$ or $\left[\begin{array}{c|c}
	0 & C_{k-1} \\
	\hline 
	C_{k-1} & 0
	\end{array}\right]$, where $C_{k-1}$ is a $2^{k-1} \times 2^{k-1}$ permutation matrix;
	\item  $C_{k-1}$ is of form $\left[\begin{array}{c|c}
		C_{k-2} & 0 \\
		\hline 
		0 & C_{k-2}
	\end{array}\right]$ or $\left[\begin{array}{c|c}
		0 & C_{k-2} \\
		\hline 
		C_{k-2} & 0
	\end{array}\right],$ where $C_{k-2}$ is a $2^{k-2} \times 2^{k-2}$ permutation matrix;
	\[\vdots\]
	\item and $C_1 = C_{k-(k-1)}$ is one of the $2 \times 2$ permutation matrices $\left[\begin{array}{cc}
		1 & 0 \\
		0 & 1
	\end{array}\right]$ or $\left[\begin{array}{cc}
	0 & 1 \\
	1 & 0
	\end{array}\right]$.
\end{itemize}

We claim that for every $\ek_\mu \in \gh_{2k}$, $\operatorname{Sh}(\rho(\ek_\mu))$ has property ($\star$). The claim holds by inspection when $k=1$. Suppose it holds for some $k \in \mathbb{N},\ k \geq 1$, and let $\ekp_\mu \in \gh_{2k+2}$. By Lemma \ref{new e mu in terms of old}, $\rho(\ekp_\mu)$ was formed from $\rho(\ek_{\underline{\mu}-2})$ (where $\underline{\mu} = \mu \setminus \{1, 2\}$) by taking the Kronecker product with $I_2$ or $B$ on the right, and then matrix multiplying by $I_{2^{k+1}}, \rho(\ekp_1), \rho(\ekp_2), $ or $\rho(\ekp_{12})$. By the induction hypothesis, $\operatorname{Sh}(\rho(\ek_{\underline{\mu}-2}))$ has property ($\star$). By Lemmas \ref{what multiplication by e1 does}, \ref{what e2 does}, and \ref{what e12 does} and properties of the Kronecker product, in forming $\operatorname{Sh}(\rho(\ekp_\mu))$, either each nonzero entry of $\operatorname{Sh}(\rho(\ek_{\underline{\mu}-2}))$ is replaced by $\left[\begin{array}{cc} 1 & 0 \\ 0 & 1 \end{array} \right]$, or each nonzero entry of $\operatorname{Sh}(\rho(\ek_{\underline{\mu}-2}))$ is replaced by $\left[\begin{array}{cc} 0 & 1 \\ 1 & 0 \end{array} \right]$, so that ($\star$) continues to hold for $\operatorname{Sh}(\rho(\ekp_\mu))$. Thus the claim holds.

Now let $\ek_\mu \in \gh_{2k}$. By the Claim, $\operatorname{Sh}(\rho(\ek_\mu))$ is of the form either (i) $\left[\begin{array}{c|c}
	C_{k-1} & 0 \\
	\hline 
	0 & C_{k-1}
\end{array}\right]$ or (ii)  $\left[\begin{array}{c|c}
0 & C_{k-1}  \\
\hline 
C_{k-1} & 0
\end{array}\right]$ where $C_{k-1}$ is a $2^{k-1} \times 2^{k-1}$ permutation matrix. Note that if $\operatorname{Sh}(\rho(\ek_\mu))$ has form (ii), then the action of $\rho(\ek_\mu)$ on any $\vec{v} \in J_2^{S_{\Delta_{2k}}}$ interchanges the first $2^{k-1}$ entries in $\vec{v}$ with the last $2^{k-1}$ entries in $\vec{v}$; but if $\operatorname{Sh}(\rho(\ek_\mu))$ has form (i), then it does not: that is, in this case the action of $\rho(\ek_\mu)$ on any $\vec{v} \in J_2^{S_{\Delta_{2k}}}$ might rearrange the first $2^{k-1}$ entries of $\vec{v}$, but does not interchange any of the first $2^{k-1}$ entries of $\vec{v}$ with any of the last $2^{k-1}$ entries. Thus if $\operatorname{Sh}(\rho(\ek_\mu))$ has form (i), we set $S_k=(1)$; and if $\operatorname{Sh}(\rho(\ek_\mu))$ has form (ii), we set $S_k=A_{2^{k-1}}$. 

Next: again by the Claim, $C_{k-1}$ is of the form either (i) $\left[\begin{array}{c|c}
	C_{k-2} & 0 \\
	\hline 
	0 & C_{k-2}
\end{array}\right]$ or (ii)  $\left[\begin{array}{c|c}
	0 & C_{k-2}  \\
	\hline 
	C_{k-2} & 0
\end{array}\right]$ where $C_{k-2}$ is a $2^{k-2} \times 2^{k-2}$ permutation matrix. If $C_{k-1}$ has form (ii), then the action of $\rho(\ek_\mu)$ on any $\vec{v} \in J_2^{S_{\Delta_{2k}}}$ interchanges the first and second blocks in $\vec{v}$ of size $2^{k-2}$, and also interchanges the third and fourth blocks of size $2^{k-2}$, so that we set $S_{k-1} = A_{2^{k-2}}$. If $C_{k-1}$ has form (i), then the action of $\rho(\ek_\mu)$ on $\vec{v}$ might rearrange the entries within these four blocks of size $2^{k-2}$; but within a block of size $2^{k-1}$, it will not switch the order among the two blocks of size $2^{k-2}$, so that we set $S_{k-1} = (1)$.

Continuing in this way, for $0 \leq j \leq k-1$, we set $S_{k-j}=(1)$ if $C_{k-j}$ has the form $\left[\begin{array}{c|c}
	C_{k-j-1} & 0 \\
	\hline 
	0 & C_{k-j-1}
\end{array}\right]$, and we set $S_{k-j}=A_{2^{k-j-1}}$ if $C_{k-j}$ has the form $\left[\begin{array}{c|c}
0 & C_{k-j-1}  \\
\hline 
C_{k-j-1} & 0
\end{array}\right]$ (where $C_k = \operatorname{Sh}(\rho(\ek_\mu))$). Note that $C_1$ will be either $\left[\begin{array}{cc}
1 & 0  \\
0 & 1
\end{array}\right]$ or $\left[\begin{array}{cc}
0 & 1  \\
1 & 0
\end{array}\right]$.

By construction, if $\ek_\mu$ had real type, then the action of $\rho(\ek_\mu)$ on any point $\vec{v} \in J_2^{S_{\Delta_{2k}}}$ is identical to that of $S_k \circ S_{k-1} \circ \cdots \circ S_1$; and if $\ek_\mu$ had imaginary type, then the action of $\rho(\ek_\mu)$ on any point $\vec{v} \in J_2^{S_{\Delta_{2k}}}$ is identical to that of $i \cdot (S_k \circ S_{k-1} \circ \cdots \circ S_1)$.

Thus we have an embedding of the set $\{[\ek_\mu]:\ \ek_\mu \in \gh_{2k}\}$ of equivalence classes mod $\sim$ into the set $\{S_{1, \ldots, k},\ i \cdot S_{1, \ldots, k}: S_{i_l} \in \{A_{2^{l-1}}, (1)\} \textrm{ for } 1 \leq l \leq k\}$. Since both of these sets have size $2^{k+1}$, they are in fact isomorphic as sets of actions on $J_2^{S_{\Delta_{2k}}}$.
\end{proof}

Thus we can represent the $2^{k+1}$ classes of Clifford actions on $J_2^{S_{\Delta_{2k}}}$ in a way that is entirely dependent on the switch permutations $A_{2^{j}}\in Alt(2^{k})\subset \mathcal{S}_{2^{k}}$. The benefit in doing this is that we gain insight into the structure of Clifford multiplication on $J_2^{S_{\Delta_{2k}}}$ without establishing which generators are associated to which equivalence classes. 

\begin{defn}\label{induced CLifford permutations}
For $\ek_\mu \in \gh_{2k}$, we will refer to the map $S_{1 , \ldots, k}$ or $i \cdot S_{1 , \ldots, k}$ whose action on $J_2^{S_{\Delta_{2k}}}$ is identical to that of $\ek_\mu$ as the \textbf{induced Clifford permutation} for $\ek_\mu$. Specifically, we will call maps of the form $S_{1, \ldots, k}$ \textbf{strictly real induced Clifford permutations}, and we will call maps of the form $i \cdot S_{1, \ldots, k}$ \textbf{strictly imaginary induced Clifford permutations}. 
\end{defn}

\begin{remark}\label{mu4 X S action on J2}
It is clear that scalar multiplication by $i$ commutes with all $A_{2^{j}}$ and $(1)$ operators. Also, by properties of multiplication on $J^{S_{\Delta_0}}_2\subset \dfrac{\C}{\mathbb{Z}\oplus i \mathbb{Z}}$ by $i$, we have  $i \cdot i \cdot S_{1,\ldots,k}=S_{1,\ldots,k}$. This means that all of the $i \cdot S_{1, \ldots, k}$ maps are involutions; and that in taking into account scalar multiplication by $i$, we need only consider maps of the form $i \cdot S_{1,\ldots,k}$ for each strictly real induced Clifford permutation $S_{1,\ldots,k}$.
That is, the $\langle i \rangle \times \mathcal{S}_{2^k}$ action reduces to a $\{1, i\} \times \mathcal{S}_{2^k}$ action on the 2-torsion points $J_2^{S_{\Delta_{2k}}}$.
\end{remark}

To illustrate how these induced Clifford permutations work, we present the following example.

\begin{exmp}
On $J_2^{S_{\Delta_6}} \subset S_{\Delta_{6}}$ ($k=3$), we have $2^4=16$ induced Clifford permutations, $8$ real and $8$ imaginary. (Recall that the imaginary ones are the same permutations with multiplication by $i$  after we have permuted the indices.) Each strictly real induced Clifford permutation acts on an element $\vec{v}$ of $J_2^{S_{\Delta_6}}$ by reordering the entries of the vector $\vec{v}$.

For example, the map $S_{1,2,3} = A_4 \circ (1) \circ A_1$ acts on an element $\vec{v} \in J_2^{S_{\Delta_6}}$ as follows, where $\vec{v}$ is represented by a vector of the form $$\vec{v} = \langle \vec{v}[1], \vec{v}[2], \vec{v}[3], \vec{v}[4], \vec{v}[5], \vec{v}[6], \vec{v}[7], \vec{v}[8] \rangle$$ with $\vec{v}[j] \in J_2^{S_{\Delta_0}}$ for $1 \leq j \leq 8$: 

\[ S_{1,2,3} \cdot \vec{v} = A_4 \circ (1) \circ A_1 \cdot 
\left(
\begin{array}{c}
\\[-9pt]
\vec{v}[1]\\[2pt]
\vec{v}[2]\\[2pt]
\vec{v}[3]\\[2pt]
\vec{v}[4]\\[2pt]
\vec{v}[5]\\[2pt]
\vec{v}[6]\\[2pt]
\vec{v}[7]\\[2pt]
\vec{v}[8]\\[3pt]
\end{array}
\right) = A_4 \circ (1) \cdot \left(
\begin{array}{c}
\\[-9pt]
\vec{v}[2]\\[2pt]
\vec{v}[1]\\[2pt]
\vec{v}[4]\\[2pt]
\vec{v}[3]\\[2pt]
\vec{v}[6]\\[2pt]
\vec{v}[5]\\[2pt]
\vec{v}[8]\\[2pt]
\vec{v}[7]\\[3pt]
\end{array}
\right) = A_{4} \cdot \left(
\begin{array}{c}
\\[-9pt]
\vec{v}[2]\\[2pt]
\vec{v}[1]\\[2pt]
\vec{v}[4]\\[2pt]
\vec{v}[3]\\[2pt]
\vec{v}[6]\\[2pt]
\vec{v}[5]\\[2pt]
\vec{v}[8]\\[2pt]
\vec{v}[7]\\[3pt]
\end{array}
\right) = \left(
\begin{array}{c}
\\[-9pt] 
\vec{v}[6]\\[2pt]
\vec{v}[5]\\[2pt]
\vec{v}[8]\\[2pt]
\vec{v}[7]\\[2pt]
\vec{v}[2]\\[2pt]
\vec{v}[1]\\[2pt]
\vec{v}[4]\\[2pt]
\vec{v}[3]\\[3pt]
\end{array}
\right). \]
(That is, here $S_{1,2,3}$ is the permutation $(16)(25)(38)(47)\in \mathcal{S}_8$.)

\end{exmp}

We can also picture the strictly real Clifford permutations acting on $J_2^{S_{\Delta_{2k}}}$ using shoelace diagrams, as shown in Figure \ref{shoelace diagrams} for $k=3$ (dimension 8), where the numbers 1 through 8 in each column represent the entries $\vec{v}[1]$ through $\vec{v}[8]$ of a vector $\vec{v} \in J_2^{S_{\Delta_6}}$. (The strictly imaginary Clifford permutations can be similarly pictured after we introduce a different way of representing 2-torsion points in Section \ref{sec: fixed points and translation constants}.)

\begin{figure}[h]
\label{shoelace diagrams}
\centering

\begin{tikzpicture}[
	number/.style={rectangle, minimum size=6mm, draw=black}]
	
	\matrix[row sep=0mm,column sep=0mm] {
		\node (1) [number] {1}; \\
		\node (2) [number] {2}; \\
		\node (3) [number] {3}; \\
		\node (4) [number] {4}; \\
		\node (1) [number] {5}; \\
		\node (2) [number] {6}; \\
		\node (3) [number] {7}; \\
		\node (4) [number] {8}; \\
	};
	
	\node at (1,-3) {$(1) \circ (1) \circ (1)$};
	
	\begin{scope}[yshift=0cm,xshift=2cm]
		\matrix[row sep=0mm,column sep=0mm] {
			\node (r1) [number] {1}; \\
			\node (r2) [number] {2}; \\
			\node (r3) [number] {3}; \\
			\node (r4) [number] {4}; \\
			\node (r5) [number] {5}; \\
			\node (r6) [number] {6}; \\
			\node (r7) [number] {7}; \\
			\node (r8) [number] {8}; \\
		};
	\end{scope}
	
	
	\node (a8) at (.3,-2.14375) {};
	\node (b8) at (1.7,-2.14375) {};
	\node (a7) at (.3,-1.53125) {};
	\node (b7) at (1.7,-1.53125) {};
	\node (a6) at (.3,-0.91875) {};
	\node (b6) at (1.7,-0.91875) {};
	\node (a5) at (.3,-0.30625) {};
	\node (b5) at (1.7,-0.30625) {};
	\node (a4) at (.3,0.30625) {};
	\node (b4) at (1.7,0.30625) {};
	\node (a3) at (.3,0.91875) {};
	\node (b3) at (1.7,0.91875) {};
	\node (a2) at (.3,1.53125) {};
	\node (b2) at (1.7,1.53125) {};
	\node (a1) at (.3,2.14375) {};
	\node (b1) at (1.7,2.14375) {};
	
	\draw [-{Stealth[length=1mm]}] (a8) -- (b8);
	\draw [-{Stealth[length=1mm]}] (a7) -- (b7);
	\draw [-{Stealth[length=1mm]}] (a6) -- (b6);
	\draw [-{Stealth[length=1mm]}] (a5) -- (b5);
	\draw [-{Stealth[length=1mm]}] (a4) -- (b4);
	\draw [-{Stealth[length=1mm]}] (a3) -- (b3);
	\draw [-{Stealth[length=1mm]}] (a2) -- (b2);
	\draw [-{Stealth[length=1mm]}] (a1) -- (b1);
	
\end{tikzpicture}
\hspace{1mm}
\begin{tikzpicture}[
	number/.style={rectangle, minimum size=6mm, draw=black}]
	
	\matrix[row sep=0mm,column sep=0mm] {
		\node (1) [number] {1}; \\
		\node (2) [number] {2}; \\
		\node (3) [number] {3}; \\
		\node (4) [number] {4}; \\
		\node (1) [number] {5}; \\
		\node (2) [number] {6}; \\
		\node (3) [number] {7}; \\
		\node (4) [number] {8}; \\
	};
	
	\node at (1,-3) {$(1) \circ (1) \circ A_1$};
	
	\begin{scope}[yshift=0cm,xshift=2cm]
		\matrix[row sep=0mm,column sep=0mm] {
			\node (r1) [number] {2}; \\
			\node (r2) [number] {1}; \\
			\node (r3) [number] {4}; \\
			\node (r4) [number] {3}; \\
			\node (r5) [number] {6}; \\
			\node (r6) [number] {5}; \\
			\node (r7) [number] {8}; \\
			\node (r8) [number] {7}; \\
		};
	\end{scope}
	
	
	\node (a8) at (.3,-2.14375) {};
	\node (b8) at (1.7,-2.14375) {};
	\node (a7) at (.3,-1.53125) {};
	\node (b7) at (1.7,-1.53125) {};
	\node (a6) at (.3,-0.91875) {};
	\node (b6) at (1.7,-0.91875) {};
	\node (a5) at (.3,-0.30625) {};
	\node (b5) at (1.7,-0.30625) {};
	\node (a4) at (.3,0.30625) {};
	\node (b4) at (1.7,0.30625) {};
	\node (a3) at (.3,0.91875) {};
	\node (b3) at (1.7,0.91875) {};
	\node (a2) at (.3,1.53125) {};
	\node (b2) at (1.7,1.53125) {};
	\node (a1) at (.3,2.14375) {};
	\node (b1) at (1.7,2.14375) {};
	
	\draw [-{Stealth[length=1mm]}] (a8) -- (b7);
	\draw [-{Stealth[length=1mm]}] (a7) -- (b8);
	\draw [-{Stealth[length=1mm]}] (a6) -- (b5);
	\draw [-{Stealth[length=1mm]}] (a5) -- (b6);
	\draw [-{Stealth[length=1mm]}] (a4) -- (b3);
	\draw [-{Stealth[length=1mm]}] (a3) -- (b4);
	\draw [-{Stealth[length=1mm]}] (a2) -- (b1);
	\draw [-{Stealth[length=1mm]}] (a1) -- (b2);
	
\end{tikzpicture}
\hspace{1mm}
\begin{tikzpicture}[
	number/.style={rectangle, minimum size=6mm, draw=black}]
	
	\matrix[row sep=0mm,column sep=0mm] {
		\node (1) [number] {1}; \\
		\node (2) [number] {2}; \\
		\node (3) [number] {3}; \\
		\node (4) [number] {4}; \\
		\node (1) [number] {5}; \\
		\node (2) [number] {6}; \\
		\node (3) [number] {7}; \\
		\node (4) [number] {8}; \\
	};
	
	\node at (1,-3) {$(1) \circ A_2 \circ (1)$};
	
	\begin{scope}[yshift=0cm,xshift=2cm]
		\matrix[row sep=0mm,column sep=0mm] {
			\node (r1) [number] {3}; \\
			\node (r2) [number] {4}; \\
			\node (r3) [number] {1}; \\
			\node (r4) [number] {2}; \\
			\node (r5) [number] {7}; \\
			\node (r6) [number] {8}; \\
			\node (r7) [number] {5}; \\
			\node (r8) [number] {6}; \\
		};
	\end{scope}
	
	
	\node (a8) at (.3,-2.14375) {};
	\node (b8) at (1.7,-2.14375) {};
	\node (a7) at (.3,-1.53125) {};
	\node (b7) at (1.7,-1.53125) {};
	\node (a6) at (.3,-0.91875) {};
	\node (b6) at (1.7,-0.91875) {};
	\node (a5) at (.3,-0.30625) {};
	\node (b5) at (1.7,-0.30625) {};
	\node (a4) at (.3,0.30625) {};
	\node (b4) at (1.7,0.30625) {};
	\node (a3) at (.3,0.91875) {};
	\node (b3) at (1.7,0.91875) {};
	\node (a2) at (.3,1.53125) {};
	\node (b2) at (1.7,1.53125) {};
	\node (a1) at (.3,2.14375) {};
	\node (b1) at (1.7,2.14375) {};
	
	\draw [-{Stealth[length=1mm]}] (a8) -- (b6);
	\draw [-{Stealth[length=1mm]}] (a7) -- (b5);
	\draw [-{Stealth[length=1mm]}] (a6) -- (b8);
	\draw [-{Stealth[length=1mm]}] (a5) -- (b7);
	\draw [-{Stealth[length=1mm]}] (a4) -- (b2);
	\draw [-{Stealth[length=1mm]}] (a3) -- (b1);
	\draw [-{Stealth[length=1mm]}] (a2) -- (b4);
	\draw [-{Stealth[length=1mm]}] (a1) -- (b3);
	
\end{tikzpicture}
\hspace{1mm}
\begin{tikzpicture}[
	number/.style={rectangle, minimum size=6mm, draw=black}]
	
	\matrix[row sep=0mm,column sep=0mm] {
		\node (1) [number] {1}; \\
		\node (2) [number] {2}; \\
		\node (3) [number] {3}; \\
		\node (4) [number] {4}; \\
		\node (1) [number] {5}; \\
		\node (2) [number] {6}; \\
		\node (3) [number] {7}; \\
		\node (4) [number] {8}; \\
	};
	
	\node at (1,-3) {$A_4 \circ (1) \circ (1)$};
	
	\begin{scope}[yshift=0cm,xshift=2cm]
		\matrix[row sep=0mm,column sep=0mm] {
			\node (r1) [number] {5}; \\
			\node (r2) [number] {6}; \\
			\node (r3) [number] {7}; \\
			\node (r4) [number] {8}; \\
			\node (r5) [number] {1}; \\
			\node (r6) [number] {2}; \\
			\node (r7) [number] {3}; \\
			\node (r8) [number] {4}; \\
		};
	\end{scope}
	
	
	\node (a8) at (.3,-2.14375) {};
	\node (b8) at (1.7,-2.14375) {};
	\node (a7) at (.3,-1.53125) {};
	\node (b7) at (1.7,-1.53125) {};
	\node (a6) at (.3,-0.91875) {};
	\node (b6) at (1.7,-0.91875) {};
	\node (a5) at (.3,-0.30625) {};
	\node (b5) at (1.7,-0.30625) {};
	\node (a4) at (.3,0.30625) {};
	\node (b4) at (1.7,0.30625) {};
	\node (a3) at (.3,0.91875) {};
	\node (b3) at (1.7,0.91875) {};
	\node (a2) at (.3,1.53125) {};
	\node (b2) at (1.7,1.53125) {};
	\node (a1) at (.3,2.14375) {};
	\node (b1) at (1.7,2.14375) {};
	
	\draw [-{Stealth[length=1mm]}] (a8) -- (b4);
	\draw [-{Stealth[length=1mm]}] (a7) -- (b3);
	\draw [-{Stealth[length=1mm]}] (a6) -- (b2);
	\draw [-{Stealth[length=1mm]}] (a5) -- (b1);
	\draw [-{Stealth[length=1mm]}] (a4) -- (b8);
	\draw [-{Stealth[length=1mm]}] (a3) -- (b7);
	\draw [-{Stealth[length=1mm]}] (a2) -- (b6);
	\draw [-{Stealth[length=1mm]}] (a1) -- (b5);
	
\end{tikzpicture}
\vspace{5mm}

\begin{tikzpicture}[
	number/.style={rectangle, minimum size=6mm, draw=black}]
	
	\matrix[row sep=0mm,column sep=0mm] {
		\node (1) [number] {1}; \\
		\node (2) [number] {2}; \\
		\node (3) [number] {3}; \\
		\node (4) [number] {4}; \\
		\node (1) [number] {5}; \\
		\node (2) [number] {6}; \\
		\node (3) [number] {7}; \\
		\node (4) [number] {8}; \\
	};
	
	\node at (1,-3) {$(1) \circ A_2 \circ A_1$};
	
	\begin{scope}[yshift=0cm,xshift=2cm]
		\matrix[row sep=0mm,column sep=0mm] {
			\node (r1) [number] {4}; \\
			\node (r2) [number] {3}; \\
			\node (r3) [number] {2}; \\
			\node (r4) [number] {1}; \\
			\node (r5) [number] {8}; \\
			\node (r6) [number] {7}; \\
			\node (r7) [number] {6}; \\
			\node (r8) [number] {5}; \\
		};
	\end{scope}
	
	
	\node (a8) at (.3,-2.14375) {};
	\node (b8) at (1.7,-2.14375) {};
	\node (a7) at (.3,-1.53125) {};
	\node (b7) at (1.7,-1.53125) {};
	\node (a6) at (.3,-0.91875) {};
	\node (b6) at (1.7,-0.91875) {};
	\node (a5) at (.3,-0.30625) {};
	\node (b5) at (1.7,-0.30625) {};
	\node (a4) at (.3,0.30625) {};
	\node (b4) at (1.7,0.30625) {};
	\node (a3) at (.3,0.91875) {};
	\node (b3) at (1.7,0.91875) {};
	\node (a2) at (.3,1.53125) {};
	\node (b2) at (1.7,1.53125) {};
	\node (a1) at (.3,2.14375) {};
	\node (b1) at (1.7,2.14375) {};
	
	\draw [-{Stealth[length=1mm]}] (a8) -- (b5);
	\draw [-{Stealth[length=1mm]}] (a7) -- (b6);
	\draw [-{Stealth[length=1mm]}] (a6) -- (b7);
	\draw [-{Stealth[length=1mm]}] (a5) -- (b8);
	\draw [-{Stealth[length=1mm]}] (a4) -- (b1);
	\draw [-{Stealth[length=1mm]}] (a3) -- (b2);
	\draw [-{Stealth[length=1mm]}] (a2) -- (b3);
	\draw [-{Stealth[length=1mm]}] (a1) -- (b4);
	
\end{tikzpicture}
\hspace{1mm}
\begin{tikzpicture}[
	number/.style={rectangle, minimum size=6mm, draw=black}]
	
	\matrix[row sep=0mm,column sep=0mm] {
		\node (1) [number] {1}; \\
		\node (2) [number] {2}; \\
		\node (3) [number] {3}; \\
		\node (4) [number] {4}; \\
		\node (1) [number] {5}; \\
		\node (2) [number] {6}; \\
		\node (3) [number] {7}; \\
		\node (4) [number] {8}; \\
	};
	
	\node at (1,-3) {$A_4 \circ (1) \circ A_1$};
	
	\begin{scope}[yshift=0cm,xshift=2cm]
		\matrix[row sep=0mm,column sep=0mm] {
			\node (r1) [number] {6}; \\
			\node (r2) [number] {5}; \\
			\node (r3) [number] {8}; \\
			\node (r4) [number] {7}; \\
			\node (r5) [number] {2}; \\
			\node (r6) [number] {1}; \\
			\node (r7) [number] {4}; \\
			\node (r8) [number] {3}; \\
		};
	\end{scope}
	
	
	\node (a8) at (.3,-2.14375) {};
	\node (b8) at (1.7,-2.14375) {};
	\node (a7) at (.3,-1.53125) {};
	\node (b7) at (1.7,-1.53125) {};
	\node (a6) at (.3,-0.91875) {};
	\node (b6) at (1.7,-0.91875) {};
	\node (a5) at (.3,-0.30625) {};
	\node (b5) at (1.7,-0.30625) {};
	\node (a4) at (.3,0.30625) {};
	\node (b4) at (1.7,0.30625) {};
	\node (a3) at (.3,0.91875) {};
	\node (b3) at (1.7,0.91875) {};
	\node (a2) at (.3,1.53125) {};
	\node (b2) at (1.7,1.53125) {};
	\node (a1) at (.3,2.14375) {};
	\node (b1) at (1.7,2.14375) {};
	
	\draw [-{Stealth[length=1mm]}] (a8) -- (b3);
	\draw [-{Stealth[length=1mm]}] (a7) -- (b4);
	\draw [-{Stealth[length=1mm]}] (a6) -- (b1);
	\draw [-{Stealth[length=1mm]}] (a5) -- (b2);
	\draw [-{Stealth[length=1mm]}] (a4) -- (b7);
	\draw [-{Stealth[length=1mm]}] (a3) -- (b8);
	\draw [-{Stealth[length=1mm]}] (a2) -- (b5);
	\draw [-{Stealth[length=1mm]}] (a1) -- (b6);
	
\end{tikzpicture}
\hspace{1mm}
\begin{tikzpicture}[
	number/.style={rectangle, minimum size=6mm, draw=black}]
	
	\matrix[row sep=0mm,column sep=0mm] {
		\node (1) [number] {1}; \\
		\node (2) [number] {2}; \\
		\node (3) [number] {3}; \\
		\node (4) [number] {4}; \\
		\node (1) [number] {5}; \\
		\node (2) [number] {6}; \\
		\node (3) [number] {7}; \\
		\node (4) [number] {8}; \\
	};
	
	\node at (1,-3) {$A_4 \circ A_2 \circ (1)$};
	
	\begin{scope}[yshift=0cm,xshift=2cm]
		\matrix[row sep=0mm,column sep=0mm] {
			\node (r1) [number] {7}; \\
			\node (r2) [number] {8}; \\
			\node (r3) [number] {5}; \\
			\node (r4) [number] {6}; \\
			\node (r5) [number] {3}; \\
			\node (r6) [number] {4}; \\
			\node (r7) [number] {1}; \\
			\node (r8) [number] {2}; \\
		};
	\end{scope}
	
	
	\node (a8) at (.3,-2.14375) {};
	\node (b8) at (1.7,-2.14375) {};
	\node (a7) at (.3,-1.53125) {};
	\node (b7) at (1.7,-1.53125) {};
	\node (a6) at (.3,-0.91875) {};
	\node (b6) at (1.7,-0.91875) {};
	\node (a5) at (.3,-0.30625) {};
	\node (b5) at (1.7,-0.30625) {};
	\node (a4) at (.3,0.30625) {};
	\node (b4) at (1.7,0.30625) {};
	\node (a3) at (.3,0.91875) {};
	\node (b3) at (1.7,0.91875) {};
	\node (a2) at (.3,1.53125) {};
	\node (b2) at (1.7,1.53125) {};
	\node (a1) at (.3,2.14375) {};
	\node (b1) at (1.7,2.14375) {};
	
	\draw [-{Stealth[length=1mm]}] (a8) -- (b2);
	\draw [-{Stealth[length=1mm]}] (a7) -- (b1);
	\draw [-{Stealth[length=1mm]}] (a6) -- (b4);
	\draw [-{Stealth[length=1mm]}] (a5) -- (b3);
	\draw [-{Stealth[length=1mm]}] (a4) -- (b6);
	\draw [-{Stealth[length=1mm]}] (a3) -- (b5);
	\draw [-{Stealth[length=1mm]}] (a2) -- (b8);
	\draw [-{Stealth[length=1mm]}] (a1) -- (b7);
	
\end{tikzpicture}
\hspace{1mm}
\begin{tikzpicture}[
	number/.style={rectangle, minimum size=6mm, draw=black}]
	
	\matrix[row sep=0mm,column sep=0mm] {
		\node (1) [number] {1}; \\
		\node (2) [number] {2}; \\
		\node (3) [number] {3}; \\
		\node (4) [number] {4}; \\
		\node (1) [number] {5}; \\
		\node (2) [number] {6}; \\
		\node (3) [number] {7}; \\
		\node (4) [number] {8}; \\
	};
	
	\node at (1,-3) {$A_4 \circ A_2 \circ (1)$};
	
	\begin{scope}[yshift=0cm,xshift=2cm]
		\matrix[row sep=0mm,column sep=0mm] {
			\node (r1) [number] {8}; \\
			\node (r2) [number] {7}; \\
			\node (r3) [number] {6}; \\
			\node (r4) [number] {5}; \\
			\node (r5) [number] {4}; \\
			\node (r6) [number] {3}; \\
			\node (r7) [number] {2}; \\
			\node (r8) [number] {1}; \\
		};
	\end{scope}
	
	
	\node (a8) at (.3,-2.14375) {};
	\node (b8) at (1.7,-2.14375) {};
	\node (a7) at (.3,-1.53125) {};
	\node (b7) at (1.7,-1.53125) {};
	\node (a6) at (.3,-0.91875) {};
	\node (b6) at (1.7,-0.91875) {};
	\node (a5) at (.3,-0.30625) {};
	\node (b5) at (1.7,-0.30625) {};
	\node (a4) at (.3,0.30625) {};
	\node (b4) at (1.7,0.30625) {};
	\node (a3) at (.3,0.91875) {};
	\node (b3) at (1.7,0.91875) {};
	\node (a2) at (.3,1.53125) {};
	\node (b2) at (1.7,1.53125) {};
	\node (a1) at (.3,2.14375) {};
	\node (b1) at (1.7,2.14375) {};

	\draw [-{Stealth[length=1mm]}] (a8) -- (b1);
	\draw [-{Stealth[length=1mm]}] (a7) -- (b2);
	\draw [-{Stealth[length=1mm]}] (a6) -- (b3);
	\draw [-{Stealth[length=1mm]}] (a5) -- (b4);
	\draw [-{Stealth[length=1mm]}] (a4) -- (b5);
	\draw [-{Stealth[length=1mm]}] (a3) -- (b6);
	\draw [-{Stealth[length=1mm]}] (a2) -- (b7);
	\draw [-{Stealth[length=1mm]}] (a1) -- (b8);

\end{tikzpicture}

\caption{The 8 real classes of Clifford multiplication on $J_2^{S_{\Delta_{6}}}$}
\end{figure}
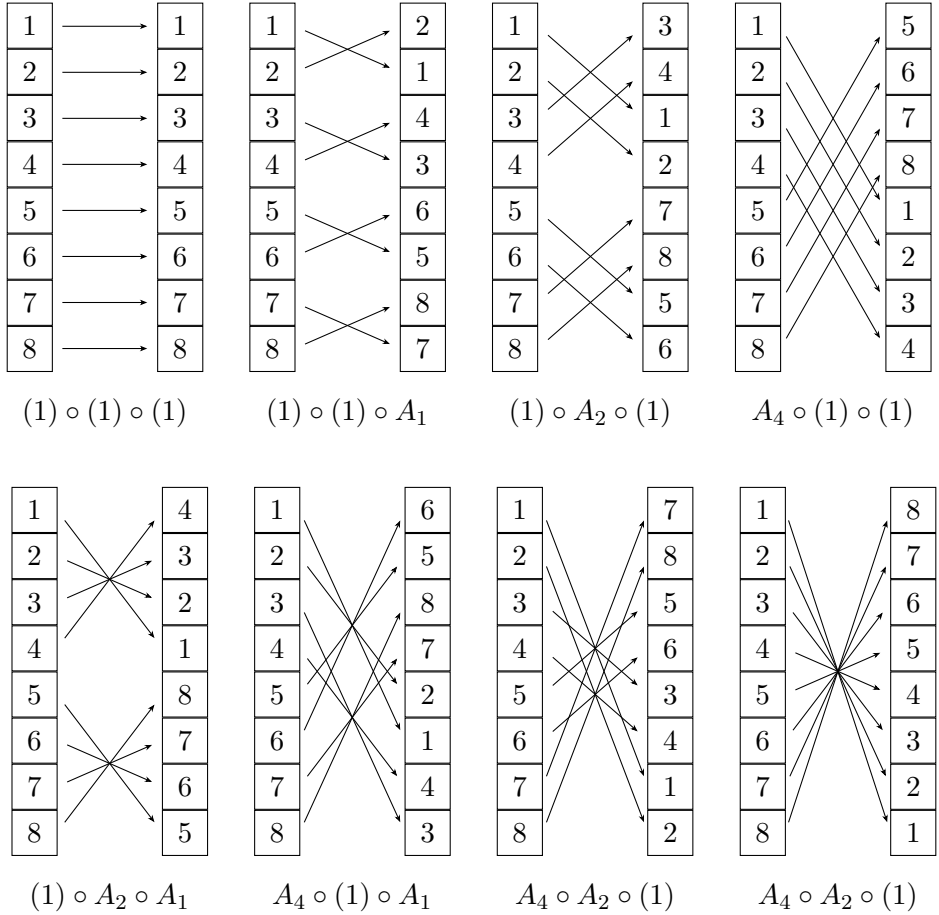

We now have three ways to view Clifford multiplication on $J_2^{S_{\Delta_{2k}}}$: as actions by Clifford algebra elements $\ek_\mu$, by representations $\rho(\ek_\mu)$, or by induced Clifford  permutations $S_{1, \ldots, k}$ or $i \cdot S_{1, \ldots, k}$.

A note on terminology: we have defined the maps $S_{1 , \ldots, k}$ as permutations on the set $\{1, \ldots, 2^k\}$. We will show in Section \ref{sec: fixed points and translation constants} that maps of the form $i \cdot S_{1 , \ldots, k}$ acting on $J_2^{S_{\Delta_{2k}}}$ can be seen as permutations on the set $\{1, \ldots, 2^{k+1}\}$, so that the term ``Clifford permutation'' is also appropriate for these maps.

We also note that each strictly real induced Clifford permutation is associated with a $\ek_\mu \in \gh_{2k}$ of real type, and each strictly imaginary induced Clifford permutation is associated with a $\ek_\mu \in \gh_{2k}$ of imaginary type.

\section{The group $\hat{\Gamma}_{2k}/J_2^{S_{\Delta_{2k}}}$}\label{section: the group of actions on 2-torsion points}

As we saw in Theorem \ref{number of unique Clifford actions}, the $2^{2k}$ generators $e_{\mu}$ acting on $J_2^{S_{\Delta_{2k}}}$ give us a total of $2^{k+1}$ representative permutations of the form $S_{1, \ldots ,k}$ or $i \cdot S_{1, \ldots ,k}$, which are all involutions. Let Cliff$(Alt(2^{k}))$ denote the set of strictly real induced Clifford permutations; that is, Cliff$(Alt(2^{k}))$ is the set of maps of the form $S_{1, \ldots, k}$. We will show that $\textrm{Cliff}(Alt(2^k))$ is an Abelian subgroup of $Alt(2^k)$. First we show that the switch permutations $A_{2^j}$ commute with each other.

For what follows, as in Definition \ref{switch permutations}, ``blocks'' of size $2^l$ within a linearly ordered set $(X, <)$ of size $2^m$ (for $0 \leq l < m$) will always refer to one of the following sets (see Figure \ref{blocks} for an illustration of these blocks):
	
	\[
	\begin{array}{lcl}
		\textrm{Block 1:} & &\textrm{the first $2^l$-many elements of $X$, in order} \\
		\textrm{Block 2:} & &\textrm{the second $2^l$-many elements of $X$, in order} \\
		& \vdots & \\
		\textrm{Block $2^{m-l}$:} & & \textrm{the last $2^l$-many elements of $X$, in order} 
	\end{array}
	\]
	
	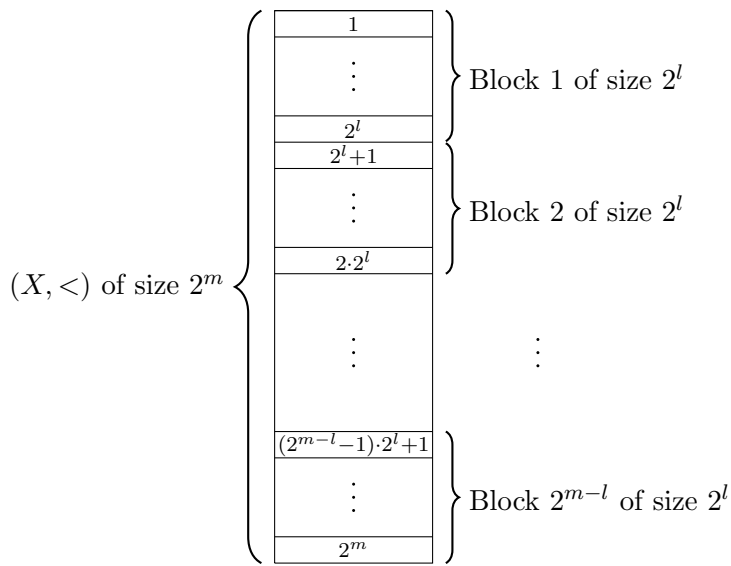
\begin{figure}[ht]
		
		\centering
	
	\begin{tikzpicture}[scale=0.7]
		\draw (0,0) -- (3,0) -- (3,10.5) -- (0,10.5) -- (0,0);
		
		\draw (0,10) -- (3,10);
		\draw (0,10.5) -- (3,10.5);
		\draw (0,8.5) -- (3,8.5);
		\draw (0,8) -- (3,8);
		
		\node[littledot] at (1.5,9.5){};
		\node[littledot] at (1.5,9.25){};
		\node[littledot] at (1.5,9.0){};
		
		\node at (1.5,10.25) {$\scriptstyle{1}$};
		\node at (1.5,8.25) {$\scriptstyle{2^l}$};
		
		\draw (0,7.5) -- (3,7.5);
		\draw (0,6) -- (3,6);
		\draw (0,5.5) -- (3,5.5);
		
		\node[littledot] at (1.5,7){};
		\node[littledot] at (1.5,6.75){};
		\node[littledot] at (1.5,6.5){};
		
		\node at (1.5,7.75) {$\scriptstyle{2^l + 1}$};
		\node at (1.5,5.75) {$\scriptstyle{2 \cdot 2^l}$};
		
		\node[littledot] at (1.5,4.25){};
		\node[littledot] at (1.5,4){};
		\node[littledot] at (1.5,3.75){};
		
		\draw (0,2.5) -- (3,2.5);
		\draw (0,2) -- (3,2);
		\draw (0,0.5) -- (3,0.5);
		
		\node[littledot] at (1.5,1.5){};
		\node[littledot] at (1.5,1.25){};
		\node[littledot] at (1.5,1){};
		
		\node at (1.5,2.25) {$\scriptstyle{(2^{m-l}-1) \cdot 2^l + 1}$};
		\node at (1.5,.25) {$\scriptstyle{2^m}$};
		
		\draw [thick, decorate,
		decoration = {brace, raise=5pt,amplitude=10pt}] (0,0) --  (0,10.5);
		\node [left] at (-.75,5.25) {$(X,<)$ of size $2^m$};
		
		\draw [thick, decorate,
		decoration = {brace, raise=5pt, amplitude=5pt, mirror}] (3,8.02) --  (3,10.5);
		\node [right] at (3.5,9.25) {Block 1 of size $2^l$};
		
		\draw [thick, decorate,
		decoration = {brace, raise=5pt, amplitude=5pt, mirror}] (3,5.5) --  (3,7.98);
		\node [right] at (3.5,6.75) {Block 2 of size $2^l$};
		
		\draw [thick, decorate,
		decoration = {brace, raise=5pt, amplitude=5pt, mirror}] (3,0) --  (3,2.5);
		\node [right] at (3.5,1.25) {Block $2^{m-l}$ of size $2^l$};
		
		\node[littledot] at (5,4.25){};
		\node[littledot] at (5,4){};
		\node[littledot] at (5,3.75){};
	\end{tikzpicture}
	
	\caption{Blocks of size $2^l$ within a linearly ordered set $(X, <)$ of size $2^m$ (where the $n$th element of $X$ is identified with the number $n$)}
	\label{blocks}
	\end{figure}

\begin{lem}\label{switch maps commute}
	Let $k \in \mathbb{N},\  k \geq 1$. For all $0 \leq i < j < k$, $A_{2^i} \circ A_{2^j} = A_{2^j} \circ A_{2^i}$.
\end{lem}

\begin{proof}

	Fix $0 \leq i < j < k$ and let $(X, < )$ be an ordered set of size $2^k$. Recall that the map $A_{2^l}$ pairwise switches adjacent blocks of size $2^l$, starting with the first and second blocks, and ending with the $(2^{k-l}-1)$th and $(2^{k-l})$th blocks.
	
	We make the following observations on the behavior of the maps $A_{2^i}$ and $A_{2^j}$.
	
	\textit{Observation 1:} The ordering \textit{among} the $2^{j-i}$ blocks of size $2^i$ within a block $Y$ of size $2^j$
	in $X$ is preserved by the map $A_{2^j}$, and so is the ordering \textit{within} each such block of size $2^i$. That is, the map $A_{2^j}$, which switches ``bigger'' blocks (of size $2^j$), will not change the ordering among the ``smaller'' blocks (of size $2^i$) within it; and it will not change the ordering of the elements within those ``smaller'' blocks either. 
	
	\textit{Observation 2:} The action of the map $A_{2^i}$ on $X$ will preserve blocks of size $2^j$ as sets, although it will not preserve the order of the elements within those blocks. That is, if $Y$ is one of the ``bigger'' blocks of size $2^j$ in $X$, then $A_{2^i}(Y) = Y$ as sets.
	
	\textit{Observation 3:} The action of the map $A_{2^i}$ on $X$ will preserve the ordering of the blocks of size $2^j$, although it will permute the order of the elements within the blocks of size $2^j$. That is, if $Y$ and $Y'$ are blocks of size $2^j$ and $Y < Y'$, then $A_{2^i}(Y) < A_{2^i}(Y)$.
	
	\textit{Observation 4:} Suppose for some $1 \leq p \leq 2^{k-j}$ that $Y$ is the $p$th block of size $2^j$ in $X$. If $p$ is odd, then $A_{2^j}(Y)$ is the $(p+1)$th block of size $2^j$ in $A_{2^j}(X)$ (with the new order); and if $p$ is even, then $A_{2^j}(Y)$ is the $(p-1)$th block of size $2^j$ in $A_{2^j}(X)$.
	
	We now show that $A_{2^j} \circ A_{2^i} = A_{2^i} \circ A_{2^j}$. Suppose that $x \in X$ is in the $p_j$th block of size $2^j$ within $X$; and that within that block of size $2^j$, $x$ is in the $q_i$th block of size $2^i$; and that within that block of size $2^i$, $x$ is the $r$th element. We need to check that $x$ gets moved to the same location by $A_{2^j} \circ A_{2^i}$ as by $A_{2^i} \circ A_{2^j}$. 
	
	There are four cases, depending on whether $p_j$ and $q_i$ are even or odd. We present the case where $p_j$ and $q_i$ are both odd; the other cases are similar.
	
	Supposing that $p_j$ and $q_i$ are odd, first consider $A_{2^i} \circ A_{2^j}$. After applying $A_{2^j}$, the $p_j$th block of size $2^j$ in $X$ will now be the $(p_j + 1)$th such block, by Observation 4. $x$ will be the $r$th element of the $q_i$th block of size $2^i$ among the blocks of size $2^i$ within that block of size $2^j$ (by Observation 1). Then after applying $A_{2^i}$, $x$ will be the $r$th element of the $(q_i + 1)$th block of size $2^i$ within the $(p_j + 1)$th block of size $2^j$, by Observations 1 and 4. 
	
	Next consider $A_{2^j} \circ A_{2^i}$. After applying $A_{2^i}$, $x$ will be the $r$th element of the $(q_i + 1)$th block of size $2^i$ within the $p_j$th block of size $2^j$, by Observations 2 and 3. Then, after applying $A_{2^j}$, $x$ will be the $r$th element of the $(q_i + 1)$th block of size $2^i$ within the $(p_j + 1)$th block of size $2^j$ (by Observations 1 and 4). 

Thus, in this case, under both $A_{2^i} \circ A_{2^j}$ and $A_{2^j} \circ A_{2^i}$, each $x$ in our ordered set $X$ of size $2^k$ gets sent to the same position within (the new ordering of) $X$, so that $A_{2^i} \circ A_{2^j} = A_{2^j} \circ A_{2^i}$.

 The other three cases (in which $p_j$ and $q_i$ are both even, or in which one of them is odd and the other even) are similar.
	
\end{proof}

\begin{defn}
	A permutation $\sigma\in\mathcal{S}_{2^{k}}$ is called a \textbf{derangement by disjoint transpositions} if 
	
	\begin{enumerate}
		\item  for every $n \in \{1, \ldots, 2^k\}$, $\sigma(n) \neq n$; and
		\item $\sigma$ is a product of disjoint transpositions. 
	\end{enumerate}
	
	Equivalently, a permutation $\sigma$ of $\{1, \ldots, 2^k\}$ is a derangement by disjoint transpositions if for each $n \in \{1, \ldots, 2^k\}$, there is an $m \in \{1, \ldots, 2^k\}$ with $m \neq n$ such that $\sigma(n)=m$ and $\sigma(m)=n$. 
\end{defn}

\begin{remark}\label{about ordered sets}
    Note that we are considering the numbers $1, \ldots, 2^k$ as indices for locations of elements under the current permutation's ordering of our original set $X$. So, for example, if $X$ is originally ordered as $\langle x_1, x_2, x_3, x_4 \rangle$ and we apply the switch permutation $A_1$, we get the new ordered set $A_1(X) = \langle x_2, x_1, x_4, x_3 \rangle$; and in this new order, $x_2$ is considered the first element (indexed by the number 1), etc. 
\end{remark}

Recall that every lattice Clifford action on $J_2^{S_{\Delta_{2k}}}$ is of order 2. We now show that the $2^{k}$ strictly real induced Clifford   permutations are derangements by disjoint transpositions of $\{1, \ldots, 2^k\}$ acting on the $2^k$ entries of a 2-torsion point $\vec{v} \in J_2^{S_{\Delta_{2k}}}$. (The other $2^{k}$ -- that is, the strictly imaginary induced Clifford permutations -- are the same permutations with a scalar multiplication by $i$.)

\begin{prop}\label{derangement by disjoint transpositions}
	Let $k \in \mathbb{N},\ k \geq 1$. Every non-identity strictly real induced Clifford permutation $S_{1, \ldots, k}$ is a derangement by disjoint transpositions of $\{1, \ldots, 2^k\}$ of order 
	$2$. 
\end{prop}

\begin{proof}
	The claim is true by inspection in the case $k=1$. 
	
	Suppose the claim holds for some $k>1$, and consider a strictly real induced Clifford permutation in $\mathcal{S}_{2^{k+1}}$. Each such permutation is of the form $(1) \circ S_{1,\ldots, k}$ or $A_{2^{k}} \circ S_{1,\ldots ,k}$ where $ S_{1,\ldots ,k}$ is the induced Clifford permutation of some  $\ek_\mu \in J_2^{S_{\Delta_{2k}}}$ of real type. By the induction hypothesis, $S_{1,\ldots,k}$ is a derangement by disjoint transpositions separately on the blocks $\{1, \ldots, 2^k\}$ and $\{2^k + 1, \ldots, 2^{k+1}\}$. Then $(1) \circ S_{1,\ldots,k}$ would be a derangement by disjoint transpositions on $\{1, \ldots, 2^{k+1}\}$, since its action is the same as that of $ S_{1,\ldots,k}$. 
	We claim that also $A_{2^{k}} \circ S_{1,\ldots,{k}}$ would be a derangement by disjoint transpositions on $\{1, \ldots, 2^{k+1}\}$: let $n \in \{1, \ldots, 2^{k+1}\}$. First, $A_{2^{k}} \circ S_{1,\ldots,{k}}(n) \neq n$ since $S_{1,\ldots,{k}}(\{1, \ldots, 2^k\}) = \{1, \ldots, 2^k\}$ and $ S_{1,\ldots,{k}}(\{2^k + 1, \ldots, 2^{k+1}\}) = \{2^k + 1, \ldots, 2^{k+1}\}$, but $A_{2^{k}}(\{1, \ldots, 2^k\})=\{2^k + 1, \ldots, 2^{k+1}\}$ and vice versa.
 By Lemma \ref{switch maps commute} and since $S_l \in \{(1), A_{2^{l-1}}\}$ for each $1 \leq l \leq 2^k$, we have 
\[
\begin{array}{lll}
((A_{2^{k}} \circ S_{1,\ldots,{k}}) \circ (A_{2^{k}}\circ S_{1,\ldots,{k}})) (n) & = & ((A_{2^k} \circ S_k \circ \cdots \circ S_1) \circ (A_{2^k} \circ S_k \circ \cdots \circ S_1))(n) \\
& = & ((A_{2^k} \circ A_{2^k}) \circ (S_k \circ S_k) \circ \cdots \circ (S_1 \circ S_1))(n) \\
& = & n.
\end{array}
 \]
Then -- recalling that a permutation has order 1 or 2 if and only if it is a product of disjoint transpositions -- since $A_{2^k} \circ S_{1, \ldots, k}$ moves all $n \in \{1, \ldots, 2^{k+1}\}$ and has order 2, it is a derangement by disjoint transpositions. This completes the induction step.   
	
\end{proof}

\begin{prop}\label{Cliff(Alt(2^k)) is a subgroup}
	For all $k \in \mathbb{N}$ with $k \geq 2$, Cliff$(Alt(2^{k}))$ is an Abelian subgroup of order $2^{k}$ of the alternating group $Alt(2^{k})\subset\mathcal{S}_{2^{k}}$.
\end{prop}

\begin{proof}
Note that we require $k \geq 2$ because in the special case $k=1$, $\operatorname{Cliff}(Alt(2^k)) = \mathcal{S}_2$ but $Alt(2^k) = \{(1)\}$. To see that every element of $\textrm{Cliff}(Alt(2^k))$ is an even permutation for $k \geq 2$: by Proposition \ref{derangement by disjoint transpositions}, we know that each map $A = S_{1, \ldots, k}$ is a derangement by disjoint transpositions. Then there exists a $J \subseteq \{1, \ldots, 2^k\}$ of size $2^{k-1}$ such that $A(J) = \{1, \ldots, 2^k\} \setminus J$. Then there are enumerations $\langle j_1, \ldots, j_{2^{k-1}} \rangle$ of $J$ and $\langle l_1, \ldots, l_{2^{k-1}} \rangle$ of $\{1, \ldots, 2^k\} \setminus J$ such that for each $1 \leq p \leq 2^{k-1}$, $A(j_p) = l_p$. Thus $A=(j_1 l_1) \cdots (j_{2^{k-1}} l_{2^{k-1}})$, which is the product of an even number (namely $2^{k-1}$) of transpositions. Thus $\textrm{Cliff}(Alt(2^k)) \subseteq Alt(2^k)$.
	
	We have put the identity permutation $(1) \in  \textrm{Cliff}(Alt(2^{k}))$, and we have already that $|\textrm{Cliff}(Alt(2^k))|=2^k$ and that every element of $\textrm{Cliff}(Alt(2^k))$ is an even permutation. Every permutation in $\textrm{Cliff}(Alt(2^k))$ is its own inverse, since they are all involutions. 
	
	For closure: let $A, B \in \textrm{Cliff}(Alt(2^k))$. Write $A=S_{1, \ldots, k} = S_{k} \circ \cdots \circ S_{1}$ where $S_{l} \in \{(1), A_{2^{l-1}}\}$ for $1 \leq l \leq k$, and write $B=R_{1, \ldots, k} = R_{k} \circ \cdots \circ R_{1}$ where $R_{l} \in \{(1), A_{2^{l-1}}\}$ for $1 \leq l \leq k$. Then by Lemma \ref{switch maps commute}, 
	\[
	\begin{array}{lll}
		A \circ B & = & S_{1, \ldots, k} \circ R_{1, \ldots, k} \\
		& = & (S_{k} \circ \cdots \circ S_{1}) \circ ( R_{k} \circ \cdots \circ R_{1}) \\
		& = & (S_{k} \circ R_{k}) \circ \cdots \circ (S_{1} \circ R_{1})
	\end{array}
	\]
	Fix $1 \leq l \leq k$ and consider $Q_l := S_{l} \circ R_{l}$. If both $S_{l}$ and $R_{l}$ are $(1)$ or if both  $S_{l}$ and $R_{l}$ are $A_{2^{l-1}}$, then $Q_l = (1)$; and if exactly one of $S_{l}$ and $R_{l}$ is $(1)$ (and the other is $A_{2^{l-1}}$), then $Q_l = A_{2^{l-1}}$. In any case, each $Q_l$ is either $(1)$ or $A_{2^{l-1}}$, so that $A \circ B = Q_k \circ \cdots \circ Q_1 \in \textrm{Cliff}(Alt(2^k))$. 
	
	Thus $\textrm{Cliff}(Alt(2^k))$ is a subgroup of $Alt(2^k)$ of size $2^k$.
	
	Lastly, for commutativity: note that, with $A=S_{1, \ldots, k}$ and $B = R_{1, \ldots, k}$ as above, by Lemma \ref{switch maps commute} we have 
	\[
	\begin{array}{lll}
		A \circ B & = & S_{1, \ldots, k} \circ R_{1, \ldots, k} \\
		& = & (S_k \circ \cdots \circ S_1) \circ (R_k \circ \cdots \circ R_1) \\
		& = & (R_k \circ \cdots \circ R_1) \circ (S_k \circ \cdots \circ S_1) \\
		& = & B \circ A.
	\end{array}
	\]
	
\end{proof}

\begin{prop}
	We have the following isomorphism: $\faktor{\hat{\Gamma}_{2k}}{\sim}\cong \mathbb{Z}_2\times \operatorname{Cliff}(Alt(2^{k}))$. For the matrix representations $\hat{\Gamma}_{2k}\xrightarrow{\rho} Aut(\C(2^{k}))$, we have $\faktor{\hat{\Gamma}_{2k}}{\sim}\cong \faktor{\rho(\hat{\Gamma}_{2k})}{\mathbb{Z}_2}$, where by $\faktor{\rho(\hat{\Gamma}_{2k})}{\mathbb{Z}_2}$ we mean the group of representations of generators modulo replacement of any elements of the matrices by their negatives.
	
\end{prop}

\begin{proof}
	
	From Theorem \ref{representing Clifford multiplication via permutations}, we have that all classes of Clifford multiplication actions on the 2-torsion points are represented by the $2^{k}$ strictly real induced Clifford permutations $S_{1, \ldots ,k}$ and the $2^k$ strictly imaginary induced Clifford permutations $i \cdot S_{1, \ldots ,k}$, giving us the $2^{k+1}$ classes mod $\sim$ in $\faktor{\hat{\Gamma}_{2k}}{\sim}$.
 We can then identify each $S_{1, \ldots, k}$ with $(\overline{0}, S_{1, \ldots, k}) \in \mathbb{Z}_2 \times \operatorname{Cliff}(Alt(2^k))$, and each $i \cdot S_{1, \ldots, k}$ with $(\overline{1}, S_{1, \ldots, k}) \in \mathbb{Z}_2 \times \operatorname{Cliff}(Alt(2^k))$. The group operation on $\mathbb{Z}_2 \times \operatorname{Cliff}(Alt(2^k))$ is mod $2$ addition on the $\mathbb{Z}_2$ components and composition on the $\operatorname{Cliff}(Alt(2^{k}))$ components.
 
 Lastly, our representative matrix actions on $J_2^{S_{\Delta_{2k}}}$ are unchanged when any of the entries in the matrices are replaced by their negatives. Then we can identify all $2^{k+1}$ classes $[e_\mu]$ of generators $e_\mu$ with matrices modulo negative signs, which we denote by $\faktor{\rho(\gh_{2k})}{\mathbb{Z}_2}$. Hence, with these identifications, we have that  $\faktor{\hat{\Gamma}_{2k}}{\sim}\cong \mathbb{Z}_2\times  \operatorname{Cliff}(Alt(2^{k}))\cong \faktor{\rho(\hat{\Gamma}_{2k})}{\mathbb{Z}_2}$.
\end{proof}

From this identification, we can view the types of Clifford actions on 2-torsion points as subgroups of order 2 of permutation groups, or as generic types of matrices that make up our classes. The benefit of these group isomorphisms is that we have Abelian matrix subgroups and an Abelian group of permutations that encode all of the information of Clifford multiplication acting on the ever-important 2-torsion points on the Dirac spinor Abelian variety $S_{\Delta_{2k}}$.

\section{Fixed points and translation constants of Clifford multiplication actions on $J_2^{S_{\Delta_{2k}}}$}\label{sec: fixed points and translation constants}

We next use properties of the induced Clifford permutations of generators $e_\mu \in \gh_{2k}$ to find the number of fixed points of each of our Clifford actions, as well as the number of their translation constants. 
 
\begin{defn}
A $2$-torsion point $\vec{v} \in J_2^{S_{\Delta_{2k}}}$ is called a \textbf{fixed point} of the induced Clifford permutation $A$ (or of its associated generator $e_\mu$) if $A \cdot \vec{v} = \vec{v}$. We denote by $\textrm{FP}(A)$ the set of fixed points of $A$:
\[\textrm{FP}(A) = \{\vec{v} \in J_2^{S_{\Delta_{2k}}}: A\cdot \vec{v} = \vec{v}\}\]
\end{defn}

We first compute the number of fixed points in $J_2^{S_{\Delta_{2k}}}$ of a generator $e_\mu \in \gh_{2k}$ of real type by viewing $e_\mu$ as its induced Clifford permutation.

\begin{prop}\label{number of fixed points}
Let $k \in \mathbb{N}, \, k \geq 1$. For any non-identity generator $e_\mu \in \gh_{2k}$ of real type, the lattice Clifford multiplication action of $e_\mu$ on $J_2^{S_{\Delta_{2k}}}$ has as its fixed points a subset of $J_2^{S_{\Delta_{2k}}}$ of size $2^{(2^k)} = |J_2^{S_{\Delta_{2k-2}}}|$. \end{prop}

\begin{proof}
Let $e_\mu \in \gh_{2k}$ be a non-identity generator of real type, and consider a fixed point $\vec{v} \in J_2^{S_{\Delta_{2k}}}$ of the action by $e_{\mu}$. By Proposition \ref{derangement by disjoint transpositions}, the induced Clifford  permutation $S_{1, \ldots, k}$ associated with $e_\mu$ is a derangement by disjoint transpositions of the $2^k$ indices in $\vec{v}$; so the entries in just half of the positions in the vector $\vec{v}$ determine those in the other half. This is for the following reason: let $m, n \in \{1, \ldots, 2^k\}$. If $S_{1,\ldots,{k}}$ maps $m$ to $n$ (and so also $n$ to $m$), then it must be that $\vec{v}[m] = \vec{v}[n]$; i.e. the $m$th and $n$th components of $\vec{v}$ are the same in this situation.  Then the number of fixed points of $e_{\mu}$ is given by the following: 
\[
|\FP(e_\mu)| = |\FP(S_{1, \ldots, k})| = (|J_2^{S_{\Delta_0}}|)^{(2^{k-1})} = 4^{(2^{k-1})} = (2^2)^{(2^{k-1})} = 2^{(2 \cdot 2^{k-1})} = 2^{(2^k)}.
\]   

Recall from Lemma \ref{number of points in J_2^{2^k}} that $| J_2^{S_{\Delta_{2k}}}|= 2^{(2^{k+1})}$ for any $k \in \mathbb{N}, \, k \geq 0$. This means that the generator $e_\mu$ acting on the 2-torsion points has as its fixed points a subset of $J_2^{S_{\Delta_{2k}}}$ of size $2^{(2^k)}= |J_2^{S_{\Delta_{2k-2}}}|$. 
\end{proof}

Next we turn our attention to generators $e_\mu$ of imaginary type, and to their associated strictly imaginary induced Clifford permutations $i \cdot S_{1, \ldots, k}$. Note that maps of the form $i \cdot S_{1,\ldots,k}$ not only interchange entries of the vector form of an element $\vec{v} \in J_2^{S_{\Delta_{2k}}}$; such maps also interchange the real and imaginary parts of the elements of $J^{S_{\Delta_0}}_2$ in each component of the $2^k$-vector  representing $\vec{v}$:
\[
\begin{array}{lllll}
i \cdot v_0 & = & i \cdot (0 + 0 i) & = & 0 + 0i \\
i \cdot v_1 & = & i \cdot (\frac{1}{2} + 0i) & = & 0 + \frac{1}{2}i \\
i \cdot v_2 & = & i \cdot (0 + \frac{1}{2}i) & = & \frac{1}{2} + 0i \\
i \cdot v_3 & = & i \cdot (\frac{1}{2} + \frac{1}{2}i) & = & \frac{1}{2} + \frac{1}{2}i
\end{array}
\]

We now find the number of fixed points of strictly imaginary induced Clifford permutations $i \cdot S_{1,\ldots,k}$ on $J_2^{S_{\Delta_{2k}}}$. We again write elements of $J_2^{S_{\Delta_{2k}}}$ as $2^k$-vectors $\vec{v} = \langle \vec{v}[1], \ldots, \vec{v}[2^k]  \rangle$ where $\vec{v}[j] \in J_2^{S_{\Delta_0}}$ for $1 \leq j \leq 2^k$.

\begin{prop}\label{count fixed points for iA}
For any $k \in \mathbb{N},\ k \geq 1$ and any $e_\mu \in \gh_{2k}$ of imaginary type, the lattice Clifford multiplication action of $e_\mu$ on $J_2^{S_{\Delta_{2k}}}$ has as its fixed points a subset of $J_2^{S_{\Delta_{2k}}}$ of size $2^{(2^k)} = |J_2^{S_{\Delta_{2k-2}}}|$.
\end{prop}

\begin{proof}

Let $e_\mu \in \gh_{2k}$ be of imaginary type, and suppose $i \cdot S_{1, \ldots, k}$ is the Clifford permutation induced by $e_\mu$. By Proposition \ref{derangement by disjoint transpositions}, $S_{1,\ldots,k}$ is a derangement by disjoint transpositions of $\{1, \ldots, 2^k\}$, so there is a $J \subseteq \{1, \ldots, 2^k\}$ of size $2^{k-1}$ so that $ S_{1,\ldots,k}(J) = \{1, \ldots, 2^k\} \setminus J$. 
Suppose $\vec{v} \in J_2^{S_{\Delta_{2k}}}$ satisfies the property \[\star \hspace{5mm} \textrm{ For all } n \in J, \ \vec{v}[n] = i \cdot \vec{v}[ S_{1,\ldots,k}(n)]\](equivalently, for all $n \in J$, $i \cdot \vec{v}[n] = \vec{v}[ S_{1,\ldots,k}(n)]$). We claim that $\vec{v}$ must then be a fixed point of $i \cdot S_{1,\ldots,k}$. Observe that 
 \[
 \begin{array}{lll}
 i \cdot S_{1,\ldots,k}\vec{v} & = & i \cdot S_{1,\ldots,k} \cdot \langle \vec{v}[n]: 1 \leq n \leq 2^k \rangle \\
 & = & i \cdot \langle \vec{v}[ S_{1,\ldots,k}(n)]: 1 \leq n \leq 2^k \rangle \\
 & = & \langle i \cdot \vec{v}[ S_{1,\ldots,k}(n)]: 1 \leq n \leq 2^k \rangle. 
 \end{array}
 \]
We show that $i \cdot \vec{v}[S_{1, \ldots, k}(n)] = \vec{v}[n]$ for all $1 \leq n \leq 2^k$. Fix $n \in \{1, \ldots, 2^k\}$. 

Case 1: $n \in J$. Then $i \cdot \vec{v}[ S_{1,\ldots ,k}(n)] = \vec{v}[n]$ by $\star$.

Case 2: $n \not\in J$. Choose $m \in J$ such that $ S_{1,\ldots,k}(m) = n$. Then 
\[
\begin{array}{lll}
i \cdot \vec{v}[S_{1, \ldots, k}(n)] & = & i \cdot \vec{v}[ S_{1,\ldots,k}( S_{1,\ldots,k}(m))] \\
& = & i \cdot \vec{v}[m] \textrm{ (as }  S_{1, \ldots ,k} \textrm{ is an involution)} \\
& = & i \cdot (i \cdot \vec{v}[ S_{1,\ldots,k}(m)]) \textrm{ (by $\star$, since $m \in J$)} \\
& = & \vec{v}[S_{1, \ldots, k}(m)] \textrm{ (as } \vec{v}[S_{1, \ldots, k}(m)] \in J_2^{S_{\Delta_0}}\textrm{)} \\
& = & \vec{v}[n].
\end{array}
\]
In either case, we have that $((i \cdot S_{1, \ldots, k}) \cdot \vec{v})[n] = i\cdot  \vec{v}[S_{1, \ldots, k}(n)] = \vec{v}[n]$ for all $1 \leq n \leq 2^k$, so that $(i \cdot S_{1, \ldots, k}) \cdot \vec{v} = \vec{v}$; that is, $\vec{v}$ is a fixed point of $i \cdot S_{1, \ldots, k}$, proving the claim.

Thus $\{\vec{v} \in J_2^{S_{\Delta_{2k}}}: \vec{v} \textrm{ satisfies } \star\} \subseteq \FP(i \cdot S_{1, \ldots, k})$. If $\vec{v}$ satisfies $\star$, then its entries in exactly half of the indices $1, \ldots, 2^k$ -- namely, the indices in the set $J$ -- determine the entries in the other half of the indices. This means, since each entry in the $2^k$-vector $\vec{v}$ is in $J_2^{S_{\Delta_{0}}}$ (which has 4 elements), that $|\{\vec{v} \in J_2^{S_{\Delta_{2k}}}: \vec{v} \textrm{ satisfies } \star\}| = 4^{|J|} = 4^{(2^{k-1})} = 2^{(2^k)}$. Therefore \[|\FP(i \cdot S_{1, \ldots, k})| \geq |\{\vec{v} \in J_2^{S_{\Delta_{2k}}}: \vec{v} \textrm{ satisfies } \star\}| = 2^{(2^k)}.\]

For the other inequality: now suppose $\vec{v}$ is a fixed point of $i \cdot S_{1, \ldots, k}$. Then for all $1 \leq n \leq 2^k$, $\vec{v}[n] = (i \cdot S_{1, \ldots, k} \cdot \vec{v})[n] =  i \cdot \vec{v}[S_{1, \ldots, k}(n)]$. In particular, this holds for all $n \in J$, so that $\vec{v}$ satisfies $\star$. Thus $\FP(i \cdot S_{1, \ldots, k}) \subseteq \{\vec{v} \in J_2^{S_{\Delta_{2k}}}: \vec{v} \textrm{ satisfies } \star\}$, and so \[|\FP(i \cdot S_{1, \ldots, k})| \leq |\{\vec{v} \in J_2^{S_{\Delta_{2k}}}: \vec{v} \textrm{ satisfies } \star\}| = 2^{(2^k)}.\] 
\end{proof}

\begin{defn}
Let $k \in \mathbb{N},\, k \geq 1,$ and let $A$ be the induced Clifford permutation of some generator $e_\mu \in \gh_{2k}$. A \textbf{translation constant} of $A$ (or of $e_\mu$) is a $\vec{v} \in J_2^{S_{\Delta_{2k}}}$ such that for some $\vec{w} \in J_2^{S_{\Delta_{2k}}}$, $A \cdot \vec{w} = \vec{v} + \vec{w}$. We denote by $\Tr(A)$ the set of translation constants of $A$: \[\Tr(A) = \{\vec{v} \in J_2^{S_{\Delta_{2k}}}: A \cdot \vec{w} = \vec{v} + \vec{w} \textrm{ for some } \vec{w} \in J_2^{S_{\Delta_{2k}}}\}.\]

For $\vec{w} \in J_2^{S_{\Delta_{2k}}}$, the (unique) translation constant for $\vec{w}$ of $A$ is denoted $\Tr_A(\vec{w})$; that is,  $\Tr_A(\vec{w})$ is the element of $J_2^{S_{\Delta_{2k}}}$ such that $A \cdot \vec{w} = \Tr_A(\vec{w}) + \vec{w}$. 

For $\vec{v} \in \Tr(A)$, we denote by $r_A(\vec{v})$ the \textbf{responsibility set} of $\vec{v}$ for $A$; that is, $$r_A(\vec{v}) =\{\vec{w} \in J_2^{S_{\Delta_{2k}}}: A \cdot \vec{w} = \vec{v} + \vec{w}\}.$$
\end{defn}

We show in Proposition \ref{FP is Tr} that $\FP(A) = \Tr(A)$ for any induced Clifford permutation $A$. To do this, we represent a vector $\vec{v} \in J_2^{S_{\Delta_{2k}}}$ as a $2^k \times 2$ matrix of 0s and 1s, where the $n$th row of the matrix contains the numerators of the real and imaginary parts, respectively, of the element of $J^{S_{\Delta_0}}_2\subset \dfrac{\C}{\mathbb{Z}\oplus i\mathbb{Z}}$ that is the $n$th component of $\vec{v}$. We will refer to this matrix as the \textit{matrix form} of the 2-torsion point $\vec{v}$.

\begin{exmp}\label{matrix form}
The matrix 
\[
\left[
\begin{array}{cc}
0 & 1 \\
1 & 1 \\
0 & 0 \\
1 & 0
\end{array}
\right]
\]
encodes the element
\[
\vec{v} = \left(
\begin{array}{c}
\\[-12pt]
\frac{0}{2} + \frac{1}{2}i \\[2pt]
\frac{1}{2} + \frac{1}{2}i \\[2pt]
\frac{0}{2} + \frac{0}{2}i \\[2pt]
\frac{1}{2} + \frac{0}{2}i \\[3pt]
\end{array}
\right)
=
\left(
\begin{array}{c}
	\\[-12pt]
	v_2 \\[2pt]
	v_3 \\[2pt]
	v_0 \\[2pt]
	v_1 \\[3pt]
\end{array}
\right)
\]
of $J_2^{S_{\Delta_4}}$. 
\end{exmp}

Recall that $J^{S_{\Delta_0}}_{2}$ is isomorphic to the Klein four-group: $J^{S_{\Delta_0}}_2\cong \mathbb{Z}_2\times \mathbb{Z}_2$. When two elements $\vec{v}$ and $\vec{w}$ of $J_2^{S_{\Delta_{2k}}}$ represented as matrices (as above) are added, the componentwise addition can be thought of as happening inside the group $\mathbb{Z}_2= \{\bar{0}, \bar{1}\} = \{0, 1\}$, with the addition rules $0 + 0 = 0$, $0 + 1 = 1$, $1 + 0 = 1$, and $1 + 1 = 0$. Also, considering scalar multiplication by $i$ as an action on the 2-torsion points $J_2^{S_{\Delta_0}}$, recall that this multiplication interchanges the real and imaginary parts of any $v \in J_2^{S_{\Delta_0}}$. This means that when a $\vec{v} \in J_2^{S_{\Delta_{2k}}}$ is represented in matrix form, scalar multiplication by $i$ is equivalent to interchanging the columns of that matrix form. Thus any strictly real induced Clifford permutation $S_{1, \ldots, k}$ can be thought of as an action that permutes only the rows of the matrix form of a $\vec{v} \in J_2^{S_{\Delta_{2k}}}$, while a strictly imaginary induced Clifford permutation $i \cdot S_{1, \ldots, k}$ both permutes the rows and interchanges the columns (or, in the case of the map $i \cdot (1)$, just interchanges the columns). We illustrate the 8 strictly imaginary induced Clifford permutations on $J_2^{S_{\Delta_6}}$ in Figure \ref{fig: complex, dimension 8}.  There the numbers 1 through 16 represent the entries $\vec{v}[1]$ through $\vec{v}[16]$ in the matrix form of an element $\vec{v} \in J_2^{S_{\Delta_6}}$, and the white and grey columns represent the real and imaginary entries, respectively, of the matrix form of $\vec{v}$.

\begin{figure}[h]
\centering

	
\begin{tikzpicture}[
	number/.style={rectangle, minimum size=5mm, inner sep = 1, draw=black, behind path}, grey/.style={rectangle, minimum size=5mm, inner sep = 1, draw=black, behind path, fill=black!10}]

	\matrix[row sep=0mm,column sep=0mm] {
		\node (l1) [number] {1}; &
		\node (r1) [grey] {2}; \\
		\node (l2) [number] {3}; &
		\node (r2) [grey] {4}; \\
		\node (l3) [number] {5}; &
		\node (r3) [grey] {6}; \\
		\node (l4) [number] {7}; &
		\node (r4) [grey] {8}; \\
		\node (l5) [number] {9}; &
		\node (r5) [grey] {10}; \\
		\node (l6) [number] {11}; &
		\node (r6) [grey] {12}; \\
		\node (l7) [number] {13}; &
		\node (r7) [grey] {14}; \\
		\node (l8) [number] {15}; &
		\node (r8) [grey] {16}; \\
	};

	\begin{scope}[yshift=0cm,xshift=2.5cm]
		\matrix[row sep=0mm,column sep=0mm] {
			\node (a1) [grey] {1}; &
			\node (b1) [number] {2}; \\
			\node (a2) [grey] {3}; &
			\node (b2) [number] {4}; \\
			\node (a3) [grey] {5}; &
			\node (b3) [number] {6}; \\
			\node (a4) [grey] {7}; &
			\node (b4) [number] {8}; \\
			\node (a5) [grey] {9}; &
			\node (b5) [number] {10}; \\
			\node (a6) [grey] {11}; &
			\node (b6) [number] {12}; \\
			\node (a7) [grey] {13}; &
			\node (b7) [number] {14}; \\
			\node (a8) [grey] {15}; &
			\node (b8) [number] {16}; \\
		};
	\end{scope}
	
	
	\node (x1) at (.51, 1.79375) {};
	\node (x2) at (.51, 1.28125) {};
	\node (x3) at (.51, 0.76875) {};
	\node (x4) at (.51, 0.25625) {};
	\node (x5) at (.51, -0.25625) {};
	\node (x6) at (.51, -0.76875) {};
	\node (x7) at (.51, -1.28125) {};
	\node (x8) at (.51, -1.79375) {};
	
	\node (y1) at (1.99, 1.79375) {};
	\node (y2) at (1.99, 1.28125) {};
	\node (y3) at (1.99, 0.76875) {};
	\node (y4) at (1.99, 0.25625) {};
	\node (y5) at (1.99, -0.25625) {};
	\node (y6) at (1.99, -0.76875) {};
	\node (y7) at (1.99, -1.28125) {};
	\node (y8) at (1.99, -1.79375) {};
	
	\draw [-{Stealth[length=1mm]}] (x1) -- (y1);
	\draw [-{Stealth[length=1mm]}] (x2) -- (y2);
	\draw [-{Stealth[length=1mm]}] (x3) -- (y3);
	\draw [-{Stealth[length=1mm]}] (x4) -- (y4);
	\draw [-{Stealth[length=1mm]}] (x5) -- (y5);
	\draw [-{Stealth[length=1mm]}] (x6) -- (y6);
	\draw [-{Stealth[length=1mm]}] (x7) -- (y7);
	\draw [-{Stealth[length=1mm]}] (x8) -- (y8);

	\node at (1.25,-2.7) {$i\, (1) \circ (1) \circ (1)$};
	
\end{tikzpicture}
\hspace{4mm}
\begin{tikzpicture}[
	number/.style={rectangle, minimum size=5mm, inner sep = 1, draw=black, behind path}, grey/.style={rectangle, minimum size=5mm, inner sep = 1, draw=black, behind path, fill=black!10}]

	\matrix[row sep=0mm,column sep=0mm] {
		\node (l1) [number] {1}; &
		\node (r1) [grey] {2}; \\
		\node (l2) [number] {3}; &
		\node (r2) [grey] {4}; \\
		\node (l3) [number] {5}; &
		\node (r3) [grey] {6}; \\
		\node (l4) [number] {7}; &
		\node (r4) [grey] {8}; \\
		\node (l5) [number] {9}; &
		\node (r5) [grey] {10}; \\
		\node (l6) [number] {11}; &
		\node (r6) [grey] {12}; \\
		\node (l7) [number] {13}; &
		\node (r7) [grey] {14}; \\
		\node (l8) [number] {15}; &
		\node (r8) [grey] {16}; \\
	};

	\begin{scope}[yshift=0cm,xshift=2.5cm]
		\matrix[row sep=0mm,column sep=0mm] {
			\node (a1) [grey] {4}; &
			\node (b1) [number] {3}; \\
			\node (a2) [grey] {2}; &
			\node (b2) [number] {1}; \\
			\node (a3) [grey] {8}; &
			\node (b3) [number] {7}; \\
			\node (a4) [grey] {6}; &
			\node (b4) [number] {5}; \\
			\node (a5) [grey] {12}; &
			\node (b5) [number] {11}; \\
			\node (a6) [grey] {10}; &
			\node (b6) [number] {9}; \\
			\node (a7) [grey] {16}; &
			\node (b7) [number] {15}; \\
			\node (a8) [grey] {14}; &
			\node (b8) [number] {13}; \\
		};
	\end{scope}
	
	
	\node (x1) at (.51, 1.79375) {};
	\node (x2) at (.51, 1.28125) {};
	\node (x3) at (.51, 0.76875) {};
	\node (x4) at (.51, 0.25625) {};
	\node (x5) at (.51, -0.25625) {};
	\node (x6) at (.51, -0.76875) {};
	\node (x7) at (.51, -1.28125) {};
	\node (x8) at (.51, -1.79375) {};
	
	\node (y1) at (1.99, 1.79375) {};
	\node (y2) at (1.99, 1.28125) {};
	\node (y3) at (1.99, 0.76875) {};
	\node (y4) at (1.99, 0.25625) {};
	\node (y5) at (1.99, -0.25625) {};
	\node (y6) at (1.99, -0.76875) {};
	\node (y7) at (1.99, -1.28125) {};
	\node (y8) at (1.99, -1.79375) {};
	
	\draw [-{Stealth[length=1mm]}] (x1) -- (y2);
	\draw [-{Stealth[length=1mm]}] (x2) -- (y1);
	\draw [-{Stealth[length=1mm]}] (x3) -- (y4);
	\draw [-{Stealth[length=1mm]}] (x4) -- (y3);
	\draw [-{Stealth[length=1mm]}] (x5) -- (y6);
	\draw [-{Stealth[length=1mm]}] (x6) -- (y5);
	\draw [-{Stealth[length=1mm]}] (x8) -- (y7);
	\draw [-{Stealth[length=1mm]}] (x7) -- (y8);

	\node at (1.25,-2.7) {$i\, (1) \circ (1) \circ A_1$};
	
\end{tikzpicture}
\hspace{4mm}
\begin{tikzpicture}[
	number/.style={rectangle, minimum size=5mm, inner sep = 1, draw=black, behind path}, grey/.style={rectangle, minimum size=5mm, inner sep = 1, draw=black, behind path, fill=black!10}]

	\matrix[row sep=0mm,column sep=0mm] {
		\node (l1) [number] {1}; &
		\node (r1) [grey] {2}; \\
		\node (l2) [number] {3}; &
		\node (r2) [grey] {4}; \\
		\node (l3) [number] {5}; &
		\node (r3) [grey] {6}; \\
		\node (l4) [number] {7}; &
		\node (r4) [grey] {8}; \\
		\node (l5) [number] {9}; &
		\node (r5) [grey] {10}; \\
		\node (l6) [number] {11}; &
		\node (r6) [grey] {12}; \\
		\node (l7) [number] {13}; &
		\node (r7) [grey] {14}; \\
		\node (l8) [number] {15}; &
		\node (r8) [grey] {16}; \\
	};

	\begin{scope}[yshift=0cm,xshift=2.5cm]
		\matrix[row sep=0mm,column sep=0mm] {
			\node (a1) [grey] {6}; &
			\node (b1) [number] {5}; \\
			\node (a2) [grey] {8}; &
			\node (b2) [number] {7}; \\
			\node (a3) [grey] {2}; &
			\node (b3) [number] {1}; \\
			\node (a4) [grey] {4}; &
			\node (b4) [number] {3}; \\
			\node (a5) [grey] {14}; &
			\node (b5) [number] {13}; \\
			\node (a6) [grey] {16}; &
			\node (b6) [number] {15}; \\
			\node (a7) [grey] {10}; &
			\node (b7) [number] {9}; \\
			\node (a8) [grey] {12}; &
			\node (b8) [number] {11}; \\
		};
	\end{scope}
	
	
	\node (x1) at (.51, 1.79375) {};
	\node (x2) at (.51, 1.28125) {};
	\node (x3) at (.51, 0.76875) {};
	\node (x4) at (.51, 0.25625) {};
	\node (x5) at (.51, -0.25625) {};
	\node (x6) at (.51, -0.76875) {};
	\node (x7) at (.51, -1.28125) {};
	\node (x8) at (.51, -1.79375) {};
	
	\node (y1) at (1.99, 1.79375) {};
	\node (y2) at (1.99, 1.28125) {};
	\node (y3) at (1.99, 0.76875) {};
	\node (y4) at (1.99, 0.25625) {};
	\node (y5) at (1.99, -0.25625) {};
	\node (y6) at (1.99, -0.76875) {};
	\node (y7) at (1.99, -1.28125) {};
	\node (y8) at (1.99, -1.79375) {};
	
	\draw [-{Stealth[length=1mm]}] (x1) -- (y3);
	\draw [-{Stealth[length=1mm]}] (x2) -- (y4);
	\draw [-{Stealth[length=1mm]}] (x3) -- (y1);
	\draw [-{Stealth[length=1mm]}] (x4) -- (y2);
	\draw [-{Stealth[length=1mm]}] (x5) -- (y7);
	\draw [-{Stealth[length=1mm]}] (x6) -- (y8);
	\draw [-{Stealth[length=1mm]}] (x7) -- (y5);
	\draw [-{Stealth[length=1mm]}] (x8) -- (y6);

	\node at (1.25,-2.7) {$i\, (1) \circ A_2 \circ (1)$};
	
\end{tikzpicture}
\vspace{4mm}

\begin{tikzpicture}[
	number/.style={rectangle, minimum size=5mm, inner sep = 1, draw=black, behind path}, grey/.style={rectangle, minimum size=5mm, inner sep = 1, draw=black, behind path, fill=black!10}]

	\matrix[row sep=0mm,column sep=0mm] {
		\node (l1) [number] {1}; &
		\node (r1) [grey] {2}; \\
		\node (l2) [number] {3}; &
		\node (r2) [grey] {4}; \\
		\node (l3) [number] {5}; &
		\node (r3) [grey] {6}; \\
		\node (l4) [number] {7}; &
		\node (r4) [grey] {8}; \\
		\node (l5) [number] {9}; &
		\node (r5) [grey] {10}; \\
		\node (l6) [number] {11}; &
		\node (r6) [grey] {12}; \\
		\node (l7) [number] {13}; &
		\node (r7) [grey] {14}; \\
		\node (l8) [number] {15}; &
		\node (r8) [grey] {16}; \\
	};

	\begin{scope}[yshift=0cm,xshift=2.5cm]
		\matrix[row sep=0mm,column sep=0mm] {
			\node (a1) [grey] {10}; &
			\node (b1) [number] {9}; \\
			\node (a2) [grey] {12}; &
			\node (b2) [number] {11}; \\
			\node (a3) [grey] {14}; &
			\node (b3) [number] {13}; \\
			\node (a4) [grey] {16}; &
			\node (b4) [number] {15}; \\
			\node (a5) [grey] {2}; &
			\node (b5) [number] {1}; \\
			\node (a6) [grey] {4}; &
			\node (b6) [number] {3}; \\
			\node (a7) [grey] {6}; &
			\node (b7) [number] {5}; \\
			\node (a8) [grey] {8}; &
			\node (b8) [number] {7}; \\
		};
	\end{scope}
	
	
	\node (x1) at (.51, 1.79375) {};
	\node (x2) at (.51, 1.28125) {};
	\node (x3) at (.51, 0.76875) {};
	\node (x4) at (.51, 0.25625) {};
	\node (x5) at (.51, -0.25625) {};
	\node (x6) at (.51, -0.76875) {};
	\node (x7) at (.51, -1.28125) {};
	\node (x8) at (.51, -1.79375) {};
	
	\node (y1) at (1.99, 1.79375) {};
	\node (y2) at (1.99, 1.28125) {};
	\node (y3) at (1.99, 0.76875) {};
	\node (y4) at (1.99, 0.25625) {};
	\node (y5) at (1.99, -0.25625) {};
	\node (y6) at (1.99, -0.76875) {};
	\node (y7) at (1.99, -1.28125) {};
	\node (y8) at (1.99, -1.79375) {};
	
	\draw [-{Stealth[length=1mm]}] (x1) -- (y5);
	\draw [-{Stealth[length=1mm]}] (x2) -- (y6);
	\draw [-{Stealth[length=1mm]}] (x3) -- (y7);
	\draw [-{Stealth[length=1mm]}] (x4) -- (y8);
	\draw [-{Stealth[length=1mm]}] (x5) -- (y1);
	\draw [-{Stealth[length=1mm]}] (x6) -- (y2);
	\draw [-{Stealth[length=1mm]}] (x7) -- (y3);
	\draw [-{Stealth[length=1mm]}] (x8) -- (y4);

	\node at (1.25,-2.7) {$i\, A_4 \circ (1) \circ (1)$};
	
\end{tikzpicture}
\hspace{4mm}
\begin{tikzpicture}[
	number/.style={rectangle, minimum size=5mm, inner sep = 1, draw=black, behind path}, grey/.style={rectangle, minimum size=5mm, inner sep = 1, draw=black, behind path, fill=black!10}]

	\matrix[row sep=0mm,column sep=0mm] {
		\node (l1) [number] {1}; &
		\node (r1) [grey] {2}; \\
		\node (l2) [number] {3}; &
		\node (r2) [grey] {4}; \\
		\node (l3) [number] {5}; &
		\node (r3) [grey] {6}; \\
		\node (l4) [number] {7}; &
		\node (r4) [grey] {8}; \\
		\node (l5) [number] {9}; &
		\node (r5) [grey] {10}; \\
		\node (l6) [number] {11}; &
		\node (r6) [grey] {12}; \\
		\node (l7) [number] {13}; &
		\node (r7) [grey] {14}; \\
		\node (l8) [number] {15}; &
		\node (r8) [grey] {16}; \\
	};

	\begin{scope}[yshift=0cm,xshift=2.5cm]
		\matrix[row sep=0mm,column sep=0mm] {
			\node (a1) [grey] {8}; &
			\node (b1) [number] {7}; \\
			\node (a2) [grey] {6}; &
			\node (b2) [number] {5}; \\
			\node (a3) [grey] {4}; &
			\node (b3) [number] {3}; \\
			\node (a4) [grey] {2}; &
			\node (b4) [number] {1}; \\
			\node (a5) [grey] {16}; &
			\node (b5) [number] {15}; \\
			\node (a6) [grey] {14}; &
			\node (b6) [number] {13}; \\
			\node (a7) [grey] {12}; &
			\node (b7) [number] {11}; \\
			\node (a8) [grey] {10}; &
			\node (b8) [number] {9}; \\
		};
	\end{scope}
	
	
	\node (x1) at (.51, 1.79375) {};
	\node (x2) at (.51, 1.28125) {};
	\node (x3) at (.51, 0.76875) {};
	\node (x4) at (.51, 0.25625) {};
	\node (x5) at (.51, -0.25625) {};
	\node (x6) at (.51, -0.76875) {};
	\node (x7) at (.51, -1.28125) {};
	\node (x8) at (.51, -1.79375) {};
	
	\node (y1) at (1.99, 1.79375) {};
	\node (y2) at (1.99, 1.28125) {};
	\node (y3) at (1.99, 0.76875) {};
	\node (y4) at (1.99, 0.25625) {};
	\node (y5) at (1.99, -0.25625) {};
	\node (y6) at (1.99, -0.76875) {};
	\node (y7) at (1.99, -1.28125) {};
	\node (y8) at (1.99, -1.79375) {};
	
	\draw [-{Stealth[length=1mm]}] (x1) -- (y4);
	\draw [-{Stealth[length=1mm]}] (x2) -- (y3);
	\draw [-{Stealth[length=1mm]}] (x3) -- (y2);
	\draw [-{Stealth[length=1mm]}] (x4) -- (y1);
	\draw [-{Stealth[length=1mm]}] (x5) -- (y8);
	\draw [-{Stealth[length=1mm]}] (x6) -- (y7);
	\draw [-{Stealth[length=1mm]}] (x7) -- (y6);
	\draw [-{Stealth[length=1mm]}] (x8) -- (y5);

	\node at (1.25,-2.7) {$i\, (1) \circ A_2 \circ A_1$};
	
\end{tikzpicture}
\hspace{4mm}
\begin{tikzpicture}[
	number/.style={rectangle, minimum size=5mm, inner sep = 1, draw=black, behind path}, grey/.style={rectangle, minimum size=5mm, inner sep = 1, draw=black, behind path, fill=black!10}]

	\matrix[row sep=0mm,column sep=0mm] {
		\node (l1) [number] {1}; &
		\node (r1) [grey] {2}; \\
		\node (l2) [number] {3}; &
		\node (r2) [grey] {4}; \\
		\node (l3) [number] {5}; &
		\node (r3) [grey] {6}; \\
		\node (l4) [number] {7}; &
		\node (r4) [grey] {8}; \\
		\node (l5) [number] {9}; &
		\node (r5) [grey] {10}; \\
		\node (l6) [number] {11}; &
		\node (r6) [grey] {12}; \\
		\node (l7) [number] {13}; &
		\node (r7) [grey] {14}; \\
		\node (l8) [number] {15}; &
		\node (r8) [grey] {16}; \\
	};

	\begin{scope}[yshift=0cm,xshift=2.5cm]
		\matrix[row sep=0mm,column sep=0mm] {
			\node (a1) [grey] {12}; &
			\node (b1) [number] {11}; \\
			\node (a2) [grey] {10}; &
			\node (b2) [number] {9}; \\
			\node (a3) [grey] {16}; &
			\node (b3) [number] {15}; \\
			\node (a4) [grey] {14}; &
			\node (b4) [number] {13}; \\
			\node (a5) [grey] {4}; &
			\node (b5) [number] {3}; \\
			\node (a6) [grey] {2}; &
			\node (b6) [number] {1}; \\
			\node (a7) [grey] {8}; &
			\node (b7) [number] {7}; \\
			\node (a8) [grey] {6}; &
			\node (b8) [number] {5}; \\
		};
	\end{scope}
	
	
	\node (x1) at (.51, 1.79375) {};
	\node (x2) at (.51, 1.28125) {};
	\node (x3) at (.51, 0.76875) {};
	\node (x4) at (.51, 0.25625) {};
	\node (x5) at (.51, -0.25625) {};
	\node (x6) at (.51, -0.76875) {};
	\node (x7) at (.51, -1.28125) {};
	\node (x8) at (.51, -1.79375) {};
	
	\node (y1) at (1.99, 1.79375) {};
	\node (y2) at (1.99, 1.28125) {};
	\node (y3) at (1.99, 0.76875) {};
	\node (y4) at (1.99, 0.25625) {};
	\node (y5) at (1.99, -0.25625) {};
	\node (y6) at (1.99, -0.76875) {};
	\node (y7) at (1.99, -1.28125) {};
	\node (y8) at (1.99, -1.79375) {};
	
	\draw [-{Stealth[length=1mm]}] (x1) -- (y6);
	\draw [-{Stealth[length=1mm]}] (x2) -- (y5);
	\draw [-{Stealth[length=1mm]}] (x3) -- (y8);
	\draw [-{Stealth[length=1mm]}] (x4) -- (y7);
	\draw [-{Stealth[length=1mm]}] (x5) -- (y2);
	\draw [-{Stealth[length=1mm]}] (x6) -- (y1);
	\draw [-{Stealth[length=1mm]}] (x7) -- (y4);
	\draw [-{Stealth[length=1mm]}] (x8) -- (y3);

	\node at (1.25,-2.7) {$i\, A_4 \circ (1) \circ A_1$};
	
\end{tikzpicture}
\vspace{4mm}

\begin{center}
\begin{tikzpicture}[
	number/.style={rectangle, minimum size=5mm, inner sep = 1, draw=black, behind path}, grey/.style={rectangle, minimum size=5mm, inner sep = 1, draw=black, behind path, fill=black!10}]

	\matrix[row sep=0mm,column sep=0mm] {
		\node (l1) [number] {1}; &
		\node (r1) [grey] {2}; \\
		\node (l2) [number] {3}; &
		\node (r2) [grey] {4}; \\
		\node (l3) [number] {5}; &
		\node (r3) [grey] {6}; \\
		\node (l4) [number] {7}; &
		\node (r4) [grey] {8}; \\
		\node (l5) [number] {9}; &
		\node (r5) [grey] {10}; \\
		\node (l6) [number] {11}; &
		\node (r6) [grey] {12}; \\
		\node (l7) [number] {13}; &
		\node (r7) [grey] {14}; \\
		\node (l8) [number] {15}; &
		\node (r8) [grey] {16}; \\
	};

	\begin{scope}[yshift=0cm,xshift=2.5cm]
		\matrix[row sep=0mm,column sep=0mm] {
			\node (a1) [grey] {14}; &
			\node (b1) [number] {13}; \\
			\node (a2) [grey] {16}; &
			\node (b2) [number] {15}; \\
			\node (a3) [grey] {10}; &
			\node (b3) [number] {9}; \\
			\node (a4) [grey] {12}; &
			\node (b4) [number] {11}; \\
			\node (a5) [grey] {6}; &
			\node (b5) [number] {5}; \\
			\node (a6) [grey] {8}; &
			\node (b6) [number] {7}; \\
			\node (a7) [grey] {2}; &
			\node (b7) [number] {1}; \\
			\node (a8) [grey] {4}; &
			\node (b8) [number] {3}; \\
		};
	\end{scope}
	
	
	\node (x1) at (.51, 1.79375) {};
	\node (x2) at (.51, 1.28125) {};
	\node (x3) at (.51, 0.76875) {};
	\node (x4) at (.51, 0.25625) {};
	\node (x5) at (.51, -0.25625) {};
	\node (x6) at (.51, -0.76875) {};
	\node (x7) at (.51, -1.28125) {};
	\node (x8) at (.51, -1.79375) {};
	
	\node (y1) at (1.99, 1.79375) {};
	\node (y2) at (1.99, 1.28125) {};
	\node (y3) at (1.99, 0.76875) {};
	\node (y4) at (1.99, 0.25625) {};
	\node (y5) at (1.99, -0.25625) {};
	\node (y6) at (1.99, -0.76875) {};
	\node (y7) at (1.99, -1.28125) {};
	\node (y8) at (1.99, -1.79375) {};
	
	\draw [-{Stealth[length=1mm]}] (x1) -- (y7);
	\draw [-{Stealth[length=1mm]}] (x2) -- (y8);
	\draw [-{Stealth[length=1mm]}] (x3) -- (y5);
	\draw [-{Stealth[length=1mm]}] (x4) -- (y6);
	\draw [-{Stealth[length=1mm]}] (x5) -- (y3);
	\draw [-{Stealth[length=1mm]}] (x6) -- (y4);
	\draw [-{Stealth[length=1mm]}] (x7) -- (y1);
	\draw [-{Stealth[length=1mm]}] (x8) -- (y2);

	\node at (1.25,-2.7) {$i\, A_4 \circ A_2 \circ (1)$};
	
\end{tikzpicture}
\hspace{4mm}
\begin{tikzpicture}[
	number/.style={rectangle, minimum size=5mm, inner sep = 1, draw=black, behind path}, grey/.style={rectangle, minimum size=5mm, inner sep = 1, draw=black, behind path, fill=black!10}]

	\matrix[row sep=0mm,column sep=0mm] {
		\node (l1) [number] {1}; &
		\node (r1) [grey] {2}; \\
		\node (l2) [number] {3}; &
		\node (r2) [grey] {4}; \\
		\node (l3) [number] {5}; &
		\node (r3) [grey] {6}; \\
		\node (l4) [number] {7}; &
		\node (r4) [grey] {8}; \\
		\node (l5) [number] {9}; &
		\node (r5) [grey] {10}; \\
		\node (l6) [number] {11}; &
		\node (r6) [grey] {12}; \\
		\node (l7) [number] {13}; &
		\node (r7) [grey] {14}; \\
		\node (l8) [number] {15}; &
		\node (r8) [grey] {16}; \\
	};

	\begin{scope}[yshift=0cm,xshift=2.5cm]
		\matrix[row sep=0mm,column sep=0mm] {
			\node (a1) [grey] {16}; &
			\node (b1) [number] {15}; \\
			\node (a2) [grey] {14}; &
			\node (b2) [number] {13}; \\
			\node (a3) [grey] {12}; &
			\node (b3) [number] {11}; \\
			\node (a4) [grey] {10}; &
			\node (b4) [number] {9}; \\
			\node (a5) [grey] {8}; &
			\node (b5) [number] {7}; \\
			\node (a6) [grey] {6}; &
			\node (b6) [number] {5}; \\
			\node (a7) [grey] {4}; &
			\node (b7) [number] {3}; \\
			\node (a8) [grey] {2}; &
			\node (b8) [number] {1}; \\
		};
	\end{scope}
	
	
	\node (x1) at (.51, 1.79375) {};
	\node (x2) at (.51, 1.28125) {};
	\node (x3) at (.51, 0.76875) {};
	\node (x4) at (.51, 0.25625) {};
	\node (x5) at (.51, -0.25625) {};
	\node (x6) at (.51, -0.76875) {};
	\node (x7) at (.51, -1.28125) {};
	\node (x8) at (.51, -1.79375) {};
	
	\node (y1) at (1.99, 1.79375) {};
	\node (y2) at (1.99, 1.28125) {};
	\node (y3) at (1.99, 0.76875) {};
	\node (y4) at (1.99, 0.25625) {};
	\node (y5) at (1.99, -0.25625) {};
	\node (y6) at (1.99, -0.76875) {};
	\node (y7) at (1.99, -1.28125) {};
	\node (y8) at (1.99, -1.79375) {};
	
	\draw [-{Stealth[length=1mm]}] (x1) -- (y8);
	\draw [-{Stealth[length=1mm]}] (x2) -- (y7);
	\draw [-{Stealth[length=1mm]}] (x3) -- (y6);
	\draw [-{Stealth[length=1mm]}] (x4) -- (y5);
	\draw [-{Stealth[length=1mm]}] (x5) -- (y4);
	\draw [-{Stealth[length=1mm]}] (x6) -- (y3);
	\draw [-{Stealth[length=1mm]}] (x7) -- (y2);
	\draw [-{Stealth[length=1mm]}] (x8) -- (y1);

	\node at (1.25,-2.7) {$i\, A_4 \circ A_2 \circ A_1$};
	
\end{tikzpicture}

\end{center}

\caption{The 8 imaginary classes of Clifford multiplication on $J_2^{S_{\Delta_{6}}}$}
\label{fig: complex, dimension 8}
\end{figure}
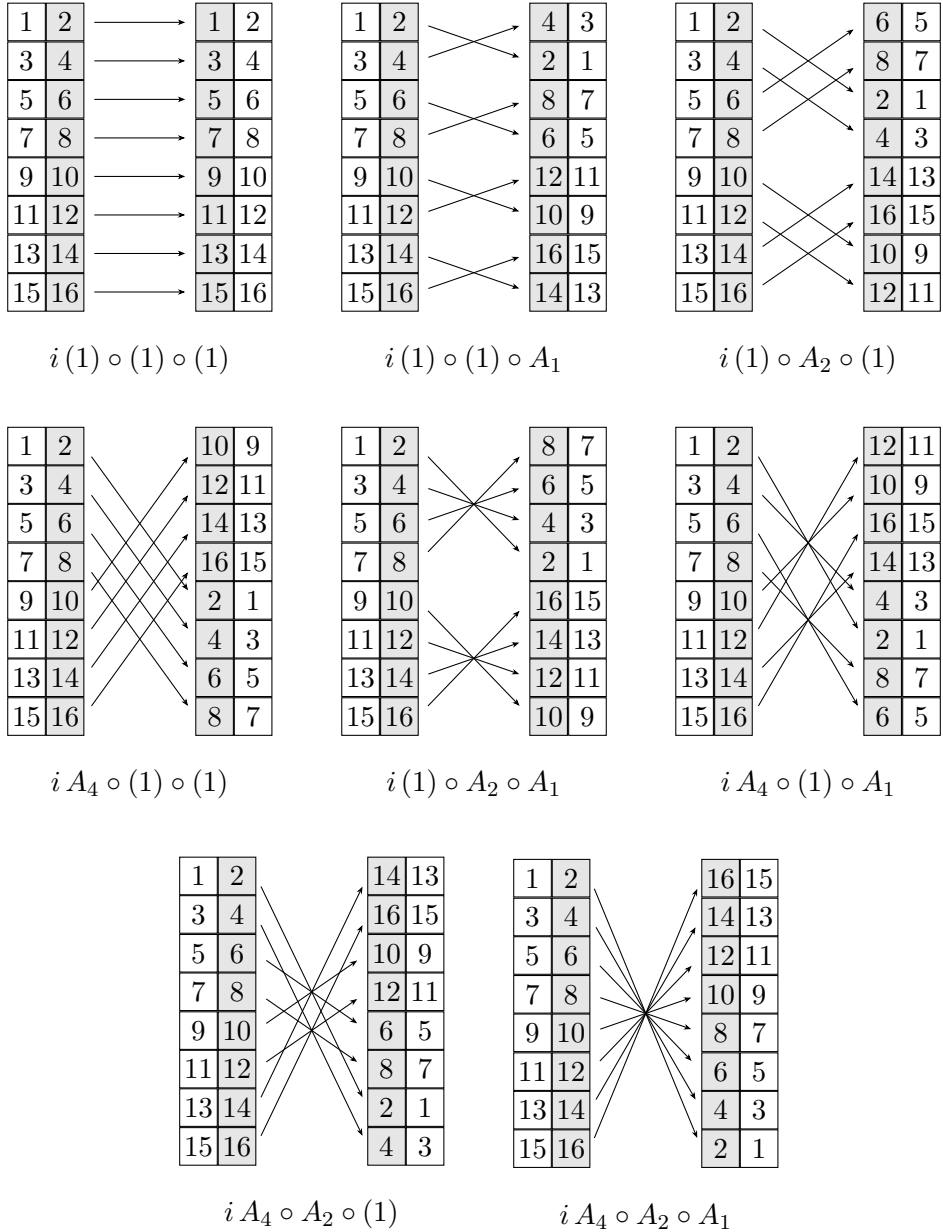

We are now ready to prove that the set of fixed points of a generator $e_\mu$ acting on the 2-torsion points is equal to the set of its translation constants.

\begin{prop}\label{FP is Tr}
For all $k \in \mathbb{N}, \, k \geq 1$, and all non-identity induced Clifford permutations $A$ acting on $J_2^{S_{\Delta_{2k}}}$, $$\Tr(A) = \textrm{FP}(A).$$

\end{prop}

\begin{proof}
Fix a generator $e_\mu \in \hat{\Gamma}_{2k}$ and let $A$ be the induced Clifford permutation for $e_\mu$. In the case where $A=S_{1, \ldots, k}$ is a strictly real induced Clifford permutation (i.e.\ when $e_\mu$ is of real type), we have characterized fixed points of $A$ by way of a permutation on the entries of a vector $\vec{w} \in J_2^{S_{\Delta_{2k}}}$. That is, in this case $A \in \mathcal{S}_{2^k}$. However, since here we wish to consider all induced Clifford permutations (both strictly real and imaginary ones) at once, it is convenient to characterize the fixed points of the map $A$ by way of a permutation $\eta_A$ on the $2^{k+1}$ indices of the $2^k \times 2$ matrix form of $A$; that is, by an element $\eta_A \in \mathcal{S}_{2^{k+1}}$. 

To do this, we label the $2^{k+1}$ indices of the matrix representing $\vec{w} \in J_2^{S_{\Delta_{2k}}}$ as follows:
\[
\left[
\begin{array}{cc}
\vec{w}[1] & \vec{w}[2] \\
\vec{w}[3] & \vec{w}[4] \\
\vdots & \vdots \\
\vec{w}[2^{k+1}-1] & \vec{w}[2^{k+1}]
\end{array}
\right]
\]
where each $\vec{w}[l]$ is either 0 or 1. If $e_\mu$ is of real type, then $A$ is a permutation of order 2 on $\{1, \ldots, 2^k\}$, since $e_\mu$ is an involution on $J_2^{S_{\Delta_{2k}}}$. In this case we define $\eta_A$ simply as $A$, but thought of as acting on the $2^k$ rows of the matrix form of an element $\vec{w}$ of $J_2^{S_{\Delta_{2k}}}$. If $e_\mu$ is of imaginary type, then $A$ will act as a permutation of the rows of the matrix form of an element $\vec{w}$ of $J_2^{S_{\Delta_{2k}}}$ followed by an interchange of the columns of that matrix form. In this case we define $\eta_A$ to be the composition of those two actions. Either way, $\eta_A$ can actually be considered as a permutation on $\{1, \ldots, 2^{k+1}\}$.

The permutation $\eta_A$ corresponding to the map $A$ then uniquely characterizes the form of a generic fixed point $\vec{z}$ of $A$: $\eta_A \in \mathcal{S}_{2^{k+1}}$ is the (unique) permutation such that whenever $\vec{z}$ is a fixed point of $A$ and $l \in \{1, \ldots, 2^{k+1}\}$, $\vec{z}[l] = \vec{z}[\eta_A(l)]$. It follows from Proposition \ref{derangement by disjoint transpositions} that $\eta_A$ is a derangement by disjoint transpositions, and so there are exactly $2^k$ such relations $\vec{z}[l] = \vec{z}[\eta_A(l)]$; each one is listed twice as $l$ ranges over $\{1, \ldots, 2^{k+1}\}$. Let $\vec{w} \in J_2^{S_{\Delta_{2k}}}$. We construct an element $\vec{v}$ (which will be the translation constant for $\vec{w}$ under $A$) as follows: for $l \in \{1, \ldots, 2^{k+1}\}$, define both $\vec{v}[l]$ and $\vec{v}[\eta_A(l)]$ by
\[
\star \hspace{10mm}
\vec{v}[l] = \vec{v}[\eta_A(l)] = \left\{
\begin{array}{lll}
0, & \textrm{ if } & \vec{w}[l] = \vec{w}[\eta_A(l)] \\
1, & \textrm{ if } & \vec{w}[l] \neq \vec{w}[\eta_A(l)]
\end{array}
\right.
\]
This is well-defined, since $\eta_A(\eta_A(l)) = l$ for each $l \in \{1, \ldots, 2^{k+1}\}$. The meaning of the two parts of the definition of $\vec{v}$ is this: if $\vec{w}$ behaved like a fixed point of 
$A$ at least at the indices $l$ and $\eta_A(l)$, then $\vec{v}$ is set to 0 at those indices; and if $\vec{w}$ did not behave like a fixed point of $A$ at the indices $l$ and $\eta_A(l)$, then $\vec{v}$ is set to 1 at those indices. 
By construction, $\vec{v}$ is a fixed point of $A$, as for every $l \in \{1, \ldots, 2^{k+1}\}$, either $\vec{v}[l] = \vec{v}[\eta_A(l)] = 0$, or $\vec{v}[l] = \vec{v}[\eta_A(l)] = 1$.

\noindent \textit{Claim 1:}  
$\vec{v}$ as constructed above is the translation constant for $\vec{w}$ of $A$; that is, $A \cdot \vec{w} = \vec{v} + \vec{w}$.

\noindent{\textit{Proof of Claim 1:}}
Fix $l \in \{1, \ldots, 2^{k+1}\}$. We need to show that $(A \cdot \vec{w})[l] = (\vec{v} + \vec{w})[l]$. We have  $(A \cdot \vec{w})[l] = \eta_A(\vec{w}[l]) = \vec{w}[\eta_A(l)]$. Also, $(\vec{v} + \vec{w})[l] = \vec{v}[l] + \vec{w}[l]$, where this lattice addition is based on the group $\mathbb{Z}_2=\{0,1\}$ as above. There are two cases to consider, depending on whether or not $\vec{w}$ acted like a fixed point of $A$ at the indices $l$ and $\eta_A(l)$.

Case 1: Suppose $\vec{w}[l] = \vec{w}[\eta_A(l)]$.  Then in $\star$ above, we put $\vec{v}[l] = \vec{v}[\eta_A(l)] = 0$, so that $\vec{w}[\eta_A(l)] = \vec{w}[l] = 0 + \vec{w}[l] = \vec{v}[l] + \vec{w}[l]$.

Case 2: Suppose $\vec{w}[l] \neq \vec{w}[\eta_A(l)]$. Then $\vec{w}[\eta_A(l)] = 1 + \vec{w}[l]$ (in $\mathbb{Z}_2$); and since we put $\vec{v}[l] = 1$ in $\star$ above, this means that $\vec{w}[\eta_A(l)] = \vec{v}[l] + \vec{w}[l]$.

This proves Claim 1. Thus for any $\vec{w} \in J_2^{S_{\Delta_{2k}}}$, the element $\vec{v}$ constructed in $\star$ above is the translation constant for $\vec{w}$ of $A$ -- that is, $\vec{v} = \Tr_A(\vec{w})$.

We can now prove that $\Tr(A) = \mathrm{FP}(A)$. Suppose $\vec{v} \in \Tr(A)$. Then $\vec{v} = \Tr_A(\vec{w})$ for some $\vec{w} \in J_2^{S_{\Delta_{2k}}}$, so that (by uniqueness of translation constants and by the construction above) $\vec{v}$ is a fixed point of $A$. Thus $\Tr(A) \subseteq \mathrm{FP}(A)$. Recall that any fixed point $\vec{z}$ of $A$ satisfies $2^k$-many relations of the form $\vec{z}[l] = \vec{z}[\eta_A(l)]$. Each of these relations must be satisfied by any translation constant for $A$, as $\Tr(A) \subseteq \mathrm{FP}(A)$. These $2^k$-many relations could be satisfied in $2^{(2^k)}$-many ways: a given translation constant $\vec{v}$ might have $\vec{v}[l] = \vec{v}[\eta_A(l)] = 0$ or $\vec{v}[l] = \vec{v}[\eta_A(l)] = 1$ (and all of those ways to satisfy the relations are necessary, since every $2^k \times 2$ matrix with entries in $\{0, 1\}$ represents a 2-torsion point in dimension $2^k$). Thus $|\Tr(A)| \geq 2^{(2^k)}$. Since $\Tr(A) \subseteq \mathrm{FP}(A)$ and we already know that $|\textrm{FP}(A)|= 2^{(2^k)}$, we have $\Tr(A) = \mathrm{FP}(A)$. 
\end{proof}

\begin{remark}
The only property of the induced Clifford permutations that was necessary in the proof of Proposition \ref{FP is Tr} is that they are derangements by disjoint transposition of $\{1, \ldots, 2^{k+1}\}$. Thus we get  that $\Tr(A) = \FP(A)$ whenever $A$ is a derangement by disjoint transpositions of $\{1, \ldots, 2^{k+1}\}$ acting on the matrix forms of 2-torsion points of $S_{\Delta_{2k}}$. 
\end{remark}

\begin{exmp}
In Table \ref{fixed points dimension 4} we list the $16$ fixed points in $J_2^{S_{\Delta_4}}$ for each of the $7$ non-identity classes $[e_\mu]$ acting on the $2$-torsion points in dimension $4$, where $e_\mu \in \gh_4$. An element of $J_2^{S_{\Delta_4}}$ is represented as a $4$-vector $\vec{v}$ each of whose entries is one of the elements $v_0, v_1, v_2,$ or $v_3$ of $J_2^{S_{\Delta_0}}$. We write $v_{abcd}$ to denote the vector $\left( \begin{array}{c}
v_a \\
v_b \\
v_c \\
v_d \\
\end{array}
\right)$ (where $a, b, c, d \in \{0, 1, 2, 3\}$). By Proposition \ref{FP is Tr}, the $16$ fixed points of each equivalence class mod $\sim$ are also the $16$ translation constants for that class. Note that $v_{0000}$ and $v_{3333}$, the constant $4$-vectors at $v_0=0$ and $v_3=\frac{1}{2} + \frac{1}{2}i$ respectively, are fixed points (and translation constants) of each class of map, and they are the only elements of $J_2^{S_{\Delta_4}}$ with this property.  

\begin{table}[h]
\centering

\begin{tabular}{c |c}
Equivalence class & Fixed points / translation constants in $J_2^{S_{\Delta_4}}$ \\
\cline{1-2}
$[e_1]$ & $v_{0000},v_{3030},v_{0003},v_{3033},v_{0030},v_{3300},v_{0033},v_{3303}$\\
& $v_{0300},v_{0303},v_{0330},v_{0333},v_{3000},v_{3003},v_{3330},v_{3333}$\\
\hline
$[e_2]$ & $v_{0000},v_{0012},v_{0021},v_{0321},v_{1200},v_{1212},v_{1221},v_{1233}$\\
& $v_{2100},v_{2112},v_{2121},v_{2133},v_{3300},v_{3312},v_{3321},v_{3333}$\\
\hline
$[e_3]$ & $v_{0000},v_{0011},v_{0022},v_{0033},v_{1100},v_{1111},v_{1122},v_{1133},$\\
& $v_{2200},v_{2211},v_{2222},v_{2233},v_{3300},v_{3311},v_{3322},v_{3333}$\\
\hline
$[e_4]$ & $v_{0000},v_{0011},v_{0110},v_{0210},v_{0220},v_{0330},v_{1111},v_{1221}$\\
& $v_{1331},v_{2002},v_{2112},v_{2222},v_{2332},v_{3003},v_{3113},v_{3333}$\\
\hline
$[e_{14}]$ & $v_{0000},v_{3333},v_{1002},v_{2001},v_{1122},v_{2211},v_{1212},v_{2121}$\\
& $v_{0210},v_{0120},v_{1332},v_{3213},v_{2331},v_{3123},v_{3003},v_{3030}$\\
\hline
$[e_{24}]$ & $v_{0000},v_{1020},v_{3333},v_{2010},v_{0102},v_{2211},v_{0201},v_{2112}$\\
& $v_{1122},v_{1221},v_{1323},v_{3132},v_{2313},v_{3231},v_{3030},v_{0303}$\\
\hline 
$[e_{34}]$ & $v_{0000},v_{0101},v_{3030},v_{3131},v_{0202},v_{3232},v_{0303},v_{3333}$\\
& $v_{1010},v_{1111},v_{1212},v_{1313},v_{2020},v_{2121},v_{2222},v_{2323}$\\

\hline

\end{tabular}
\caption{Fixed points of equivalence classes of elements of $\gh_4$ mod $\sim$ on $J_2^{S_{\Delta_{4}}}$}\label{fixed points dimension 4}
\end{table}
\end{exmp}

\section{Fixed points and translation constants of entry-permuting maps on $n$-torsion points}\label{entry permuting maps}

In this section we show that the result of Proposition \ref{FP is Tr} is related to a special case of a result that holds for a more general class of maps. 

\begin{defn}\label{n-torsion points}
For $n, k \in \mathbb{N}$ with $n \geq 2$ and $k \geq 1$, we define the $n$\textbf{-torsion points} of our spinor Abelian variety $S_{\Delta_{2k}}$ as $J_n^{S_{\Delta_{2k}}} = \{x \in S_{\Delta_{2k}}: n \cdot x = 0\}.$
\end{defn}

We can write any element of $J_n^{S_{\Delta_0}}$, the $n$-torsion points in dimension 1, as $\frac{a}{n} + \frac{b}{n}i$ where $a, b \in \{0, \ldots, n-1\}$. This means that $|J_n^{S_{\Delta_0}}| = n^2$. For any $k \geq 1$, we can write an $n$-torsion point $\vec{v}$ in dimension $2^k$ as a vector of length $2^k$ whose entries are elements of $J_n^{S_{\Delta_0}}$. Similarly to what we did in the case of 2-torsion points, we can also represent the $n$-torsion point $\vec{v}$ as a $2^k \times 2$ matrix with entries in $\{0, \ldots, n-1\}.$ The first and second entries in the $l$th row of this matrix are the numerators of the real and imaginary parts, respectively, of the element of $J_n^{S_{\Delta_0}}$ that is the $l$th component of $\vec{v}.$ 
Let $\vec{v}[j]$ denote the number in the $j$th entry of the matrix form of $\vec{v}$, for $j \in \{1, \ldots, 2^{k+1}\}$. We again number the entries in this matrix row-wise from top to bottom just as we did for 2-torsion points:
\[
\left[
\begin{array}{cc}
\vec{v}[1] & \vec{v}[2] \\
\vec{v}[3] & \vec{v}[4] \\
\vdots & \vdots \\
\vec{v}[2^{k+1}-1] & \vec{v}[2^{k+1}]
\end{array}
\right]
\]
where this time each $\vec{v}[j]$ is in $\{0, \ldots, n-1\}$.

We have that 
\[ |J_n^{S_{\Delta_{2k}}}| = (n^2)^{(2^k)} = n^{2 \cdot 2^k} = n^{(2^{k+1})}.\] 
(We can see this in two ways: to represent an $n$-torsion point $\vec{v}$ in dimension $2^k$ as a $2^k$-vector, we must make $2^k$ choices from among the $n^2$ elements of $J_n^{S_{\Delta_0}}$; or, to represent $\vec{v}$ as a $2^k \times 2$ matrix, we must make $2^{k+1}$ choices from among the numbers $\{0, \ldots, n-1\}$.)

\begin{exmp}
 Let $n=3$ and $k=2$, and consider a $3$-torsion point $\vec{v}$ in dimension $4$. Such an element of $J_3^{S_{\Delta_4}}$ is given by 
\[
\vec{v} = 
\left(
\begin{array}{c}
\\[-9pt]
\frac{a_1}{3} + \frac{b_1}{3}i\\[2pt]
\frac{a_2}{3} + \frac{b_2}{3}i\\[2pt]
\frac{a_3}{3} + \frac{b_3}{3}i\\[2pt]
\frac{a_4}{3} + \frac{b_4}{3}i\\[3pt]
\end{array}
\right)
\]
where $a_l, b_l \in \{0, 1, 2\}$ for $1 \leq l \leq 4$. The numbers $a_l, b_l$ can also be used to encode the $3$-torsion point $\vec{v}$ as a $4 \times 2$ matrix with entries in $\{0, 1, 2\}$:
\[
\left[
\begin{array}{cc}
a_1 & b_1 \\
a_2 & b_2 \\
a_3 & b_3 \\
a_4 & b_4
\end{array}
\right]
\]
For example, the $3$-torsion point 
\[
\vec{v} = 
\left(
\begin{array}{c}
\\[-9pt]
\frac{1}{3} + \frac{0}{3}i\\[2pt]
\frac{2}{3} + \frac{2}{3}i\\[2pt]
\frac{0}{3} + \frac{2}{3}i\\[2pt]
\frac{2}{3} + \frac{1}{3}i\\[3pt]
\end{array}
\right)
\]
is represented in matrix form as
\[
\left[
\begin{array}{cc}
1 & 0 \\
2 & 2 \\
0 & 2 \\
2 & 1
\end{array}
\right]
\]
\end{exmp}

Addition of $n$-torsion points modulo the integer lattice then corresponds to componentwise addition mod $n$ of the entries in the $2^k \times 2$ matrices representing those $n$-torsion points. 

Note that for an arbitrary map $A:J_n^{S_{\Delta_{2k}}} \to J_n^{S_{\Delta_{2k}}}$, the set of fixed points of $A$ is equal to the responsibility set of $\vec{v}_0$, the constant zero vector in $J_n^{S_{\Delta_{2k}}}$; that is, $\FP(A) = \resp_A(\vec{v}_0)$. We next show that distinct responsibility sets are disjoint for any map $A$ from $J_n^{S_{\Delta_{2k}}}$ to itself.

\begin{lem}\label{distinct resp sets disjoint}
Let $A: J_n^{S_{\Delta_{2k}}} \to J_n^{S_{\Delta_{2k}}}$ be any function and $\vec{v}, \vec{w} \in \Tr(A)$. If $\vec{v} \neq \vec{w}$, then $\resp_A(\vec{v}) \cap \resp_A(\vec{w}) = \emptyset$.
\end{lem}

\begin{proof}
Let $\vec{v}, \vec{w} \in \Tr(A)$ and suppose there is some $\vec{z} \in J_n^{S_{\Delta_{2k}}}$ with $\vec{z} \in \resp_A(\vec{v}) \cap \resp_A(\vec{w})$. Then $A \cdot \vec{z} = \vec{v} + \vec{z}$ and $A \cdot \vec{z} = \vec{w} + \vec{z}$, so that $\vec{v} + \vec{z} = \vec{w} + \vec{z}$. Since $J_n^{S_{\Delta_{2k}}}$ is a group under addition mod the integer lattice, $\vec{v} = \vec{w}$.
\end{proof}

\begin{defn}
Representing elements of $J_n^{S_{\Delta_{2k}}}$ as $2^k \times 2$ matrices with entries in $\{0, \ldots, n-1\}$ as above, call a function $A: J_n^{S_{\Delta_{2k}}} \to J_n^{S_{\Delta_{2k}}}$ an \textbf{entry-permuting map} if $A$ is the map induced by a permutation $\sigma \in \mathcal{S}_{2^{k+1}}$. That is, for all $1 \leq i, j \leq 2^{k+1}$, if $\sigma(i)=j$, then for any $\vec{v} \in J_n^{S_{\Delta_{2k}}}$, $(A \cdot \vec{v})[j]=\vec{v}[i]$. 
\end{defn}

\begin{remark}\label{S maps are entry-permuting}
The induced Clifford permutations $S_{1, \ldots, k}$ and $i \cdot S_{1, \ldots, k}$ on the $2$-torsion points in dimension $2^k$ are entry-permuting maps on $J_2^{S_{\Delta_{2k}}}$. However, most entry-permuting maps on $J_2^{S_{\Delta_{2k}}}$ are not induced Clifford permutations; there are only $2^{k+1}$ induced Clifford permutations, but there are as many entry-permuting maps on $J_2^{S_{\Delta_{2k}}}$ as there are permutations on $\{1, \ldots, 2^{k+1}\}$, which is $(2^{k+1})!$.
\end{remark}

Recall that any $\sigma \in \mathcal{S}_{2^{k+1}}$ can be decomposed as a product of disjoint cycles. A \textit{nontrivial} cycle is a cycle of length greater than one.

For what follows, we suppose that $A: J_2^{S_{\Delta_{2k}}} \to J_2^{S_{\Delta_{2k}}}$ is the entry-permuting map corresponding to the permutation $\sigma \in \mathcal{S}_{2^{k+1}}$ given by $\sigma = C_1 \cdots C_p D_1 \cdots D_q$ where for each $1 \leq l \leq p$, $C_l$ is the nontrivial cycle $C_l = (i_{l(1)} \cdots i_{l(t_l)})$; and where for each $1 \leq s \leq q$, $D_s$ is the trivial cycle $(i_s)$. 

\begin{lem}\label{size of FP(A)}
Let $A: J_n^{S_{\Delta_{2k}}} \to J_n^{S_{\Delta_{2k}}}$ be an entry-permuting map with $p$ nontrivial and $q$ trivial cycles as above, where $n, k \in \mathbb{N}$, $n \geq 2$, and $k \geq 1$. Then $|\FP(A)| = n^{p+q}$.
\end{lem}

\begin{proof}
Suppose $\vec{v} \in \FP(A)$, and view $\vec{v}$ as a 
$2^k \times 2$ matrix with entries in $\{0, \ldots, n-1\}$. Then for each $l \in \{1, \ldots, p\}$, $\vec{v}$ must have the same value at all entries corresponding to the cycle $C_l = (i_{l(1)} \cdots i_{l(t_l)})$; that is, $\vec{v}[i_{l(1)}] = \cdots = \vec{v}[i_{l(t_l)}]$. 

Thus to specify a fixed point $\vec{v}$ of $A$, we have to choose its value (a number in $\{0, \ldots, n-1\}$) to be constant on the entries in each of the cycles  $C_1, \ldots, C_p, D_1, \ldots, D_q$; and there are $n^{p+q}$ ways to do this.
\end{proof}

\begin{lem}\label{values in resp sets within cycles}
Let $n, k \in \mathbb{N}$, $n \geq 2$, and $k \geq 1$, and let $A: J_n^{S_{\Delta_{2k}}} \to J_n^{S_{\Delta_{2k}}}$ be an entry-permuting map with $p$ nontrivial and $q$ trivial cycles as above. Let $\vec{v} \in \Tr(A)$ and let $\vec{w} \in \resp_A(\vec{v})$. Then for all $1 \leq l \leq p$ and any $j \in \{1, \ldots, t_l\}$, the value of $\vec{w}[i_{l(j)}]$ is completely determined by $\vec{v}$ and $\vec{w}[i_{l(1)}]$. 
\end{lem}

\begin{proof}
Let $\vec{v} \in \Tr(A)$, $\vec{w} \in \resp_A(\vec{v})$, $1 \leq l \leq p$, and $j \in \{1, \ldots, t_l\}$. By definition of responsibility sets, $A \cdot \vec{w} = \vec{v} + \vec{w}$; so for each $j \in \{1, \ldots, t_l\}$, we have 

\[
(A \cdot \vec{w})[i_{l(j)}] = (\vec{v} + \vec{w})[i_{l(j)}] = \vec{v}[i_{l(j)}] + \vec{w}[i_{l(j)}]
\]

Since also $C_l = (i_{l(1)} \cdots i_{l(t_l)})$ is a cycle in $\sigma$, we have 

\[
\begin{array}{lclcl}
\vec{w}[i_{l(1)}] & = & (A \cdot \vec{w})[i_{l(2)}] & = & \vec{v}[i_{l(2)}] + \vec{w}[i_{l(2)}] \\
\vec{w}[i_{l(2)}] & = & (A \cdot \vec{w})[i_{l(3)}] & = & \vec{v}[i_{l(3)}] + \vec{w}[i_{l(3)}] \\
& \vdots & & \vdots & \\
\vec{w}[i_{l(t_l - 1)}] & = & (A \cdot \vec{w})[i_{l(t_l)}] & = & \vec{v}[i_{l(t_l)}] + \vec{w}[i_{l(t_l)}]\\
\vec{w}[i_{l(t_l)}] & = & (A \cdot \vec{w})[i_{l(1)}] & = & \vec{v}[i_{l(1)}] + \vec{w}[i_{l(1)}]
\end{array}
\]
(The left-hand set of equalities is true for the following reason: when $A$ acts on the vector $\vec{w}$, it produces a vector whose $i_{l(2)}$th entry was the $i_{l(1)}$th entry in $\vec{w}$, etc.)

Then we can express each entry of $\vec{w}$ that is contained in the cycle $C_l$ solely in terms of a single entry of $\vec{w}$ at one of the numbers in $C_l$ (and in terms of its translation constant $\vec{v}$). Here we express each $\vec{w}[i_{l(j)}]$ in terms of $\vec{w}[i_{l(1)}]$ and $\vec{v}$:

\[
\begin{array}{lcl}
\vec{w}[i_{l(2)}] & = & \vec{w}[i_{l(1)}] - \vec{v}[i_{l(2)}] \\
\vec{w}[i_{l(3)}] & = & \vec{w}[i_{l(2)}] - \vec{v}[i_{l(3)}] = (\vec{w}[i_{l(1)}] - \vec{v}[i_{l(2)}]) - \vec{v}[i_{l(3)}] \\
& \vdots & \\
\vec{w}[i_{l(t_l)}] & = & \vec{w}[i_{l(t_l - 1)}] - \vec{v}[i_{l(t_l)}] = \vec{w}[i_{l(1)}] - \displaystyle  \sum_{j=2}^{t_l} \vec{v}[i_{l(j)}]
\end{array}
\]
\end{proof}

(Note that although in the proof of Lemma \ref{values in resp sets within cycles} we expressed each $\vec{w}[i_{l(j)}]$ in terms of $\vec{w}[i_{l(1)}]$ and $\vec{v}$, we could just has well have expressed $\vec{w}[i_{l(j)}]$ in terms of $\vec{w}[i_{l(m)}]$ and $\vec{v}$ for any chosen $m$ with $1 \leq m \leq t_l$.)

\begin{lem}\label{size of resp sets}
Let $n, k \in \mathbb{N}$, $n \geq 2$, and $k \geq 1$; let $A: J_n^{S_{\Delta_{2k}}} \to J_n^{S_{\Delta_{2k}}}$ be an entry-permuting map with $p$ nontrivial and $q$ trivial cycles as above; and let $\vec{v}$ be a translation constant of $A$. Then $|\resp_A(\vec{v})| = n^{p+q}$.
\end{lem}

\begin{proof}
By Lemma \ref{values in resp sets within cycles}, in order to specify an element $\vec{w}$ of the responsibility set $\resp_A(\vec{v})$ of $\vec{v}$, we need only fix a value in $\{0, \ldots, n-1\}$ for $\vec{w}[i_{l(1)}]$ for each $1 \leq l \leq p$, and this will determine the behavior of $\vec{w}$ at all indices in each nontrivial cycle $C_l$ of $\sigma$; and we need also to fix a value in $\{0, \ldots, n-1\}$ for $\vec{w}$ in each of the trivial cycles $D_s$. There are then $n$ ways to pick a value for $\vec{w}[i_{l(1)}]$, for $1 \leq l \leq p$; and there are also $n$ ways to pick a value for $\vec{w}$ on $D_s$ for $1 \leq s \leq q$. Then the total number of ways to choose a $\vec{w} \in \resp_A(\vec{v})$ is 
\[
n^p n^q = n^{p+q}.
\]
\end{proof}

Now we can define $\resp(A)$ to be the common value of $|\resp_A(\vec{v})|$ for any $\vec{v} \in \Tr(A)$, and this quantity is well-defined for any entry-permuting map $A$ on $J_n^{S_{\Delta_{2k}}}$. By Lemmas \ref{size of FP(A)} and \ref{size of resp sets}, we have:

\begin{cor}\label{r(A)=Tr(A)}
Let $n, k \in \mathbb{N}$, $n \geq 2$, and $k \geq 1$. 
For any entry-permuting map $A$ on $J_n^{S_{\Delta_{2k}}}$, \[\resp(A) = n^{p+q} = |\FP(A)|\] where $p$ is the number of nontrivial cycles in the disjoint cycle decomposition of $A$, and $q$ is the number of trivial cycles. \hfill $\square$
\end{cor}

\begin{cor}\label{Tr times FP is ...}
Let $n, k \in \mathbb{N}$, $n \geq 2$, and $k \geq 1$.
For any entry-permuting map $A$ on $J_n^{S_{\Delta_{2k}}}$, $$| \Tr(A)| \cdot |\FP(A)| = | J_n^{S_{\Delta_{2k}}}| = n^{(2^{k+1})}.$$
\end{cor}

\begin{proof}
By Lemma \ref{distinct resp sets disjoint}, $\{r_A(\vec{v}): \vec{v} \in \Tr(A)\}$ is a partition of $J_n^{S_{\Delta_{2k}}}$; and by Lemma \ref{size of resp sets}, each responsibility set $r_A(\vec{v})$ has size $| \FP(A)|$.
\end{proof}

In the case $n=2$, by Propositions \ref{number of fixed points} and \ref{count fixed points for iA} we have that for every generator $e_\mu \in \gh_{2k}$ with induced Clifford permutation $A$, $|\FP(A)| = 2^{(2^k)}$. Then by Corollary \ref{Tr times FP is ...}, 
\[
|\Tr(A)| \cdot 2^{(2^k)} = 2^{(2^{k+1})}
\]
so that
\[
|\Tr(A)| = \frac{2^{(2^{k+1})}}{2^{(2^k)}} = 2^{(2^k)}
\]
That is, we get that $|\Tr(A)| = |\FP(A)|$, which we know already by Proposition \ref{FP is Tr}. Recall that Proposition \ref{FP is Tr} gives the stronger result  that $\Tr(A) = \FP(A)$ when $A$ is an induced Clifford permutation (acting on 2-torsion points).  

For $n > 2$, we do not always get this stronger result, even when we have equality of $|\Tr(A)|$ and $|\FP(A)|$: that is, even when there are as many translation constants as fixed points, these two sets need not be the same. For instance, define the map $A$ on $J_4^{S_{\Delta_{2}}}$, the 4-torsion points in dimension 2, by $A \cdot \vec{v} = A \cdot \left(
\begin{array}{c}
a\\
b
\end{array}
\right) = \left(
\begin{array}{c}
b\\
a
\end{array}
\right)$ where $a, b \in J_4^{S_{\Delta_0}}$. One can show that here $|\Tr(A)| = |\FP(A)| = 16$, but $|\Tr(A) \cap \FP(A)| = 4$.

In general, for $A$ an entry-permuting map acting on the matrix forms of elements of $J_2^{S_{\Delta_{2k}}}$, we do not necessarily have $|\Tr(A)| = |\FP(A)|$. For instance, consider the map $A: J_2^{S_{\Delta_4}} \to J_2^{S_{\Delta_4}}$ defined on the matrix form of points of $J_2^{S_{\Delta_4}}$  by the permutation $\sigma  = (17)(28) \in \mathcal{S}_8$. Note that $A$ is an involution, but not a derangement; it has the four trivial cycles $(3), (4), (5)$, and $(6)$, in addition to the two nontrivial cycles $(17)$ and $(28)$. Since here $p=2$, $q=4$, $k=2$, and $n=2$, $|\FP(A)| = r(A) = 2^{2+4}$ = 64 by Corollary \ref{r(A)=Tr(A)}; and then by Corollary \ref{Tr times FP is ...}, 
\[
|\Tr(A)| = \frac{2^{(2^{3})}}{64} = 4.
\]
By Corollary \ref{r(A)=Tr(A)}, each of the 4 translation constants has a responsibility set of size 
64.

Also in general, for generators $e_\mu \in \gh_{2k}$ acting on $n$-torsion points, the action of $e_\mu$ on the matrix form of an element $\vec{w} \in J_n^{S_{\Delta_{2k}}}$ is not an entry-permuting map. For instance, consider the generator $\ek_1$ acting on the $n$-torsion point $\vec{w} = 
\left(
\begin{array}{c}
\\[-9pt]
\frac{1}{n} + \frac{1}{n}i\\[2pt]
\frac{1}{n} + \frac{1}{n}i\\[2pt]
\vdots\\[2pt]
\frac{1}{n} + \frac{1}{n}i\\[3pt]
\end{array}
\right)$ in some dimension $k \in \mathbb{N}$, $k \geq 1$. In matrix form, $\vec{w} = 
\left[
\begin{array}{cc}
1 & 1\\
1 & 1\\
\vdots & \vdots\\
1 & 1\\
\end{array}
\right]$. Then $\ek_1 \cdot \vec{w} = \left(
\begin{array}{c}
\\[-9pt]
\frac{n-1}{n} + \frac{1}{n}i\\[2pt]
\frac{1}{n} + \frac{n-1}{n}i\\[2pt]
\vdots\\[2pt]
\frac{1}{n} + \frac{n-1}{n}i\\[3pt]
\end{array}
\right)$, which in matrix form is $\left[
\begin{array}{cc}
n-1 & 1\\
1 & n-1\\
\vdots & \vdots\\
n-1 & 1\\
1 & n-1\\
\end{array}
\right]$ -- and so $\ek_1$ acting on the matrix forms of elements of $J_n^{S_{\Delta_{2k}}}$ is clearly not an entry-permuting map on $\{1, \ldots, 2^{k+1}\}$ if $n>2$.

\vspace{15mm}

\textsc{Competing Interests:} The authors have no competing interests to declare.

\vspace{15mm}




\begin{thebibliography}{widestlabel}
	
	\bibitem{BL} Birkenhake C, Lange H (2004) Complex Abelian Varieties. Springer-Verlag, Berlin
	
	\bibitem{FR} Friedrich T (2000) Dirac Operators in Riemannian Geometry. AMS, Rhode Island
	
	\bibitem{GH} Griffith P, Harris J (1994) Principles of Algebraic Geometry. Wiley and Sons
	
	\bibitem{LM} Lawson H, and Michelson M (1989) Spin Geometry. Princeton University Press
	
	\bibitem{EM} Meinrenken E (2013) Clifford Algebras and Lie Theory. Springer-Verlag Berlin Heidelberg
	
	\bibitem{RS} Su\'{a}rez R (2024) Clifford Multiplication on spinor Abelian varieties and algebraic curves. Dissertation, Universit\`{a} degli Studi di Torino
	
\end{thebibliography}
\end{document}